\documentclass[12pt]{amsart}
\usepackage{cases}
\usepackage{mathrsfs}
\usepackage{color}
\usepackage{latexsym,amssymb}
\usepackage{a4wide}
\usepackage{amsthm}
\usepackage{amsmath}
\usepackage{graphicx}
\input xy
\xyoption{all}
\usepackage{verbatim}
\usepackage[lite,abbrev,msc-links]{amsrefs}
\usepackage{framed}

\numberwithin{equation}{section}
\begin{document}
	\theoremstyle{plain}
	\newtheorem{thm}{Theorem}[section]
	\newtheorem{lem}[thm]{Lemma}
	\newtheorem{cor}[thm]{Corollary}
	\newtheorem{cor*}[thm]{Corollary*}
	\newtheorem{prop}[thm]{Proposition}
	\newtheorem{prop*}[thm]{Proposition*}
	\newtheorem{conj}[thm]{Conjecture}
	\theoremstyle{definition}
	\newtheorem{construction}{Construction}
	\newtheorem{notations}[thm]{Notations}
	\newtheorem{question}[thm]{Question}
	\newtheorem{prob}[thm]{Problem}
	\newtheorem{rmk}[thm]{Remark}
	\newtheorem{remarks}[thm]{Remarks}
	\newtheorem{defn}[thm]{Definition}
	\newtheorem{claim}[thm]{Claim}
	\newtheorem{assumption}[thm]{Assumption}
	\newtheorem{assumptions}[thm]{Assumptions}
	\newtheorem{properties}[thm]{Properties}
	\newtheorem{exmp}[thm]{Example}
	\newtheorem{comments}[thm]{Comments}
	\newtheorem{blank}[thm]{}
	\newtheorem{observation}[thm]{Observation}
	\newtheorem{defn-thm}[thm]{Definition-Theorem}
	\newtheorem{defn-lem}[thm]{Definition-Lemma}
	\newtheorem*{Setting}{Setting}

	\newcommand{\sA}{\mathscr{A}}
	\newcommand{\sB}{\mathscr{B}}
	\newcommand{\sC}{\mathscr{C}}
	\newcommand{\sD}{\mathscr{D}}
	\newcommand{\sE}{\mathscr{E}}
	\newcommand{\sF}{\mathscr{F}}
	\newcommand{\sG}{\mathscr{G}}
	\newcommand{\sH}{\mathscr{H}}
	\newcommand{\sI}{\mathscr{I}}
	\newcommand{\sJ}{\mathscr{J}}
	\newcommand{\sK}{\mathscr{K}}
	\newcommand{\sL}{\mathscr{L}}
	\newcommand{\sM}{\mathscr{M}}
	\newcommand{\sN}{\mathscr{N}}
	\newcommand{\sO}{\mathscr{O}}
	\newcommand{\sP}{\mathscr{P}}
	\newcommand{\sQ}{\mathscr{Q}}
	\newcommand{\sR}{\mathscr{R}}
	\newcommand{\sS}{\mathscr{S}}
	\newcommand{\sT}{\mathscr{T}}
	\newcommand{\sU}{\mathscr{U}}
	\newcommand{\sV}{\mathscr{V}}
	\newcommand{\sW}{\mathscr{W}}
	\newcommand{\sX}{\mathscr{X}}
	\newcommand{\sY}{\mathscr{Y}}
	\newcommand{\sZ}{\mathscr{Z}}
	\newcommand{\bZ}{\mathbb{Z}}
	\newcommand{\bN}{\mathbb{N}}
	\newcommand{\bQ}{\mathbb{Q}}
	\newcommand{\bC}{\mathbb{C}}
	\newcommand{\bR}{\mathbb{R}}
	\newcommand{\bH}{\mathbb{H}}
	\newcommand{\bD}{\mathbb{D}}
	\newcommand{\bE}{\mathbb{E}}
	\newcommand{\bV}{\mathbb{V}}
	\newcommand{\bfM}{\mathbf{M}}
	\newcommand{\bfN}{\mathbf{N}}
	\newcommand{\bfX}{\mathbf{X}}
	\newcommand{\bfY}{\mathbf{Y}}
	\newcommand{\cH}{\mathcal{H}}
	\newcommand{\cV}{\mathcal{V}}
	\newcommand{\cF}{\mathcal{F}}
	\newcommand{\spec}{\textrm{Spec}}
	\newcommand{\dbar}{\bar{\partial}}
	\newcommand{\ddbar}{\partial\bar{\partial}}
	\newcommand{\redref}{{\color{red}ref}}

	\title[$L^2$-representation of Hodge Modules] {$L^2$-representation of Hodge Modules}

	\author[Junchao Shentu]{Junchao Shentu}
	\email{stjc@ustc.edu.cn}
	\address{School of Mathematical Sciences,
		University of Science and Technology of China, Hefei, 230026, China}

	\author[Chen Zhao]{Chen Zhao}
	\email{czhao@ustc.edu.cn}
	\address{School of Mathematical Sciences,
		University of Science and Technology of China, Hefei, 230026, China}

	\begin{abstract}
		Over an arbitrary compact complex space or an arbitrary germ of complex space $X$, we provide fine resolutions of pure Hodge modules with strict supports $IC_X(\bV)$ via differential forms with locally $L^2$ boundary conditions. When $\bV=\bC_{X_{\rm reg}}$ is the trivial variation of Hodge structure, we give a solution to a Cheeger-Goresky-MacPherson type conjecture: For any compact complex space $X$, there is a complete hermitian metric $ds^2$ on $X_{\rm reg}$ such that there is a canonical isomorphism
		$$H^i_{(2)}(X_{\rm reg},ds^2)\simeq IH^i(X),\quad \forall i.$$
		Such metric $ds^2$ could be K\"ahler if $X$ is a K\"ahler space.
		As an application, we give a differential geometrical proof of the K\"ahler package of the hypercohomology of pure Hodge modules. We also prove the K\"ahler version of Kashiwara's conjecture in the absolute case.
	\end{abstract}

	\maketitle
	\tableofcontents

	\section{Introduction}
	This paper grows out of an attempt to generalize the Topology-Geometry-Analysis nature of the Hodge theory to compact K\"ahler spaces with coefficients. The classical Hodge theory for a compact K\"ahler manifold $X$ can be summarized by the graph
	$$
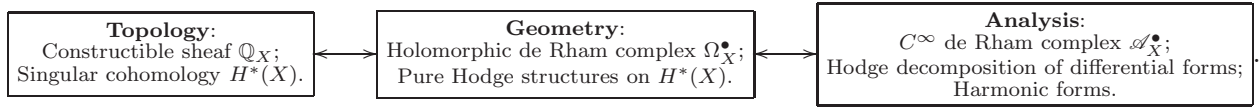
\xymatrix{
		*+[F]{\substack{{\bf Topology:}\\\textrm{Constructible sheaf }\bQ_X;\\\textrm{Singular cohomology } H^\ast(X).}} \ar@{<->}[r] & *+[F]{\substack{{\bf Geometry:}\\\textrm{Holomorphic de Rham complex }\Omega_X^\bullet;\\
				\textrm{Pure Hodge structures on } H^\ast(X).}}\ar@{<->}[r]&
		*+[F]{\substack{{\bf Analysis:}\\\textrm{$C^\infty$ de Rham complex }\sA_X^\bullet;\\
				\textrm{Hodge decomposition of differential forms;}\\\textrm{Harmonic forms.}}}
	}.$$
	This picture is generalized by P. Deligne \cite{Delign1974} to algebraic varieties which may admits singularities:
	$$
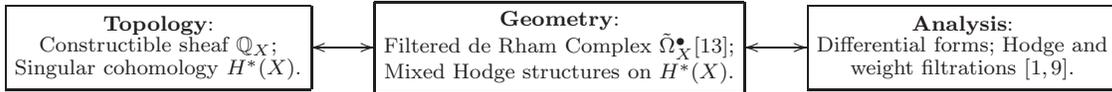
\xymatrix{
		*+[F]{\substack{{\bf Topology:}\\\textrm{Constructible sheaf }\bQ_X;\\\textrm{Singular cohomology } H^\ast(X).}} \ar@{<->}[r] & *+[F]{\substack{{\bf Geometry:}\\\textrm{Filtered de Rham Complex }\tilde{\Omega}_X^\bullet\textrm{\cite{DuBois1981}};\\
				\textrm{Mixed Hodge structures on } H^\ast(X).}}\ar@{<->}[r]&
		*+[F]{\substack{{\bf Analysis:}\\
				\textrm{Differential forms; Hodge and}\\ \textrm{weight filtrations \cite{Deligne1971_2,Bernard-Vincenzo2006}.}}}
	}.$$
	In the analytic counterpart, the differential forms  live either on a hyper-covering of $X$ \cite{Deligne1971_2,Delign1974} or a system of desingularization diagrams of $X$ \cite{Bernard-Vincenzo2006}, instead of on $X$ itself. Moreover, the Hodge decomposition theorem and the theory of harmonic integrals are missing on the analytic side.

	Motivated by the theory of weight in characteristic $p>0$ \cite{BBD}, the theory of mixed Hodge structures and variation of these objects \cite{Cattani_Kaplan_Schmid1986,Schmid1973,Steenbrink-Zucker1985} are integrated by M. Saito into his theory of mixed Hodge modules \cite{MSaito1988,MSaito1990}. Saito's theory has become an indispensable tool in the study of complex geometry, arithmetic geometry and representation theories. Conceptually, they establish the Hodge theory to the most generality so far:
	$$
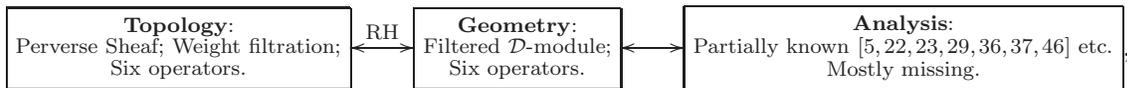
\xymatrix{
		*+[F]{\substack{{\bf Topology:}\\\textrm{Perverse Sheaf; Weight filtration;}\\\textrm{Six operators.}}} \ar@{<->}[r]^-{\rm RH} & *+[F]{\substack{{\bf Geometry:}\\\textrm{Filtered $\mathcal{D}$-module;}\\\textrm{Six operators.}}}\ar@{<->}[r]&
		*+[F]{\substack{{\bf Analysis:}\\
				\textrm{Partially known \cite{Zucker1979,Cattani_Kaplan_Schmid1987,Kashiwara_Kawai1987,Saper_Stern1990,Looijenga1988,Saper1992,Ohsawa1991} etc.}\\ \textrm{Mostly missing.}}}
	},$$
	where {\rm RH} stands for the Riemann-Hilbert correspondence.
	Although the theory of mixed Hodge modules has been well established, its analytic counterpart is much less known. The first problem that we encounter is to find a nice analytic perspective of pure Hodge modules with strict support\footnote{Every pure Hodge module is a (locally finite) direct sum of pure Hodge modules with strict supports (\cite[\S 5.16]{MSaito1988}).}:
	\begin{prob}\label{prob_vague}
		Let $X$ be a complex space and $M$ a pure Hodge module with strict support $X$. Assume that the underlying perverse sheaf of $M$ is $IC_X(\bV)$ where $\bV$ is a polarized variation of Hodge structure over some dense Zariski open subset $X^o\subset X_{\rm reg}$. Find a fine resolution $\sD^\bullet$ of $IC_X(\bV)$ such that
		\begin{enumerate}
			\item $\sD^\bullet$ is consist of differential forms in some reasonable sense. E.g. smooth forms with combinatory or analytic boundary conditions.
			\item When $X$ is compact, there is a Hodge decomposition theorem for elements in $\sD^\bullet$. Moreover every cohomology class of $\bH^\ast(X,IC_X(\bV))$ is uniquely represented by a harmonic form.
			\item When $X$ is compact and admits a K\"ahler hermitian metric, there is a natural $(p,q)$-decomposition of $\sD^\bullet$ which gives a pure Hodge structure on $\bH^\ast(X,IC_X(\bV))$.
		\end{enumerate}
	\end{prob}
	It is a general philosophy that Hodge modules are closely related to $L^2$-differential forms. For one thing, the classical Hodge theory is established via the $L^2$-technique. On the other hand, in Saito's theory of Hodge modules, the study of the cohomology of Hodge modules (e.g. the proof of Saito's decomposition theorem) is reduced to S. Zucker's work \cite{Zucker1979} on the $L^2$-representation of Hodge modules over a smooth algebraic curve. In the era when intersection cohomology was defined, J. Cheeger, M. Goresky and R. MacPherson made a conjecture that the intersection cohomology is represented by a certain $L^2$-cohomology (see Conjecture \ref{conj_CGM} below), so that the intersection cohomology of a projective variety owns a package of theorems including the Hard Lefschetz theorem and the existence of a pure Hodge structure. These properties are proved by Be\u{\i}linson-Bernstein-Deligne-Gabber \cite{BBD} and Saito \cite{MSaito1988, MSaito1990} by alternative methods.

	In recent years, the technique of mixed Hodge modules and $L^2$-methods both play important roles in the development of complex and algebraic geometry. It is interesting to search for their relations and combine them together.

	According to this philosophy, a promising candidate of such a  $\sD^\bullet$ is the complex of locally $L^2$-$\bV$-valued differential forms. Let $ds^2$ be a metric on $X^o$ and denote by $\sD^\bullet_{X,\bV;ds^2,h}$ the complex of sheaves on $X$ consisting of $\bV$-valued measurable differential forms $\alpha$ such that $\alpha$ and $\nabla\alpha$ are locally square integrable over $X$ (including the singular points) with respect to $ds^2$ and the Hodge metric $h$ on $\bV$. Compared to other boundary conditions (combinatory or $L^p$, $p\neq 2$), this kind of complex has the following advantages.
	\begin{enumerate}
		\item For an amount of interesting metrics $ds^2$, $\sD^\bullet_{X,\bV;ds^2,h}$ is a complex of fine sheaves (Lemma \ref{lem_fine_sheaf}).
		\item When $X$ is compact, $\Gamma(X,\sD^\bullet_{X,\bV;ds^2,h})$ automatically admits a Hodge decomposition theorem whenever $H^\ast(\Gamma(X,\sD^\bullet_{X,\bV;ds^2,h}))$ are finite dimensional (\cite[Theorem A.2.2]{Kashiwara_Kawai1987}).
		\item When $X$ is compact and $ds^2$ is a complete K\"ahler metric, $H^\ast(\Gamma(X,\sD^\bullet_{X,\bV;ds^2,h}))$ admits a natural pure Hodge structure of the right weight as long as it is finite dimensional (\cite[Theorem 6.2.3]{Kashiwara_Kawai1987} or \cite[\S 7]{Zucker1979}). As the classical case, the Hodge structure is induced from the total bi-degree decomposition of differential forms.
	\end{enumerate}
	Now Problem \ref{prob_vague} turns to a precise one, which is the main subject of this paper.
	\begin{prob}\label{prob_main}
		Let $X$ be a complex space and let $(\bV,h)$ underly an $\bR$-polarized variation of Hodge structure over some dense Zariski open subset $X^o\subset X_{\rm reg}$ with quasi-unipotent local monodromies. Is there a hermitian (resp. K\"ahler) metric $ds^2$ on $X^o$ such that
		\begin{enumerate}
			\item $\sD^\bullet_{X,\bV;ds^2,h}$ is a complex of fine sheaves and
			\item $\sD^\bullet_{X,\bV;ds^2,h}$ is quasi-isomorphic to $IC_X(\bV)$ in $D(X)$, the derived category of sheaves of $\bC$-vector spaces?
		\end{enumerate}
	\end{prob}
	Problem \ref{prob_main} has been solved in some cases.
	\begin{itemize}
		\item When $ds^2$ is only required to be a Riemannian metric and $\bV=\bC_{X_{\rm reg}}$ is the trivial variation of Hodge structure, Problem \ref{prob_main} is solved by Cheeger \cite{Cheeger1980}. Cheeger's work is the starting point of the analytical study of perverse sheaves.
		\item When $X$ is a smooth algebraic curve and $ds^2$ is a hermitian metric asymptotic of Poincar\'e type along the isolated points $X\backslash X^o$, Problem \ref{prob_main} is solved by Zucker \cite{Zucker1979}.
		Zucker's result is generalized by Jost-Yang-Zuo\cite{Jost_Zuo2007} to the case when the coefficients are in a unipotent harmonic bundle.
		\item More generally when $X$ is a complex manifold, $D:=X\backslash X^o$ is a normal crossing divisor and $ds^2$ is a K\"ahler metric asymptotic of Poincar\'e type along $D$, Problem \ref{prob_main} is solved by Cattani-Kaplan-Schmid \cite{Cattani_Kaplan_Schmid1987} and Kashiwara-Kawai \cite{Kashiwara_Kawai1987} respectively.
		\item When $X$ is K\"ahler and admits only isolated singularities, $X^o=X_{\rm reg}$ and $\bV=\bC_{X_{\rm reg}}$ is the trivial variation of Hodge structure, L. Saper \cite{Saper1985,Saper1992} constructs a family of  complete K\"ahler metrics $ds^2$ which solves Problem \ref{prob_main}. Based on Saper's work, T. Ohsawa \cite{Ohsawa1991,Ohsawa1993} settles this case when $ds^2$ is a hermitian metric on $X$.
		\item When $X^o$ is a locally symmetric variety $G/\Gamma$, $X$ is its Satake-Baily-Borel compactification, $ds^2$ is the canonical complete metric and $(\bV,h)$ is a hermitian local system associated to a representation of $G$, Problem \ref{prob_main} is a conjecture made by Zucker \cite{Zucker1982} and is proved by E. Looijenga\cite{Looijenga1988} and Saper-Stern\cite{Saper_Stern1990} respectively.
	\end{itemize}
	In the case when $X^o=X_{\rm reg}$ and $\bV=\bC_{X_{\rm reg}}$ is the trivial variation of Hodge structure, Problem \ref{prob_main} is closely related to the famous
	\begin{conj}[Cheeger-Goresky-Macpherson, \cite{CGM1982}]\label{conj_CGM}
		Let $X\subset \mathbb{P}^N$ be a projective variety and let $ds^2_{\rm FS}$ be the Fubini-Study metric on $X_{\rm reg}$. Then there is an isomorphism
		$$H^\ast_{(2)}(X_{\rm reg},ds^2_{\rm FS})\simeq IH^\ast(X,\bC).$$
	\end{conj}
	Conjecture \ref{conj_CGM} is solved affirmatively by W. C. Hsiang and V. Pati \cite{Hsiang_Pati1985} for the case of normal surfaces, and by Ohsawa \cite{Ohsawa1991} for the case when $X$ admits only isolated singularities. The general case is still unknown.

	\subsection{Main results}
	The aim of this paper is to give a solution to Problem \ref{prob_main} in the most general situation. Our motivation is to relate the theory of Hodge modules and $L^2$-techniques. For the purpose of applications we focus on complete (K\"ahler) metrics.

	Let $X$ be a complex space of pure dimension and $X^o\subset X_{\rm reg}$ a dense Zariski open subset. We define a family of hermitian metrics on $X^o$ which is called distinguished associated to $(X,X^o)$ (Definition \ref{defn_Grauert_type_metric}). Such kind of metric is complete and of finite volume locally at every point of $X$ (Corollary \ref{cor_Gt_complete_finite_vol}). The distinguished metric exists in either of the following cases (Lemma \ref{lem_Grauert_metric_exist}).
	\begin{enumerate}
		\item $(X,X\backslash X^o)$ is compactable (e.g. when $X$ is compact or $(X,X^o)$ is a pair of algebraic varieties).
		\item $X$ is a germ of complex space.
	\end{enumerate}
	When $X$ is a Zariski open subset of a compact K\"ahler space or is a germ of a complex space, $(X,X^o)$ admits a K\"ahler distinguished metric.

	The main result of this paper is
	\begin{thm}[= Theorem \ref{thm_main_harmonic_proof}]\label{thm_main_harmonic}
		Let $X$ be a complex space of pure dimension and $X^o\subset X_{\rm reg}$ a dense Zariski open subset. Let $ds^2$ be a distinguished metric associated to $(X,X^o)$.
		Let $(\bV,h)$ be a $0$-tame harmonic bundle on $(X,X^o)$. Then
		\begin{enumerate}
			\item $\sD^\bullet_{X,\bV;ds^2,h}$ is a complex of fine sheaves.
			\item There is a canonical quasi-isomorphism
			$$\sD^\bullet_{X,\bV;ds^2,h}\simeq IC_X(\bV)$$ in $D(X)$.
		\end{enumerate}
		Moreover if $X$ is compact, there is a natural isomorphism
		$$H^k_{(2)}(X^o,\bV;ds^2,h)\simeq IH^k(X,\bV):=\bH^k(X,IC_X(\bV))$$
		for every $k$.
		In particular, $H^k_{(2)}(X^o,\bV;ds^2,h)$ is finite dimensional for every $k$ when $X$ is compact.
	\end{thm}
	A $0$-tame harmonic bundle on $(X,X^o)$ (Definition \ref{defn_0tame_harmonic_bundle}) is a pluriharmonic bundle on $X^o$ whose metric satisfies a certain constrain of the asymptotic behavior near $X\backslash X^o$. The class of $0$-tame harmonic bundles contains (Remark \ref{rmk_harmonic_bundle_regular_singularity_examples}) the underlying hermitian local system of an $\bR$-polarized variation of Hodge structure with quasi-unipotent local monodromies and semisimple local systems (do not necessarily have quasi-unipotent local monodromies) endowed with the (canonical) Corlette-Jost-Zuo metric (\cite{Jost_Zuo1997},\cite{Mochizuki20072}). Hence Theorem \ref{thm_main_harmonic} gives an $L^2$-representation of (the $\sD$-module part of) twisted $\sD$-modules, which are introduced by C. Sabbah \cite{Sabbah2005} as generalizations of Hodge modules. In particular, Theorem \ref{thm_main_harmonic} gives a solution to Problem \ref{prob_main}.

	By taking $\bV=\bC_{X_{\rm reg}}$ we solve Conjecture \ref{conj_CGM} for the distinguished metric.
	\begin{cor}
		Let $X$ be a compact complex space, then there exists a complete hermitian metric $ds^2$ on $X_{\rm reg}$ such that there is a natural isomorphism
		$$H^\ast_{(2)}(X_{\rm reg},ds^2)\simeq IH^\ast(X,\bC).$$
		When $X$ admits a K\"ahler hermitian metric, $ds^2$ could be K\"ahler.
	\end{cor}
	\begin{rmk}
		\begin{enumerate}
			\item Compared to the Fubini-Study metric in Conjecture \ref{conj_CGM}, the distinguished metric is complete and it provides good Hodge theoretic properties (c.f. \cite[\S 7]{Zucker1979}). However, the distinguished metric is not canonical.
			\item When $X$ admits only isolated singularities, the distinguished metric is different from Saper's distinguished metric \cite{Saper1992} since the distinguished metric in this paper admits bounded potential functions locally. This property ensures a MacPherson's conjecture type result for the $L^2$-representation of Hodge modules (Theorem \ref{thm_resolve_S_sheaf}).
			\item When $X$ is a compact manifold and $D:=X\backslash X^o$ is a normal crossing divisor, the distinguished metric is not asymptotic of Poincar\'e type along $D$. The technique of this paper does not work on Poincar\'e type metric directly. We will deal with this natural metric in another paper.
		\end{enumerate}
	\end{rmk}
	The $L^2$-representation of the intersection complex allows us, as it is hoped, to prove the K\"ahler package of the intersection cohomology with coefficients in a variation of Hodge structure or the hypercohomology of a semisimple local system more generally. This package of theorems includes
	\begin{enumerate}
		\item The Poincar\'e duality (Theorem \ref{thm_main_L2_local_duality}, \ref{thm_main_L2_duality}) and its interpretation by integration.
		\item Hard Lefschetz theorem, Lefschetz decomposition (Theorem \ref{thm_Lefschetz_package}). This is a generalization of C. Simpson's Hard Lefschetz theorem for semisimple local systems to semisimple perverse sheaves, which is the K\"ahler version of a special case of Kashiwara's conjecture \cite{Kashiwara1998}.

		Precisely we show that (Corollary \ref{cor_Kashiwara_conj}) for a semisimple perverse sheaf $P$ on a compact K\"ahler space $X$, the Hard Lefschetz theorem holds for the hypercohomology $\bH^\ast(X,P)$. The proofs use the $L^2$-harmonic representatives of cohomology classes followed by standard arguments. It is important to have $ds^2$ being complete since one needs the K\"ahler identity in the sense of unbounded operators (\cite[(7.5)]{Zucker1979}).
		\item The existence of the pure Hodge structure on $\bH^\ast(X,M)$ via the $(p,q)$-decomposition of $L^2$-forms (\cite[\S 6]{Kashiwara_Kawai1987}, \cite[\S 7]{Zucker1979}) where $M$ is a pure Hodge module. We also make a conjecture that this $L^2$-Hodge structure is compatible with Saito's (Conjecture \ref{conj_Hodge_structure}). Evidences for this conjecture will be provided in \S \ref{section_Hodge_structures}.
	\end{enumerate}
	\subsection{The construction of the distinguished metric}
	One of the main difficulties for the $L^2$-representation over a general complex analytic singularity, compared to isolated singularities, is the problem of equisingularity. The first step towards the $L^2$-representation of the intersection complex is to find a stratification along which $\sD^\bullet_{X,\bV;ds^2,h}$ is constructible. In order to achieve this, we require the metrics $ds^2$ and $h$ to be quasi-isotrivial along each stratum. In \cite{Saper_Stern1990,Looijenga1988,Kashiwara_Kawai1987,Cattani_Kaplan_Schmid1987}, the singularity loci is naturally stratified so that $ds^2$ and $h$ are quasi-isotrivial along each stratum in a rather nice way. More precisely, for every stratum $S$ and every point $x\in S$, there is a neighborhood $x\in U_x$ in $X$, a contractible precompact neighborhood $x\in U_{S,x}$ in $S$, an analytic subspace $C_x$ which intersects transversally with $S$ at the single point $x$ and a $C^\infty$ diffeomorphic fibration
	\begin{align}\label{align_top_str}
	\alpha_x:U_{S,x}\times C_x\simeq U_x
	\end{align}
	such that
	\begin{align}\label{align_metric_str}
	\alpha_x^\ast ds^2\sim ds^2_{S}+ds^2|_{C_x},\quad \alpha_x^\ast h\sim q^\ast(h|_{C_x})
	\end{align}
	for some hermitian metric $ds^2_S$ on $S$. Here $q:U_{S,x}\times C_x\to C_x$ is the projection. This ensures that the cohomology sheaves of $\sD^\bullet_{X,\bV;ds^2,h}$ are locally trivial along the stratum $S$.

	However, due to the possible appearance of Whitney's four plane $x(x+y)(x-y)(x-zy)=0$ (\cite[p.g. 470]{Goresky2012}), it is too much to hope $\alpha_x$ to be $C^\infty$ for a general analytic singularity. For a hermitian metric $ds^2_0$ on $X$, i.e. locally $ds^2_0$ is quasi-isometric to the euclidean metric via a local embedding $X\subset\bC^N$, T. Mostowski \cite{Mostowski1985} introduces the Lipschitz stratification so that $ds^2_0$ is locally quasi-isotrivial along each stratum. In this case, (\ref{align_top_str}) is a bi-Lipschitz homeomorphism.

	For other interesting metrics, a bi-Lipschitz fiberation is not enough to ensure the quasi-isotriviality (\ref{align_metric_str}).  One way to solve this difficulty is to use the resolution of singularity. As the case of Whitney's four plane, one may resolve the singularity into a normal crossing divisor to get a $C^\infty$ fiberation. Although this fiberation may not desend to a stratified fiberation on $X$, this causes no serious problem for this paper since the $L^2$-cohomology is unchanged under bimeromorphic transformations.

	Therefore the problem turns to find a stratification which admits an equisingular desingularization. In this paper, for every pair $(X,X^o)$ we introduce a Whitney stratification $0\subset X_0\subset\dots\subset X_n=X$ and a desingularization $\pi:\widetilde{X}\to X$ so that the preimage of every $X_i$ is a simple normal crossing divisor. Moreover, along each stratum the fibers $\pi^{-1}(x)$ are smoothly stratified isotrivial. Such a pair $(\{X_i\},\pi:\widetilde{X}\to X)$ is called a Whitney stratified desingularization (Definition \ref{defn_Lstr_desingularization}). It exists when $X$ is compact or is a germ of complex space (Theorem \ref{thm_Lstr_desingularization}). For this construction, along each stratum the $C^\infty$ fiberation (\ref{align_top_str}) exists on the smooth model $\widetilde{X}$ (Proposition \ref{prop_bimero_fibering_neigh}), rather than on $X$. A general phenomenon is that \textit{every motivic hermitian metric (a metric defined by holomorphic functions) is quasi-isotrivial along the $C^\infty$ fibration associated to some Whitney stratified desingularization, as long as its quasi-isometric local expression is independent of the choice of $C^\infty$ complex coordinates.}
	This is useful to produce many interesting constructible $L^2$-de Rham complexes (Proposition \ref{prop_Dstr}).

	The distinguished metric could be described as follows (Definition \ref{defn_Grauert_type_metric}). Let $X$ be a complex space and $X^o\subset X_{\rm reg}$ a dense Zariski open subset. A hermitian metric on $X^o$ is called distinguished associated to $(X,X^o)$ if for every point $x\in X$ there is a neighborhood $U$ and a Whitney stratified desingularization $\pi:\widetilde{U}\to U$ of $(U,X^o\cap U)$ such that $ds^2$ is quasi-isometric to
	\begin{align*}
	\sum_{i=1}^r\frac{dz_i d\overline{z}_i}{|z_i|^2\log^2|z_1\cdots z_r|^2\log^2(-\log|z_1\cdots z_r|^2)}
	+\frac{\sum_{i=1}^ndz_id\overline{z}_i}{-\log|z_1\cdots z_r|^2\log^2(-\log|z_1\cdots z_r|^2)}+\pi^\ast\omega_0.
	\end{align*}
	Here $\omega_0$ is a hermitian metric on $X$ and $(z_1,\dots,z_n)$ is a holomorphic coordinate on $\widetilde{U}$ such that $\pi^{-1}(U\backslash X^o)=\{z_1\cdots z_r=0\}$. The distinguished metric witnesses the stratification $\{X_i\}$, which is necessary to construct an $L^2$-representation of the intersection complex.
	\subsection{Strategy of the proof of the main theorem}
	One of the consequences of Theorem \ref{thm_main_harmonic} is the Verdier duality
	\begin{align*}
	\bD(\sD^\bullet_{X,\bV;ds^2,h}[\dim X])\simeq \sD^\bullet_{X,\bV^\ast;ds^2,h^\ast}[\dim X]\in D(X).
	\end{align*}
	Here $\bD$ is the Verdier dual operator on $D(X)$.
	Locally at every point $x\in X$ this implies the local duality
	\begin{thm}[Locally Duality, Theorem \ref{thm_main_L2_local_duality}]\label{intro_thm_main_L2_local_duality}
		Notations as in Theorem \ref{thm_main_harmonic}. Let $x\in X$ be a point. Then for every $0\leq k\leq 2\dim X$ there is a duality
		\begin{align*}
		\varinjlim_{x\in U}\bH^{2\dim X-k}(U,\sD^\bullet_{X,\bV;ds^2,h})\simeq \varinjlim_{x\in U}\left(\bH^k_{\rm cpt}(U,\sD^\bullet_{X,\bV^\ast;ds^2,h^\ast})\right)^\ast,
		\end{align*}
		where $U$ ranges over all open neighborhoods of $x$. Moreover both sides of the isomorphism are finite dimensional.
	\end{thm}
    Denote the statements
    \begin{itemize}
    	\item ${\bf IC}_n$: Theorem \ref{thm_main_harmonic} holds when $\dim X\leq n$.
    	\item ${\bf VD}_n$: Theorem \ref{intro_thm_main_L2_local_duality} holds for $\dim X\leq n$ and $k\leq\dim X$, under the additional assumptions that $$\varinjlim_{x\in U}\left(\bH^k_{\rm cpt}(U,\sD^\bullet_{X,\bV^\ast;ds^2,h^\ast})\right)^\ast$$ is finite dimensional for every $k\leq \dim X$ and
    	$$\varinjlim_{x\in U}\left(\bH^{\dim X}_{\rm cpt}(U,\sD^\bullet_{X,\bV;ds^2,h})\right)^\ast$$
    	is finite dimensional.
    \end{itemize}
    The proof of Theorem \ref{thm_main_harmonic} follows from a bi-induction procedure:
    \begin{description}
    	\item[Initial Step] ${\bf IC}_0$. This case is trivial.
    	\item[Bi-induction Step] ${\bf IC}_{n-1}\Rightarrow {\bf VD}_{n}$, ${\bf VD}_{n}\Rightarrow {\bf IC}_{n}$.
    \end{description}
    Notice that in the induction procedure, non-irreducible complex space appears even $X$ is irreducible. Therefore in Theorem \ref{thm_main_harmonic} we allow the space to have many components.

    Now let us explain the ideas in the proof of the bi-induction steps.

    ${\bf IC}_{n-1}\Rightarrow {\bf VD}_{n}$ follows from an abstract $L^2$-Poincar\'e duality (Proposition \ref{prop_L2_Poincare_duality}). The main difficulty is to show that $H^{\dim X}_{(2),\rm max}(U,\bV;ds^2,h)$ is finite dimensional for a certain neighborhood $U$ of $x$. Our strategy is to use an $L^2$-Mayer-Vietoris sequence to cut this cohomology into two parts: one is $\bH^{\dim X}_{\rm cpt}(U,\sD^\bullet_{X,\bV;ds^2_{\pi},h})$ which is finite dimensional by the assumption. The other is the $L^2$-cohomology over the boundary $\partial U$. Since $\partial U$ is a compact stratified space which intersects with $X$ on the strata of positive dimensions, ${\bf IC}_{n-1}$ ensures that the $L^2$-cohomology over $\partial U$ is finite dimensional.

	${\bf VD}_{n}\Rightarrow {\bf IC}_{n}$ is divided into six steps.
	\begin{enumerate}
		\item Reduce the problem ${\bf IC}_{n}$ to the case of a germ of complex space $0\in X\subset \bC^N$ (Corollary \ref{cor_IC_is_local}). This is a consequence of the Goresky-MacPherson-Deligne criterion (Proposition \ref{prop_GM_criterion}).
		\item Show that the distinguished metric and the harmonic metric are bimeromorphically stratified (Definition \ref{defn_metric_str}) along the fibering stratification of $X\subset \bC^N$ with respect to some Whitney stratified desingularization of $X$. This implies that the cohomology sheaves of $\sD^\bullet_{X,\bV;ds^2,h}$ are weakly constructible.
		\item According to Step (2), the $L^2$-Poincar\'e lemma (Lemma \ref{lem_Kunneth_locL2}) and the $L^2$-Thom isomorphism (Lemma \ref{lem_Thom_locL2}) reduce the problem (via Proposition \ref{prop_Dstr}) to two types of vanishings: Notations as in Theorem \ref{thm_main_harmonic}. Assume that $0\in Z:=X$ is a germ of complex space. Then there are vanishings
		\begin{align}\label{align_vanishing}
		\varinjlim_{0\in U}\bH^k(U,\sD^\bullet_{Z,\bV;ds^2,h})=0,\quad \forall k\geq \dim Z
		\end{align}
		and
		\begin{align}\label{align_vanishing_cpt}
		\varprojlim_{0\in U}\bH^k_{\rm cpt}(U,\sD^\bullet_{Z,\bV;ds^2,h})=0,\quad \forall k\leq \dim Z.
		\end{align}
		Here $U$ ranges over the open neighborhoods of $0$.
		\item Assuming ${\bf VD}_n$, it suffices to prove the vanishing (\ref{align_vanishing_cpt}) for both $\bV$ and $\bV^\ast$. Since  $(\bV^\ast,h^\ast)$ is also a $0$-tame harmonic bundle, it is enough to show (\ref{align_vanishing_cpt}) for an arbitrary $(\bV,h)$. This is the only place where we assume ${\bf VD}_n$.
		\item (\ref{align_vanishing_cpt}) is proved via calculating the $L^2$-cohomologies of $U_{\delta}:=\{z\in Z|\|z\|<\delta\}$, $0<\delta\ll1$. The key feature of $U_\delta$ is that its $L^2$-cohomologies are independent of $\delta>0$. I.e. the 'extension by zero' morphism
		\begin{align}\label{align_cptL2_constant}
		\bH^k_{\rm cpt}(U_{\delta'},\sD^\bullet_{Z,\bV;ds^2,h})\to \bH^k_{\rm cpt}(U_{\delta},\sD^\bullet_{Z,\bV;ds^2,h})
		\end{align}
		is an isomorphism for every $0<\delta'<\delta\ll1$ (Proposition \ref{prop_cpt_vs_min}). Hence (\ref{align_vanishing_cpt}) is equivalent to the vanishing
		\begin{align}\label{align_main_covanishing}
		\bH^k_{\rm cpt}(U_{\delta},\sD^\bullet_{Z,\bV;ds^2,h})=0,\quad \forall k\leq \dim Z.
		\end{align}
		We prove this vanishing by solving the equation $\nabla\beta=-\alpha$ for $\nabla\alpha=0$ in the space of $L^2$ differential forms whose support is compact in $U_\delta$.
		In order to do this, we introduce another cofinal system of neighborhoods $V_\delta$, $0<\delta\ll1$. The main feature of $V_\delta$ is the existence of a fiberation
		\begin{align}\label{align_intro_fibering_U_delta}
		V_{\delta}\cap Z^o\simeq(0,\delta)\times L
		\end{align}
		for some real submanifold $L\subset Z^o$, such that the distinguished metric satisfies the estimate (Proposition \ref{prop_Grauert_asymptotic_behavior})
		\begin{align}\label{align_asymptotic_Gt}
		\frac{d\rho^2}{\rho^2\log^2\rho\log^2(-\log\rho)}+\frac{ds^2_{L}}{\log^2 \rho\log^2(-\log\rho)}
		\lesssim ds^2\lesssim\frac{d\rho^2}{\rho^2\log^2\rho\log^2(-\log \rho)}+ds^2_{L}.
		\end{align}
		Here $ds^2_L$ is a Riemannian metric on $L$ which can be extended smoothly to the boundary.

		To get rid of the complicated boundary condition (\cite[Proposition 2.1.1]{Hormander1965}) caused by the incompleteness on $\partial U_{\delta}$, we make a modification on $ds^2$ along $\partial U_\delta$, denoted by $ds^2_\delta$, so that $ds^2_\delta$ is complete and $\|-\|_{ds^2,h}\lesssim \|-\|_{ds^2_\delta,w_\delta h}$ for some weight function $w_\delta$. If ${\rm supp}\alpha$ is compact in $U_\delta$, the modification on $\partial U_\delta$ does not affect the analytical behavior of $\alpha$. This is one of the reasons why we choose to prove (\ref{align_vanishing_cpt}) instead of (\ref{align_vanishing}).

		Due to the completeness of $ds^2_\delta$, the equation $\nabla\beta=-\alpha$ is solvable in the Hilbert space of $L^2$ differential forms with respect to $ds^2_\delta$ and $w_\delta h$ if one shows the estimate (\cite[\S 2.1.2 Lemma 2.1]{Ohsawa2018})
		\begin{align}\label{align_intro_aim_est}
		|(\alpha,u)|_{ds^2_\delta,w_\delta h}\leq C\|\nabla^\ast_\delta u\|_{ds^2_\delta,w_\delta h},\quad \forall u\in A^{k}_{\rm cpt}(U_{\delta}\cap Z^o,\bV), \forall k\leq \dim Z.
		\end{align}
		Here $\nabla^\ast_\delta$ is the Hilbert adjoint of $\nabla$ with respect to $(ds^2_{\delta},w_\delta h)$. Thanks to (\ref{align_cptL2_constant}), we may assume that ${\rm supp}\alpha$ lies in some $V_{\delta''}\subset U_{\delta}$.
		Let $\alpha=\alpha_1+d\rho\wedge\alpha_0$ be the decomposition according to (\ref{align_intro_fibering_U_delta}).
		There is a simple observation that the integration of $\alpha_0$ along $(0,\delta'')$ in (\ref{align_intro_fibering_U_delta}) gives a formal solution to the equation $\nabla\beta=-\alpha$. However this solution fails to be square integrable with respect to $(ds^2_\delta,w_\delta h)$.
		Our strategy is to deform $(ds^2_{\delta},w_\delta h)$ to  $(ds^2_{\delta,\epsilon},w_{\delta,\epsilon} h)$, $\epsilon\in(0,1)$ so that $$\|\alpha\|_{ds^2_{\delta,\epsilon},w_{\delta,\epsilon}h}\lesssim \|\alpha\|_{ds^2_{\delta},w_\delta h}.$$ The asymptotic behavior (\ref{align_asymptotic_Gt}) and the $0$-tameness of the harmonic metric $h$ ensure that $$\beta\in  L_{(2)}^{k-1}(U_\delta\cap Z^o,\bV;ds^2_{\delta,\epsilon},w_{\delta,\epsilon}h)$$ for every $\epsilon\in(0,1)$. Hence there is an estimate
		\begin{align}\label{align_intro_estimate_n}
		|(\alpha,u)|_{ds^2_{\delta,\epsilon},w_{\delta,\epsilon}h}\leq C_\epsilon\|\nabla^\ast_{\delta,\epsilon} u\|_{ds^2_{\delta,\epsilon},w_{\delta,\epsilon}h},\quad \forall u\in A^{k}_{\rm cpt}(U_{\delta}\cap Z^o,\bV), \forall k\leq \dim Z.
		\end{align}
		Here $\nabla^\ast_{\delta,\epsilon}$ is the Hilbert adjoint of $\nabla$ with respect to $(ds^2_{\delta,\epsilon},w_{\delta,\epsilon}h)$ and $$C_\epsilon\leq\|\beta\|_{ds^2_{\delta,\epsilon},w_{\delta,\epsilon}h}.$$
		Although $\|\beta\|_{ds^2_{\delta,\epsilon},w_{\delta,\epsilon}h}$ may tend to infinity when $\epsilon\to 0$, a twisted Donelly-Fefferman type estimate (Theorem \ref{thm_L2-estimate_main}) for $(ds^2_{\delta,\epsilon},w_{\delta,\epsilon}h)$ shows that the constants $C_\epsilon$ could be chosen in a way independent of $\epsilon\in(0,1)$.
		Hence (\ref{align_intro_aim_est}) is the limit of (\ref{align_intro_estimate_n}). This solves the equation $\nabla\beta=-\alpha$.

		To get a solution of $\nabla\beta=-\alpha$ with compact supports, one may shrink $U_\delta$ first and use the zero extension of the solution. This is valid due to the introducing of the weight $w_{\delta}$ (Lemma \ref{lem_i<n}). As a consequence of the local vanishings there is a quasi-isomorphism $IC_X(\bV)\simeq_{\rm qis}\sD^\bullet_{X,\bV;ds^2,h}$.

		\item Make the quasi-isomorphism $IC_X(\bV)\simeq_{\rm qis} \sD_{X,\bV;ds^2,h}$ canonical. This can be done by Proposition \ref{prop_GM_criterion}.
	\end{enumerate}
	\begin{rmk}
		Recall that in \cite{Saper1992,Cattani_Kaplan_Schmid1987,Kashiwara_Kawai1987}, the purity theorem which comes from the Hodge theory of singularities plays crucial roles in the proof of (\ref{align_vanishing}). Our proof uses a different strategy. The vanishings come from the  $L^2$-estimate for the $\nabla$-operator. It is interesting to formulate a purity theorem for an arbitrary complex analytic singularity.
	\end{rmk}
	\begin{rmk}
		The choice of the deformation $(ds^2_{\delta,\epsilon},w_\delta)$ is deeply inspired by Ohsawa's work \cite{Ohsawa1993} on isolated singularities. To perform the trick by Ohsawa to non-isolated singularities, one faces the difficulty that the intersection $(X\backslash X^o)\cap \partial U_\delta$ is always non-empty when $X\backslash X^o$ is non-isolated. The metric $ds^2_{\delta}$ on $(X\backslash X^o)\cap \partial U_\delta$ is a complicated mixture of $ds^2$ and the modifier $ds^2_\delta-ds^2$. Hence it is difficult to study the local $L^2$-cohomology at $(X\backslash X^o)\cap \partial U_\delta$. To handle this analytical difficulty we choose to calculate (\ref{align_vanishing_cpt}) instead of (\ref{align_vanishing}) and introduce the trick of two neighborhoods $U_\delta$, $V_\delta$. So we may assume that ${\rm supp}\alpha$ is far away from $(X\backslash X^o)\cap \partial U_\delta$ in the equation $\nabla\beta=-\alpha$.
	\end{rmk}
The paper is organized as follows.

Section 2 and 3 are the review of some basic knowledge of the intersection complex, the  $L^2$-cohomology and the  $L^2$-de Rham complex, in order to fix notations. Some preliminary results are proved.

In Section 4 we establish the basic theory of Whitney stratified desingularization of a pair of complex spaces. Distinguished metrics and the notion of $0$-tame harmonic bundles are introduced. We also prove that they are stratified along the bimeromorphic fibering neighborhood associated to a given Whitney stratified desingularization. The weakly constructibility of the $L^2$-de Rham complex is then a consequence.

Section 5 is devoted to the geometry and the asymptotic behavior of the distinguished metric at an arbitrary complex singularity. The Whitney stratified desingularization is the basic tool. We fix the basic setting and constructions in \S\ref{section_setting} and \S\ref{section_U_delta} which will be used in the rest of the paper.

We prove the key vanishing theorem of this paper in Section 6.

Everything is combined together in Section 7 to give a bi-induction proof of the main theorem.

Section 8 contains some applications of the main theorem which have been mentioned in the introduction.

	\textbf{Acknowledgment:}
	The deepest thank goes to professor Jun Li for bringing this problem to the authors' interest and his continual help and discussions during their postdoctoral period and beyond. We would like to thank professor Takeo Ohsawa and professor Adam Parusi\'nski for answering our questions patiently.
	We also wish to thank Zhenqian Li and Guozhen Wang for helpful conversations.
	Finally, we would like to thank professor Sen Hu, professor Kefeng Liu, professor Mao Sheng, professor Xiaotao Sun and professor Kang Zuo for their interest in this work and their continual encouragements.

	{\bf Notations:}
	\begin{itemize}
		\item All complex spaces are assume to be reduced, separated, paracompact, countable at infinity and of pure dimension.
		\item $h^i(\sA^\bullet)$ stands for the $i$-th cohomology sheaf of the complex of sheaves $\sA^\bullet$.
		\item Let $\alpha,\beta$ be functions (or metrics, $(1,1)$-forms, etc.) on a complex space. Denote by $\alpha\lesssim\beta$ if $C\beta-\alpha\geq0$ for some constant $C>0$.
		Denote $\alpha\sim\beta$ if both $\alpha\lesssim\beta$ and $\beta\lesssim\alpha$  are satisfied.
		\item Let $M$ be a complex manifold and $\bV$ a local system, $A^k(M,\bV)$ (resp. $A^k_{\rm cpt}(M,\bV)$) stands for the space of $C^\infty$ $\bV$-valued $k$-forms (resp. with compact supports) on $M$. Denote by $\sA^k_M(\bV)$ the sheaf of $\bV$-valued $C^\infty$ $k$-forms on $M$.
	\end{itemize}
	\section{Intersection Complex on Complex Spaces}
	Due to the nature of the  $L^2$-cohomology, we need to consider constructible sheaves that may not be locally finitely generated. In this section we review the basic knowledge of such sheaves and prove some technical results for later usage.

	Let $X$ be a complex space of pure dimension $n$ and $X^o\subset X_{\rm reg}$ a dense Zariski open subset.
	Throughout this paper, a stratification of the pair $(X,X^o)$ is a topological stratification
	\begin{align*}
	\emptyset=X_{-1}\subset X_0\subset\cdots \subset X_n=X
	\end{align*}
	in the sense of \cite[\S I-1.1]{Borel1984}, in addition that $X_{n-1}=X\backslash X^o$ and $X_i$ is a Zariski closed analytic subset for every $i=0,\dots,n$. Roughly speaking, $X_i\backslash X_{i-1}$ is an $i$-dimensional complex submanifold of $X$ for every $i$ and the links along each $X_i\backslash X_{i-1}$ are topologically isotrivial. A typical example of a stratification is the Whitney stratification (c.f. \cite{Mather2012}).

	A connected component of $\coprod_i X_i\backslash X_{i-1}$ is called a stratum of $\{X_i\}$.
	\begin{defn}
		Let $X$ be a complex space and $A$ a sheaf of $\bC$-vector spaces on $X$. $A$ is called a weakly local system on $X$ if for every point $x\in X$ there is a neighborhood $x\in U$ such that $A|_U\simeq \bC_U^I$, where $I$ is an (possibly infinite) index set. In this case we say that $A$ is a local system if the  $I$s are finite sets.

		$A$ is called (weakly) constructible if there is a stratification $\{X_i\}$ such that $A|_{X_i\backslash X_{i-1}}$ is a (weakly) local system for every $i$. In this situation we also say that $A$ is (weakly) constructible with respect to $\{X_i\}$.
	\end{defn}
	The category of weakly constructible sheaves on $X$ is an abelian category.

	Let $\bV$ be a local system on $X^o$. Denote   $U_i:=X\backslash X_{i}$ and denote by $j_i:U_i\to U_{i-1}$ ($U_{-1}:=X$) the inclusion. The intersection complex $IC_X(\bV)$ is defined, in the derived category $D(X)$ of sheaves of $\bC$-vector spaces on $X$, by
	\begin{align}\label{align_intersection_complex}
	IC_X(\bV):=R^{< n}j_{0\ast}R^{< n-1}j_{1\ast}\cdots R^{< 1}j_{n-1\ast}\bV.
	\end{align}
	The right handside is independent of the stratification $\{X_i\}$ up to quasi-isomorphisms and (\ref{align_intersection_complex}) gives a well defined bounded complex in $D(X)$.

	In order to make the intersection complex a perverse sheaf (\cite{BBD}), one usually uses a Tate twist $IC_X(\bV)[n]$ instead of $IC_X(\bV)$. In this paper the difference between twisting causes no problem.

	The following proposition is a mild generalization of Goresky-MacPherson-Deligne criterion for intersection complexes to weakly constructible complexes. It also shows that in the set of isomorphisms ${\rm Iso}_{D(X)}(\sA^\bullet,IC_X)$, there is a canonical one, represented by a natural zigzag of quasi-isomorphisms between complexes, as long as a quasi-isomorphism $\iota_0:\bV\to \sA^\bullet|_{X^o}$ is fixed.  A canonical isomorphism is necessary if we want to compare the Hodge structures on both sides of Theorem \ref{thm_main_harmonic} (Conjecture \ref{conj_Hodge_structure}).
	\begin{prop}[Goresky-MacPherson-Deligne criterion]\label{prop_GM_criterion}
		Let $X$ be a complex space of pure dimension $n$ and $X^o\subset X_{\rm reg}$ a dense Zariski open subset. Let $\{X_i\}$ be a stratification of $(X,X^o)$ and $\bV$ a local system on $X^o$.
		A bounded complex of sheaves of $\bC$-vector spaces $\sA^\bullet$ is isomorphic to the intersection complex $IC_X(\bV)$ in $D(X)$ if and only if the following conditions hold.
		\begin{enumerate}
			\item $\sA^\bullet|_{X^o}$ is quasi-isomorphic to $\bV$.
			\item For every $j$, $h^j(\sA^\bullet)$ is weakly constructible with respect to $\{X_i\}$.
			\item For each stratum $S$ of $\{X_i\}$ and every point $x\in S$, denote by $i_x:\{x\}\to X$ the inclusion. Then
			\begin{align}\label{align_GMD_vanishing}
			h^k(i_x^\ast \sA^\bullet)=0,\quad \forall k\geq n-\dim S
			\end{align}
			and
			\begin{align}\label{align_GMD_covanishing}
			h^k(i_x^! \sA^\bullet)=0,\quad \forall k\leq n+\dim S.
			\end{align}
		\end{enumerate}
		Moreover, given a quasi-isomorphism of complexes
		$\iota_0:\bV\to \sA^\bullet|_{X^o}$,
		there is a canonical zigzag of quasi-isomorphisms between complexes (rather than morphisms in the derived category)
		\begin{align*}
		&\sA^\bullet\leftarrow\tau^{< n}\sA^\bullet\to\tau^{< n}Rj_{0\ast}\left(\sA^\bullet|_{U_0}\right)\leftarrow\tau^{< n}Rj_{0\ast}\left(\tau^{<n-1}\sA^\bullet|_{U_0}\right)\to\tau^{<n}Rj_{0\ast}\tau^{<n-1}Rj_{1\ast}\left(\sA^\bullet|_{U_1}\right)\leftarrow\\\nonumber
		&\cdots\leftarrow\tau^{<n}Rj_{0\ast}\cdots\tau^{<{2}}Rj_{n-2\ast}\left(\tau^{<{1}}\sA^\bullet|_{U_{n-2}}\right)\to\tau^{<n}Rj_{0\ast}\cdots\tau^{<1}Rj_{n-1\ast}\left(\sA^\bullet|_{X^o}\right)\leftarrow IC_X(\bV)\nonumber
		\end{align*}
	\end{prop}
	\begin{proof}
		The first part is well known (c.f. \cite[chap. 5]{Borel1984} or \cite[Prop. 8.2.11]{Ryoshi-Kiyoshi-Toshiyuki2008}). To prove the second part, it suffices to show that the natural diagram
		\begin{align*}
		\xymatrix{
			\tau^{< n-d}\sA^\bullet|_{U_{d-1}}\ar[d]\ar[r] & \tau^{< n-d}Rj_{d\ast}\left(\sA^\bullet|_{U_{d}}\right)\\
			\sA^\bullet|_{U_{d-1}}&
		}
		\end{align*}
		is a diagram of quasi-isomorphisms for every $d$. Consider the distinguished triangle
		\begin{align*}
		\tau^{< n-d}\sA^\bullet|_{U_{d-1}}\to\sA^\bullet|_{U_{d-1}}\to\tau^{\geq n-d}\sA^\bullet|_{U_{d-1}}\to\tau^{< n-d}\sA^\bullet|_{U_{d-1}}[1].
		\end{align*}
		By (\ref{align_GMD_vanishing}), $\tau^{\geq n-d}\sA^\bullet|_{U_{d-1}}=0$. Hence the canonical map $$\tau^{< n-d}\sA^\bullet|_{U_{d-1}}\to\sA^\bullet|_{U_{d-1}}$$ is a quasi-isomorphism. Denote by $i:U_{d-1}\backslash U_d\to U_{d-1}$ and $i_x:\{x\}\to U_{d-1}\backslash U_d$ the closed immersions. Since $U_{d-1}\backslash U_d$ is smooth with real dimension $2d$, we have $i_x^!\simeq i_x^{\ast}[-2d]$. Hence by (\ref{align_GMD_covanishing}), we see that
		\begin{align*}
		h^k((i_xi)^!\sA^\bullet|_{U_{d-1}})\simeq i^\ast_xh^{k-2d}(i^!\sA^\bullet|_{U_{d-1}})=0,\quad \forall k\leq n+d.
		\end{align*}
		Equivalently $\tau^{\leq n-d}i_!i^!\sA^\bullet|_{U_{d-1}}=0$.

		This, together with the distinguished triangle
		\begin{align*}
		i_!i^!\sA^\bullet|_{U_{d-1}}\to\sA^\bullet|_{U_{d-1}}\to Rj_{d\ast}\left(\sA^\bullet|_{U_{d}}\right)\to i_!i^!\sA^\bullet|_{U_{d-1}}[1]
		\end{align*}
		shows that the canonical map
		$$\tau^{< n-d}\sA^\bullet|_{U_{d-1}}\to \tau^{< n-d}Rj_{d\ast}\left(\sA^\bullet|_{U_{d}}\right)$$
		is a quasi-isomorphism.
	\end{proof}
	\begin{rmk}\label{rmk_GM_criterion}
		Condition (3) in Proposition \ref{prop_GM_criterion} is equivalent to the following statement. For every $x\in X_i\backslash X_{i-1}$,
		\begin{align*}
		\varinjlim_{x\in U_x}\bH^k(U_x,\sA^\bullet)=0,\quad\forall k\geq n-i
		\end{align*}
		and
		\begin{align*}
		\varprojlim_{x\in U_x}\bH^k_{\rm cpt}(U_x,\sA^\bullet)=0,\quad \forall k\leq n+i.
		\end{align*}
	\end{rmk}
	\begin{cor}\label{cor_IC_is_local}
		Let $X$ be a complex space and $X^o\subset X_{\rm reg}$ a dense Zariski open subset. Let $\bV$ be a local system on $X^o$ and $X=\bigcup_{i\in I}U_i$ an open covering. Then a complex of sheaves of $\bC$-vector spaces $\sA^\bullet$ is quasi-isomorphic to $IC_X(\bV)$ if and only if $\sA^\bullet|_{X^o}\simeq\bV$ and $\sA^\bullet|_{U_i}$ is quasi-isomorphic to $IC_{U_i}(\bV|_{U_i\cap X^o})$ for every $i\in I$.
	\end{cor}
	\begin{proof}
		Notice that a sheaf $A$ is weakly constructible if $A|_{U_i}$ is weakly constructible for every $i\in I$ (\cite[Proposition 4.1.13]{Dimca2004}). Now the corollary follows from the Goresky-MacPherson-Deligne   Criterion (Proposition \ref{prop_GM_criterion}).
	\end{proof}
	\section{$L^2$-de Rham Complex and $L^2$-Cohomology}
	\subsection{$L^2$-cohomology} Let $(M,ds^2_M)$ be an orientable Riemannian stratified space with boundary and $\bV$ a complex local system on $M$. For this notion we always assume that $ds^2_M$ can be extended smoothly to the boundary. Denote by $(\cV:=\bV\otimes_\bC\sO_M,\nabla)$ the associated flat bundle. Let $h$ be a hermitian metric on $\cV$ which is not necessarily compatible with $\nabla$. Such a pair $(\bV,h)$ is called a \textit{hermitian local system}. Denote by $L_{(2)}^k(M,\bV;ds^2_M,h)$ the space of $\bV$-valued measurable $k$-forms which are square integrable with respect to $(ds^2_M,h)$. Let
	$$D^k_{\rm max}(M,\bV;ds^2_M,h):=\left\{\alpha\in L_{(2)}^k(M,\bV;ds^2_M,h)\big|\nabla\alpha\in L_{(2)}^{k+1}(M,\bV;ds^2_M,h)\right\}$$
	and
	$$D^\bullet_{\rm max}(M,\bV;ds^2_M,h):=D^0_{\rm max}(M,\bV;ds^2_M,h)\stackrel{\nabla}{\to}D^1_{\rm max}(M,\bV;ds^2_M,h)\stackrel{\nabla}{\to}\cdots\stackrel{\nabla}{\to} D^{\dim_\bR M}_{\rm max}(M,\bV;ds^2_M,h).$$
	Here $\nabla$ are operators in the sense of distribution.

	The $L^2$-cohomology of $(M,\bV;ds^2_M,h)$ with the maximal ideal boundary condition is defined by
	$$H^\ast_{(2),\rm max}(M,\bV;ds^2_M,h):= H^\ast\left(D^\bullet_{\rm max}(M,\bV;ds^2_M,h)\right).$$
	Subscripts such as $ds^2_M$ and $h$ would be omitted when they are clear from the context. $\bV$ would be omitted when $\bV\simeq\bC$ is the trivial local system.

	There is another ideal boundary condition for the operator $\nabla$. Let
	$$D^k_{\rm min}(M,\bV;ds^2_M,h):=\left\{\alpha\in D^k_{\rm max}(M,\bV;ds^2_M,h)\big|\exists \alpha_n\in A^k_{\rm cpt}(M,\bV),  \alpha_n\to\alpha, \nabla\alpha_n\to \nabla\alpha\right\}$$
	and
	$$D_{\rm min}^\bullet(M,\bV;ds^2_M,h):=D^0_{\rm min}(M,\bV;ds^2_M,h)\stackrel{\nabla}{\to}D^1_{\rm min}(M,\bV;ds^2_M,h)\stackrel{\nabla}{\to}\cdots\stackrel{\nabla}{\to} D^{\dim_\bR M}_{\rm min}(M,\bV;ds^2_M,h).$$

	The $L^2$-cohomology of $(M,ds^2_M)$ with minimal ideal boundary condition is defined by
	$$H^\ast_{(2),\rm min}(M,\bV;ds^2_M,h):= H^\ast\left(D^\bullet_{\rm min}(M,\bV;ds^2_M,h)\right).$$

	To distinguish the ideal boundary conditions we denote by $\nabla_{\rm max}$ and $\nabla_{\rm min}$ the unbounded operators with domains ${\rm Dom}\nabla_{\rm max}:=\bigoplus_{i}D^i_{\rm max}(M,\bV;ds^2_M,h)$ and ${\rm Dom}\nabla_{\rm min}:=\bigoplus_{i}D^i_{\rm min}(M,\bV;ds^2_M,h)$ respectively.
	These two ideal boundary conditions coincide when $M$ is a manifold without boundary and $ds^2_M$ is complete (\cite[Theorem 7.2]{Zucker1979}). Hence there is a canonical isomorphism
	\begin{align*}
	H^k_{(2),\rm min}(M,\bV;ds^2_M,h)\simeq H^k_{(2),\rm max}(M,\bV;ds^2_M,h)
	\end{align*}
	whenever $(M,ds^2_M)$ is a complete Riemannian manifold (without boundary). In this case we denote the $L^2$-cohomology by $H^\ast_{(2)}(-)$ for short.

	Generally these two kinds of $L^2$-cohomologies are related by the $L^2$-Poincar\'e duality.
	\begin{prop}\label{prop_L2_Poincare_duality}
		Let $(M,ds^2_M)$ be an orientable Riemannian manifold.
		Assume that either of the following statements holds.
		\begin{enumerate}
			\item $H^k_{(2),\rm min}(M,\bV^\ast;ds^2_M,h^\ast)$ and $H^{k+1}_{(2),\rm min}(M,\bV^\ast;ds^2_M,h^\ast)$ are finite dimensional.
			\item $H^k_{(2),\rm min}(M,\bV^\ast;ds^2_M,h^\ast)$ and $H^{\dim_\bR M-k}_{(2),\rm max}(M,\bV;ds^2_M,h)$ are finite dimensional.
		\end{enumerate}
		Then there is a well defined perfect pairing
		$$H^{\dim_\bR M-k}_{(2),\rm max}(M,\bV;ds^2_M,h)\times H^k_{(2),\rm min}(M,\bV^\ast;ds^2_M,h^\ast)\to\bC,\quad ([\alpha],[\beta])\mapsto \int_{M}\alpha\wedge\beta.$$
		Here $(\bV^\ast,h^\ast)$ is the dual hermitian local system.
	\end{prop}
	\begin{proof}
		Denote  $L^k_{(2)}:=L^{k}_{(2)}(M,\bV;ds^2_M,h)$ and $L^{\ast k}_{(2)}:=L^{k}_{(2)}(M,\bV^\ast;ds^2_M,h^\ast)$ for short.
		Consider the diagram
		\begin{align}\label{align_Poincare_dual_1}
		\xymatrix{
		L^{\ast k-1}_{(2)}\ar@<1ex>[r]^{\nabla_{\rm min}} & L^{\ast k}_{(2)} \ar@<1ex>[r]^{\nabla_{\rm min}}\ar@<1ex>[l]^{(\nabla_{\rm min})^\ast} & L^{\ast k+1}_{(2)} \ar@<1ex>[l]^{(\nabla_{\rm min})^\ast}
	    }
		\end{align}
		and its Hodge star dual
		\begin{align}\label{align_Poincare_dual_2}
		\xymatrix{
			L^{n-k+1}_{(2)}\ar@<1ex>[r]^{(\nabla_{\rm max})^\ast} & L^{n-k}_{(2)} \ar@<1ex>[r]^{(\nabla_{\rm max})^\ast}\ar@<1ex>[l]^{\nabla_{\rm max}} & L^{n-k-1}_{(2)} \ar@<1ex>[l]^{\nabla_{\rm max}}
		},\quad n=\dim_\bR M.
		\end{align}
		Under either of the assumptions all the eight operators in (\ref{align_Poincare_dual_1}) and (\ref{align_Poincare_dual_2}) have closed ranges. Therefore there are Hodge decompositions (\cite[Theorem A.2.2]{Kashiwara_Kawai1987})
		\begin{align*}
		L^{\ast k}_{(2)}=\sH^{\ast k}_{\rm min}\oplus \nabla_{\rm min}(L^{\ast k-1}_{(2)})\oplus(\nabla_{\rm min})^\ast(L^{\ast k+1}_{(2)})
		\end{align*}
		and
		\begin{align*}
		L^{n-k}_{(2)}=\sH^{n-k}_{\rm max}\oplus \nabla_{\rm max}(L^{n-k-1}_{(2)})\oplus(\nabla_{\rm max})^\ast(L^{n-k+1}_{(2)}),
		\end{align*}
		where $$\sH^{\ast k}_{\rm min}={\rm Ker}\nabla_{\rm min}\cap {\rm Ker}(\nabla_{\rm min})^\ast\cap L^{\ast k}_{(2)}$$
		and $$\sH^{n-k}_{\rm max}={\rm Ker}\nabla_{\rm max}\cap {\rm Ker}(\nabla_{\rm max})^\ast\cap L^{n-k}_{(2)}.$$
		Hence there are isomorphisms
		\begin{align*}
		H^k_{(2),\rm min}(M,\bV^\ast;ds^2_M,h^\ast)\simeq \sH^{\ast k}_{\rm min}
		\end{align*}
		and
		\begin{align*}
		H^{n-k}_{(2),\rm max}(M,\bV;ds^2_M,h)\simeq \sH^{n-k}_{\rm max}.
		\end{align*}
		Since the Hodge star operator gives an isomorphisms between $\sH^{\ast k}_{\rm min}$ and $\sH^{n-k}_{\rm max}$, the pairing
		$$\sH^{n-k}_{\rm max}\times \sH^{\ast k}_{\rm min}\to\bC,\quad (\alpha,\beta)\mapsto \int_{M}\alpha\wedge\beta=\pm\int_M\langle\alpha,\ast\beta\rangle{\rm vol}$$
		is perfect.
	\end{proof}
	The $L^2$-cohomology with minimal ideal boundary condition admits the Kunneth-type formula, as is observed by Zucker \cite{Zucker1982}.
	\begin{prop}\label{prop_L2_Kunneth}
		Let $(M,ds^2_M)$ and $(N,ds^2_N)$ be Riemannian manifolds with boundaries. Let $(\bV_1,h_1)$ and $(\bV_2,h_2)$ be hermitian local systems on $M$ and $N$ respectively. Let $ds^2_{X}:=ds^2_M+ds^2_N$ be the product metric on $X:=M\times N$ and let  $(\bV,h):=(\bV_1\boxtimes\bV_2,h_1h_2)$ be the product hermitian local system on $X$. Assume that $H^\ast_{(2),\rm min}(M,\bV_1;ds^2_M,h_1)$ are finite dimensional. Then there are canonical isomorphisms
		\begin{align}\label{align_L2_Kunneth}
		H^k_{(2),\rm min}(X,\bV;ds^2_{X},h)\simeq \bigoplus_{i+j=k}H^i_{(2),\rm min}(M,\bV_1;ds^2_M,h_1)\otimes H^j_{(2),\rm min}(N,\bV_2;ds^2_N,h_2).
		\end{align}
	\end{prop}
	\begin{proof}
		The proof is originated by \cite[Theorem 2.29]{Zucker1982} with a mild modification. For the convenience of readers we give a sketched proof.

		Since $\bigoplus_{i+j=k}A_{\rm cpt}^i(M,\bV_1)\otimes A_{\rm cpt}^j(N,\bV_2)$ is dense in $A_{\rm cpt}^k(X,\bV)$ under the graph norm $\|-\|+\|\nabla-\|$, by taking completion under the graph norm we obtain
		\begin{align*}
		D^k_{\rm min}(X,\bV;ds^2_{X},h)\simeq \bigoplus_{i+j=k}D^i_{\rm min}(M,\bV_1;ds^2_M,h_1)\widehat{\otimes} D^j_{\rm min}(N,\bV_2;ds^2_N,h_2).
		\end{align*}
		Denote by $\nabla_1$ the flat connection associated to $\bV_1$.
		Since $H^\ast_{(2),\rm min}(M,\bV_1;ds^2_M,h_1)$ is finite dimensional, ${\rm Im}\nabla_{1,\rm min}$ is closed and there is a bounded operator
		\begin{align*}
		H_M:D_{\rm min}^\bullet(M,\bV_1;ds^2_M,h_1)\to D_{\rm min}^{\bullet-1}(M,\bV_1;ds^2_M,h_1)
		\end{align*}
		such that
		\begin{align}\label{align_L2_Kunneth_homotopy_id}
		{\rm Id}=P_M+\nabla_{1,\rm min}H_M+H_M\nabla_{1,\rm min}
		\end{align}
		on $D^\bullet_{\rm min}(M,\bV_1;ds^2_M,h_1)$.
		Here $$P_M:L_{(2)}^\bullet(M,\bV_1;ds^2_M,h_1)\to\sH^\bullet(M,\bV_1;ds^2_M,h_1):={\rm Ker}(\nabla_{1,\rm min})\cap{\rm Ker}(\nabla_{1,\rm min})^\ast$$ is the projection via the abstract Hodge decomposition
		\begin{align*}
		L_{(2)}^\bullet(M,\bV_1;ds^2_M,h_1)=\sH^\bullet(M,\bV_1;ds^2_M,h_1)\oplus{\rm Im}\nabla_{1,\rm min}\oplus{\rm Im}(\nabla_{1,\rm min})^\ast.
		\end{align*}
		$(\nabla_{1,\rm min})^\ast$ is the Hilbert adjoint of $\nabla_{1,\rm min}$ and
		$H_M=(\nabla_{1,\rm min})^\ast G_M$ where $G_M$ is the abstract Green operator (\cite[Theorem A.2.2]{Kashiwara_Kawai1987}).
		This shows that the bounded operator $P_M$ gives a homotopy equivalence from $D^\bullet_{\rm min}(M,\bV_1;ds^2_M,h_1)$ to the null complex $\sH^\bullet(M,\bV_1;ds^2_M,h_1)$.

		By boundedness, (\ref{align_L2_Kunneth_homotopy_id}) naturally extends to the homotopy to identity formula on
		$$D^\bullet_{\rm min}(M,\bV_1;ds^2_M,h_1)\widehat{\otimes} D^\bullet_{\rm min}(N,\bV_2;ds^2_N,h_2)$$
		and implies that it is homotopic  equivalent to
		\begin{align}\label{align_L2_Kunneth_2}
		\sH^\bullet(M,\bV_1;ds^2_M,h_1)\widehat{\otimes} D^\bullet_{\rm min}(N,\bV_2;ds^2_N,h_2)
		=\sH^\bullet(M,\bV_1;ds^2_M,h_1)\otimes D^\bullet_{\rm min}(N,\bV_2;ds^2_N,h_2).
		\end{align}
		The reason for the last equality is the assumption that $$H^\ast_{(2),\rm min}(M,\bV_1;ds^2_M,h_1)\simeq \sH^\ast(M,\bV_1;ds^2_M,h_1)$$ is finite dimensional. Now the proposition is proved because the cohomologies of (\ref{align_L2_Kunneth_2}) are the right hand side of (\ref{align_L2_Kunneth}).
	\end{proof}
	Another feature of the minimal boundary condition is that the extension of a form by zero is well defined.
	Let $U\subset M$ be an open subset, then
	the map
	\begin{align*}
	A^\bullet_{\rm cpt}(U,\bV)\to A^\bullet_{\rm cpt}(M,\bV),\quad \alpha\mapsto\widetilde{\alpha}=\begin{cases}
	\alpha, & x\in U \\
	0, & x\in M\backslash U
	\end{cases}
	\end{align*}
	is bounded under the graph norm $\|-\|+\|\nabla-\|$. Hence it induces a map
	\begin{align*}
	D^\bullet_{\rm min}(U,\bV;ds^2_M,h)\to D^\bullet_{\rm min}(M,\bV;ds^2_M,h),\quad \alpha\mapsto\widetilde{\alpha}.
	\end{align*}
	We will simply call this map \textit{extension by zero} when the spaces are clear in the context.
	\subsection{$L^2$-de Rham complex} Let $X$ be a real or complex analytic space and $X^o\subset X_{\rm reg}$ an open subset of the loci of regular points $X_{\rm reg}$ so that $X^o$ is dense in $X$. Let $ds^2$ be a Riemannian metric on $X^o$ and $(\bV,h)$ a hermitian local system on $X^o$. Denote by $(\cV:=\bV\otimes_\bC\sO_M,\nabla)$ the associated flat bundle. Denote by $\sL^k_{X,\bV;ds^2,h}$ the sheaf of $\bV$-valued measurable $k$-forms on $X$ which are locally square integrable with respect to the metrics $ds^2$ and $h$, i.e. for every open subset $U\subset X$, $\sL^k_{X,\bV;ds^2,h}(U)$ is the space of measurable $\bV$-valued $k$-forms $\alpha$ on $U\cap X^o$ such that locally at every point $x\in U$ there is a neighborhood $V_x\subset U$ of $x$ so that
	$$\int_{V_x\cap X^o}|\alpha|^2_{ds^2,h}{\rm vol}_{ds^2}<\infty.$$
	Denote
	$$\sD^k_{X,\bV;ds^2,h}:=\left\{\alpha\in \sL^k_{X,\bV;ds^2,h}\big|\nabla\alpha\in \sL^{k+1}_{X,\bV;ds^2,h}\right\}.$$
	The {\bf $L^2$-de Rham complex} of $(X,\bV;ds^2,h)$ is defined by
	$$\sD^\bullet_{X,\bV;ds^2,h}:=\sD^0_{X,\bV;ds^2,h}\stackrel{\nabla}{\to}\sD^1_{X,\bV;ds^2,h}\stackrel{\nabla}{\to}\cdots\stackrel{\nabla}{\to} \sD^{\dim_{\bR}X}_{X,\bV;ds^2,h},$$
	where $\nabla$ are operators in the sense of distribution.
	\begin{defn}\label{defn_hermitian_metric}
		Let $X$ be a complex space and $ds^2$ a hermitian metric on $X_{\rm reg}$. We say that $ds^2$ is a {\bf hermitian metric on} $X$ if for every $x\in X$ there is a neighborhood $U$ and a holomorphic closed immersion $U\subset V$ into a complex manifold such that $ds^2|_U\sim ds^2_V|_{U}$ for some hermitian metric $ds^2_V$ on $V$. If moreover the associated $(1,1)$-form associated to $ds^2$ is $d$-closed on $X^o$, we say that $ds^2$ is a {\bf K\"ahler hermitian metric on} $X$
	\end{defn}
	\begin{rmk}
		To distinguish hermitian metrics on $X$ and hermitian metrics on the open set $X^o$, we will always emphasize the space on which the metric lies.
	\end{rmk}
	Many interesting metrics, including the distinguished metric introduced in this paper, are locally bounded from below by a hermitian one. In this case the $L^2$-de Rham complexes are complexes of fine sheaves.
	\begin{lem}\label{lem_fine_sheaf}
		Let $X$ be a complex space and $X^o\subset X_{\rm reg}$ a dense Zariski open subset. Let $ds^2$ be a hermitian metric on $X^o$ and $(\bV,h)$ a hermitian local system on $X^o$. Suppose that for every point $x\in X\backslash X^o$ there is a neighborhood $x\in U_x$ and a hermitian metric $ds^2_0$ on $U_x$ such that $ds^2_0|_{X^o\cap U_x}\lesssim ds^2|_{X^o\cap U_x}$. Then
		\begin{enumerate}
			\item for every $k$, $\sD^{k}_{X,\bV;ds^2,h}$ is a fine sheaf.
			\item For every $C^\infty$ form $\alpha$ on $X$, $|\alpha|_{ds^2}$ is locally $L^\infty$ bounded.
		\end{enumerate}
	\end{lem}
	\begin{proof}
		It suffices to show that for every $W\subset\overline{W}\subset V\subset X$ where $W$ and $V$ are small open subsets, there is a continuous function $f$ on $V$ such that
		\begin{itemize}
			\item ${\rm supp}(f)\subset \overline{W}$,
			\item $f$ is $C^\infty$ on $V\cap X^o$
			\item $df$ has bounded fiberwise norm with respect to the metric $ds^2$.
		\end{itemize}
		We may assume that there is a closed embedding $V\subset M$ where $M$ is a complex manifold. Let $W'\subset\overline{W'}\subset M$ where $W'$ is an open subset such that $W'\cap V=W$. Let $ds^2_M$ be a hermitian metric on $M$ so that $ds^2_0|_{V\cap X^o}\sim ds^2_M|_{V\cap X^o}$. Let $g$ be a smooth function on $M$ whose support lies in $\overline{W'}$. Denote $f=g|_V$, then ${\rm supp}(f)\subset \overline{W}$ and $f$ is $C^\infty$ on $V\cap X^o$. Since $V\cap X^o\subset M$ is a submanifold, one has the orthogonal decomposition
		\begin{align*}
		T_{M,x}=T_{V\cap X^o,x}\oplus T_{V\cap X^o,x}^\bot, \quad \forall x\in V\cap X^o.
		\end{align*}
		Consequently, the restriction map $T^\ast_{M,x}\to T^\ast_{V\cap X^o,x}$ is distance-decreasing for every $x\in V\cap X^o$.
		Therefore $|df|_{ds^2}\lesssim|d f|_{ds^2_0}\lesssim |dg|_{ds^2_M}<\infty$. This proves the first claim of the lemma. For the second claim, assume that $\alpha|_V=\alpha'|_V$ for some smooth form $\alpha'$ on $M$, then $|\alpha|_{ds^2}\lesssim|\alpha|_{ds^2_0}\lesssim |\alpha'|_{ds^2_M}<\infty$.
	\end{proof}
	When $X$ is compact,  $$\Gamma\left(X,\sD^\bullet_{X,\bV;ds^2,h}\right)=D_{\rm max}^\bullet(X,\bV;ds^2,h).$$
	Therefore we have
	\begin{cor}
		Notations as in Lemma \ref{lem_fine_sheaf}. Assume that $X$ is compact, then there is a natural isomorphism
		$$H^\ast_{(2),\rm max}(X^o,\bV;ds^2,h)\simeq \bH^\ast(X,\sD^\bullet_{X,\bV;ds^2,h}).$$
	\end{cor}
	\section{Constructibility of $L^2$-de Rham Complex}
	\subsection{$L^2$-Poincar\'e Lemma and $L^2$-Thom isomorphism}
	This subsection is a preparation for the proof of the main result (Theorem \ref{prop_Dstr}) of the next subsection.
	In particular we prove an $L^2$-Poincar\'e lemma (Lemma \ref{lem_Kunneth_locL2}) and an $L^2$-Thom isomorphism for trivial line bundles (Lemma \ref{lem_Thom_locL2}).

	Let $(X,\bV;ds^2,h)$ be a Riemannian stratified space with a hermitian local system on the union of open strata $X^o:=X\backslash X_{\dim_\bR X-2}$ . Denote by $q:(0,1)\times X\to X$ the projection. Let $ds^2_{(0,1)\times X}:=d\rho^2+ds^2$ and let $q^\ast h$ be the pullback metric on $q^\ast \bV$. The pullback operator $q^\ast$ sends a locally $L^2$-form to a locally $L^2$-form. Hence it gives a well defined map
	\begin{align}\label{align_L2_Poincare_q*}
	q^\ast: \sD^\bullet_{X,\bV;ds^2,h}(X)\to \sD^\bullet_{(0,1)\times X,q^\ast\bV;ds^2_{(0,1)\times X},q^\ast h}((0,1)\times X).
	\end{align}
	\begin{lem}\label{lem_Kunneth_locL2}
		Notations as above, the canonical morphism (\ref{align_L2_Poincare_q*})
		is a homotopy equivalence.
	\end{lem}
	\begin{proof}
		For simplicity we denote $$A^\bullet_{X,(2)}:=A^\bullet(X^o,\bV)\cap \sD^\bullet_{X,\bV;ds^2,h}$$
		and $$A^\bullet_{(0,1)\times X,(2)}:=A^\bullet((0,1)\times X^o,\bV)\cap \sD^\bullet_{(0,1)\times X,q^\ast\bV;ds^2_{(0,1)\times X},q^\ast h}((0,1)\times X).$$
		Using the decomposition
		\begin{align*}
		\alpha=\alpha_0+d\rho\wedge\alpha_1,\quad \iota_{\frac{\partial}{\partial \rho}}\alpha_0=0,~~\iota_{\frac{\partial}{\partial \rho}}\alpha_1=0,
		\end{align*}
		we define
		\begin{align*}
		Q(\alpha)=2\int_{\frac{1}{4}}^{{\frac{3}{4}}}\alpha_0(t,x)dt
		\end{align*}
		and
		\begin{align*}
		H(\alpha)=2\int_{\frac{1}{4}}^{{\frac{3}{4}}}\int_{t}^\rho\alpha_1(\rho',x)d\rho'dt.
		\end{align*}
		It follows from the definition that $Qq^\ast={\rm Id}$. We are going to show that $q^\ast Q$ is homotopic to the identity. For this we need to show the following two claims.
		\begin{description}
			\item[A]
			$\nabla H+H\nabla={\rm Id}-q^\ast Q$. Recall that
			$$\nabla=d\rho\wedge\frac{\partial}{\partial\rho}+\nabla_{\bV}$$
			where $\nabla_{\bV}$ is the flat connection associated to $\bV$ and $\nabla=q^\ast\nabla_{\bV}$.
			Claim {\bf A} follows from the equalities (\ref{align_dH2}) and (\ref{align_Hd2}) below:
			\begin{align}\label{align_dH2}
			\nabla H(\alpha)&=2\int_{\frac{1}{4}}^{{\frac{3}{4}}}\left(d\rho\wedge\frac{\partial}{\partial \rho}\int_{t}^\rho\alpha_1(\rho',x)d\rho'+\int_{t}^\rho \nabla_{\bV} \alpha_1(\rho',x)d\rho'\right)dt\\\nonumber
			&=d\rho\wedge\alpha_1(\rho,x)+2\int_{\frac{1}{4}}^{{\frac{3}{4}}}\int_{t}^\rho \nabla_{\bV}\alpha_1(\rho',x)d\rho'dt.\nonumber
			\end{align}
			\begin{align}\label{align_Hd2}
			H\nabla(\alpha)&=2\int_{\frac{1}{4}}^{{\frac{3}{4}}}\int_{t}^\rho\left(\frac{\partial}{\partial \rho'}\alpha_0(\rho',x)-\nabla_{\bV}\alpha_1(\rho',x)\right)d\rho'dt\\\nonumber
			&=\alpha_0(\rho,x)-2\int_{\frac{1}{4}}^{\frac{3}{4}}\alpha_0(t,x)dt-2\int_{\frac{1}{4}}^{{\frac{3}{4}}}\int_{t}^\rho \nabla_{\bV}\alpha_1(\rho',x)d\rho'dt.\nonumber
			\end{align}
			\item[B] $Q$, $H$ send locally $L^2$ forms to locally $L^2$ forms. This follows from the following estimates (\ref{align_estimate_Q_02}) and (\ref{align_estimate_Q_H_02}).
			\begin{align}\label{align_estimate_Q_02}
			\int_{ M}\left|2\int_{\frac{1}{4}}^{\frac{3}{4}}\alpha_0(t,x)dt\right|^2{\rm vol}_{M}
			\leq 2\int_{ M}\int_{\frac{1}{4}}^{\frac{3}{4}}|\alpha_0(t,x)|^2dt{\rm vol}_{M}
			\leq 2\|\alpha\|^2_{[\frac{1}{4},\frac{3}{4}]\times M}.
			\end{align}
			\begin{align}\label{align_estimate_Q_H_02}
			&\int_{[a,b]\times  M}\left|2\int_{\frac{1}{4}}^{\frac{3}{4}}\int_{t}^\rho\alpha_1(\rho',x)d\rho'dt\right|^2{\rm vol}_{[a,b]\times M}\\\nonumber
			\leq&2\int_{[a,b]\times  M}\int_{\frac{1}{4}}^{\frac{3}{4}}\int_{a}^{b}\left|\alpha_1(\rho',x)\right|^2d\rho'dt{\rm vol}_{[a,b]\times  M}\\\nonumber
			=&\int_{ M}\int_{a}^{b}\left(\int_{a}^{b}\left|\alpha_1(\rho',x)\right|^2d\rho'\right)d\rho{\rm vol}_{M}\\\nonumber
			\leq&\int_{M}\left(\int_{a}^{b}|\alpha_1(\rho,x)|^2d\rho \right){\rm vol}_{M}\\\nonumber
			\leq&\int_{M}\left(\int_{a}^{b}|\alpha(\rho,x)|^2d\rho \right){\rm vol}_{M}\\\nonumber
			=&\|\alpha\|^2_{[a,b]\times  M}.
			\end{align}
		\end{description}
		Here $M\subset X$ is an arbitrary compact subset and $0<a<\frac{1}{4}$, $\frac{3}{4}<b<1$.
		As a consequence, $Q:A^\bullet_{(0,1)\times X,(2)}\to A^\bullet_{X,(2)}$ is well defined. Also $H:A^\bullet_{(0,1)\times X,(2)}\to A^{\bullet-1}_{(0,1)\times X,(2)}$ is well defined and gives a homotopy of $q^\ast Q$ to the identity. On the other hand, $Qq^\ast={\rm Id}$. Hence $q^\ast$ is a homotopy equivalence. Since both $Q$ and $H$ are continuous operators (claim {\bf B}), $\nabla H+H\nabla={\rm Id}-q^\ast Q$ extends to a homotopy from
		$$q^\ast Q:\sD^\bullet_{(0,1)\times X,q^\ast\bV;ds^2_{(0,1)\times X},q^\ast h}((0,1)\times X)\to\sD^\bullet_{(0,1)\times X,q^\ast\bV;ds^2_{(0,1)\times X},q^\ast h}((0,1)\times X)$$
		to the identity. This proves the lemma.
	\end{proof}
	For compact supported smooth forms, there is a pushforward map via integral along the fibers of $q$. For each locally $L^2$ integrable $q^\ast\bV$-valued smooth form $\alpha$ on $(0,1)\times X^o$ such that ${\rm supp}\alpha$ is precompact, there is a decomposition
	\begin{align*}
	\alpha=\alpha_0+d\rho\wedge\alpha_1,\quad \iota_{\frac{\partial}{\partial \rho}}\alpha_0=0,~~\iota_{\frac{\partial}{\partial \rho}}\alpha_1=0.
	\end{align*}
	Define
	$$q_\ast(\alpha)=\int_0^1 dt\wedge\alpha_1.$$
	Because
	\begin{align*}
	\int_{ M}\left|\int_0^1dt\wedge\alpha_1\right|^2{\rm vol}_{M}
	\leq\int_{ M}\int_{0}^{1}|\alpha_1|^2dt{\rm vol}_{M}
	\lesssim\|\alpha\|^2_{(0,1)\times M}
	\end{align*}
	where $M\subset X$ is any precompact subset, $q_\ast$ is a continuous operator. On the other hand
	\begin{align*}
	q_\ast \nabla(\alpha)&= q_\ast(\nabla_{\bV}\alpha_0+dt\wedge(\frac{\partial\alpha_0}{\partial t}-\nabla_{\bV}\alpha_1))\\\nonumber
	&= \int^1_0(\frac{\partial\alpha_0}{\partial t}-\nabla_{\bV}\alpha_1)dt\\\nonumber
	&= -\int^1_0dt\wedge \nabla_\bV\alpha_1 \quad\left(\int^1_0dt\wedge\frac{\partial\alpha_0}{\partial t}=\alpha_0(1,x)-\alpha_0(0,x)=0\right)\\\nonumber
	&= \nabla_\bV\int^1_0 dt\wedge\alpha_1=\nabla q_\ast(\alpha).
	\end{align*}
	Therefore $q_\ast$ extends uniquely to a continuous morphism
	\begin{align}\label{align_L2_Poincare_q_*}
	q_\ast: \Gamma_{\rm cpt}\left((0,1)\times X, \sD^\bullet_{(0,1)\times X,q^\ast\bV;ds^2_{(0,1)\times X},q^\ast h}\right)\to \Gamma_{\rm cpt}\left( X,\sD^{\bullet-1}_{X,\bV;ds^2,h}\right).
	\end{align}
	\begin{lem}\label{lem_Thom_locL2}
		Notations as above, the pushforward map (\ref{align_L2_Poincare_q_*})
		is a homotopy equivalence.
	\end{lem}
	\begin{proof}
		For simplicity we denote
		$$A^\bullet_{X,{\rm cpt},(2)}:=\Gamma_{\rm cpt}\left( X,\sD^{\bullet}_{X,\bV;ds^2,h}\right)\cap A^\bullet(X^o,\bV)$$ and $A^\bullet_{(0,1)\times X,{\rm cpt},(2)}$ similarly.

		Fix a compact supported smooth 1-form $e(\rho)=f(\rho)d\rho$ on $(0,1)$ where $f\geq0$  with its total integral 1.
		Define
		\begin{align*}
		e_\ast: A^\bullet_{X,{\rm cpt},(2)}\to A^\bullet_{(0,1)\times X,{\rm cpt},(2)},\quad e_\ast(\alpha)=e\wedge\alpha,
		\end{align*}
		and
		\begin{align*}
		H(\alpha)=\int^\rho_0\alpha_1dt-\int^\rho_0e\int^1_0\alpha_1dt.
		\end{align*}
		It follows directly from the definition that $H\alpha$ has compact support and $q_\ast e_\ast={\rm Id}$. We are going to show that $e_\ast q_\ast$ is homotopic to the identity. For this we need to show the following two claims.
		\begin{description}
			\item[A]
			$\nabla H+H\nabla={\rm Id}-e_\ast q_\ast$. This follows from the equalities (\ref{align_dH3}) and (\ref{align_Hd3}) below:
			\begin{align}\label{align_dH3}
			\nabla H(\alpha)&=\nabla\left(\int^\rho_0\alpha_1dt-\int^\rho_0e\int^1_0\alpha_1dt\right)\\\nonumber
			&=d\rho\wedge\alpha_1+\int_0^\rho\nabla_\bV\alpha_1dt-e\wedge\int_0^1\alpha_1dt-\int_0^\rho e\int_0^1\nabla_\bV\alpha_1dt.
			\end{align}
			\begin{align}\label{align_Hd3}
			H\nabla(\alpha)&=\int_0^\rho\left(\frac{\partial}{\partial t}\alpha_0-\nabla_\bV\alpha_1\right)dt-\int_0^\rho e\int_0^1\left(\frac{\partial}{\partial t}\alpha_0-\nabla_\bV\alpha_1\right)dt\\\nonumber
			&=\alpha_0-\int_0^\rho\nabla_\bV\alpha_1dt+\int_0^\rho e\int_0^1\nabla_\bV\alpha_1 dt.
			\end{align}
			\item[B] $H$ sends compact supported $L^2$ forms to compact supported $L^2$ forms.
			This follows from the estimate
			\begin{align}\label{align_estimate_Q_H_03}
			&\int_{(0,1)\times  X}\left|\int^t_0d\rho\wedge\alpha_1-\int^t_0e\int^1_0d\rho\wedge\alpha_1\right|^2{\rm vol}_{(0,1)\times X}\\\nonumber
			\leq&2\int_{(0,1)\times X}\left|\int^t_0d\rho\wedge\alpha_1\right|^2{\rm vol}_{(0,1)\times X}+2\int_{(0,1)\times X}\left|\int^t_0e\int^1_0d\rho\wedge\alpha_1\right|^2{\rm vol}_{(0,1)\times X}\\\nonumber
			\leq&2\int_{(0,1)\times X}\int^1_0\left|\alpha_1\right|^2d\rho{\rm vol}_{(0,1)\times X}+2\int_{(0,1)\times X}\left|\int^1_0e\right|^2\int^1_0\left|\alpha_1\right|^2d\rho{\rm vol}_{(0,1)\times X}\\\nonumber
			=&4\int_{X}\int_{0}^1|dt\wedge\alpha_1|^2dt{\rm vol}_{X}\\\nonumber
			\leq&4\|\alpha\|^2_{(0,1)\times X}<\infty.
			\end{align}
		\end{description}
		As a consequence, $$H:A^\bullet_{(0,1)\times X,{\rm cpt},(2)}\to A^{\bullet-1}_{(0,1)\times X,{\rm cpt},(2)}$$ is well defined and gives a homotopy from $e_\ast q_\ast$ to the identity. Since both $e_\ast$ and $H$ are continuous operators (claim {\bf B}), $\nabla H+H\nabla={\rm Id}-e_\ast q_\ast$ extends to a homotopy from
		$$e_\ast q_\ast:\Gamma_{\rm cpt}\left((0,1)\times X, \sD^\bullet_{(0,1)\times X,q^\ast\bV;ds^2_{(0,1)\times X},q^\ast h}\right)\to \Gamma_{\rm cpt}\left((0,1)\times X, \sD^\bullet_{(0,1)\times X,q^\ast\bV;ds^2_{(0,1)\times X},q^\ast h}\right)$$
		to the identity. This proves the lemma.
	\end{proof}
	\subsection{A formal criteria}
	In this subsection we give a sufficient condition for the metrics $ds^2$ and $h$ to let the cohomology sheaves of $\sD^\bullet_{X,\bV;ds^2,h}$ be weakly constructible.
	\begin{defn}\label{defn_diffeomorphic_tubular_neigh}
		Let $X$ be a complex space of pure dimension $n$ and $X^o\subset X_{\rm reg}$ a dense Zariski open subset.
		A stratification $\{X_i\}$ of $(X,X^o)$ is called a {\bf bimeromorphic fibering stratification} if for every stratum $S$ and every point $x\in S$, there is an analytic subspace $C_x\subset X$ passing through $x$, a  neighborhood $x\in U_x$ in $X$, a neighborhood $x\in U_{S,x}$ in $S$ and a proper bimeromorphic morphism $\pi_x:\widetilde{U_x}\to U_x$ such that
		\begin{enumerate}
			\item $\pi_x$ is biholomorphic over $U_x\cap X^o$.
			\item There is a neighborhood $\pi^{-1}_x(U_{S,x})\subset V\subset \widetilde{U_x}$ and a $C^\infty$ diffeomorphism
			$$\alpha_x: V\simeq (V\cap\pi^{-1}_xC_x)\times U_{S,x}$$
			such that
			\begin{itemize}
				\item $\alpha_x(y)=(y,x)$ for any $y\in V\cap\pi^{-1}_xC_x$,
				\item $p_2\circ\alpha_x|_{\pi^{-1}_xU_{S,x}}=\pi_x|_{\pi^{-1}_xU_{S,x}}$ where $p_2:(V\cap\pi^{-1}_xC_x)\times U_{S,x}\to U_{S,x}$ is the projection.
				\item $\alpha_x(V\cap\pi^{-1}_xX^o)=(V\cap\pi^{-1}_x(C_x\cap X^o))\times U_{S,x}$.
			\end{itemize}
		\end{enumerate}
		The six tuple $(U_x,U_{S,x},C_x,V,\pi_x,\alpha_x)$ is called a bimeromorphic fibering neighborhood of $x\in S$.
	\end{defn}
    In this paper we mainly use bimeromorphic fibering stratifications and bimeromorphic fibering neighborhoods produced by a Whitney stratified desingularization (Definition \ref{defn_bimero_fib_whitneystr}).
	\begin{defn}\label{defn_metric_str}
		Let $X$ be a complex space and $X^o\subset X_{\rm reg}$ a dense Zariski open subset with a hermitian metric $ds^2$ on $X^o$. Let $(\bV,h)$ be a hermitian local system on $X^o$. Let $\{X_i\}$ be a bimeromorphic fibering stratification of $(X,X^o)$. We say that $(ds^2,h)$ is bimeromorphically stratified along $\{X_i\}$ if for every stratum $S$ and every point $x\in S$, there is a bimeromorphic fibering neighborhood $(U_x,U_{S,x},C_x,V,\pi_x,\alpha_x)$ of $x\in S$ such that the following statements hold.
		\begin{enumerate}
			\item $\pi^\ast_xds^2|_{V}\sim \pi^\ast_xds^2|_{\pi^{-1}_xC_x\cap V}+ds^2_{S}$ via $\alpha_x$ for some hermitian metric $ds^2_{S}$ on $S$.
			\item There is an isomorphism of the local systems\footnote{This trivialization exists when $U_{S_x}$ is contractible.} $$q^\ast(\bV|_{\pi^{-1}_x(C_x\cap X^o)\cap V})\simeq \bV|_{V\cap\pi^{-1}_x X^o}$$
			where $q:=p_1\circ\alpha_x:V\simeq (\pi^{-1}_xC_x\cap V)\times U_{S,x}\to (\pi^{-1}_xC_x\cap V)$ is the projection. Moreover $$h|_{V\cap\pi^{-1}_x X^o}\sim q^\ast(h|_{\pi^{-1}_x(C_x\cap X^o)\cap V}).$$
		\end{enumerate}
		In this case we say that $(ds^2,h)$ is stratified along $(U_x,U_{S,x},C_x,V,\pi_x,\alpha_x)$.
	\end{defn}
	\begin{prop}\label{prop_Dstr}
		Let $X$ be a complex space of pure dimension $n$ and $X^o\subset X_{\rm reg}$ a dense Zariski open subset  with a hermitian metric $ds^2$ on $X^o$. Let $(\bV,h)$ be a hermitian local system on $X^o$. Assume that $(ds^2, h)$ is bimeromorphically stratified along a bimeromorphic fibering stratification $\{X_i\}$ of $(X,X^o)$, then the cohomology sheaves of $\sD^\bullet_{X,\bV;ds^2,h}$ are weakly constructible with respect to $\{X_i\}$.

		Assume moreover that $\sD^\bullet_{X,\bV;ds^2,h}$ is a complex of fine sheaves. Let $(U_x,U_{S,x},C_x,V,\pi_x,\alpha_x)$ be a bimeromorphic fibering neighborhood of $x\in S$ for some stratum $S$ along which $(ds^2,h)$ is bimeromorphically stratified. For every $k$ there are canonical isomorphisms
		\begin{align}\label{align_prop_Dstr_stalk}
		H^k\left(i^\ast_x\sD^\bullet_{X,\bV;ds^2,h}\right)\simeq \varinjlim_{x\in U}H^k\left(\Gamma(U,\sD^\bullet_{C_x,\bV|_{C_x};ds^2|_{C_x},h|_{C_x}})\right)
		\end{align}
		and
		\begin{align}\label{align_prop_Dstr_costalk}
		H^k\left(i^!_x\sD^\bullet_{X,\bV;ds^2,h}\right)\simeq\varinjlim_{x\in U}H^{k-2\dim S}\left(\Gamma_{\rm cpt}(U,\sD^\bullet_{C_x,\bV|_{C_x};ds^2|_{C_x},h|_{C_x}})\right)^\ast.
		\end{align}
		Here $i_x:\{x\}\to X$ is the inclusion and $U$ ranges over the open neighborhoods of $x$ in $C_x$.
	\end{prop}
	\begin{proof}
		To show that $h^k(\sD^\bullet_{X,\bV;ds^2,h})|_{X_d\backslash X_{d-1}}$ is a weakly local system for every $k$ and every $d$, it is equivalent to show that: for every point $x\in X_d\backslash X_{d-1}$ and every precompact neighborhood $x\in W\subset X_d\backslash X_{d-1}$ such that $W$ is diffeomorphic to $(-1,1)^{2d}$, the natural maps
		\begin{align*}
		\Gamma(\overline{W_\epsilon}, h^k(\sD^\bullet_{X,\bV;ds^2,h}))\to H^k(\sD^\bullet_{X,\bV;ds^2,h})_x,\quad\forall\epsilon\in(0,1)
		\end{align*}
		are isomorphisms. Here we assume that $x=0\in (-1,1)^{2d}$ and denote $W_\epsilon=(-\epsilon,\epsilon)^{2d}$. Equivalently, we shall show that
		\begin{align}\label{align_prop_Dstr_1}
		\varinjlim_{\overline{W_\epsilon}\subset W'}H^k(\Gamma(W',\sD^\bullet_{X,\bV;ds^2,h}))\to H^k(\sD^\bullet_{X,\bV;ds^2,h})_x, \quad\forall\epsilon\in(0,1)
		\end{align}
		are isomorphisms where $W'$ ranges over open neighborhoods of $\overline{W_\epsilon}$ in $X$.

		Let $(U_x,U_{S,x},C_x,V,\pi_x,\alpha_x:\widetilde{U_x}\to U_x)$ be a bimeromorphic fibering neighborhood along which $(ds^2,h)$ is stratified, i.e. (1)-(2) of Definition \ref{defn_metric_str} hold. Without loss of generality we assume that $\overline{W}\subset U_{S,x}$. Since $\pi_x$ is a proper map and $\pi_x={\rm Id}$ over $U_x\cap X^o$, we obtain that
		\begin{align*}
		H^k(\sD^\bullet_{X,\bV;ds^2,h})_x&\simeq \varinjlim_{x\in W'\subset U_x}H^k\left(\Gamma(W',\sD^\bullet_{X,\bV;ds^2,h})\right)\\\nonumber
		&\simeq\varinjlim_{\pi^{-1}\{x\}\subset \widetilde{W}'\subset \widetilde{U_x}}H^k\left(\Gamma(\widetilde{W}',\sD^\bullet_{\widetilde{U_x},\bV;ds^2,h})\right)\\\nonumber
		&\simeq\varinjlim_{\substack{\epsilon\to 0\\x\in W_x\subset C_x}}H^k\left(\Gamma(\alpha^{-1}_x( \pi^{-1}_xW_x\times W_\epsilon),\sD^\bullet_{\widetilde{U_x},\bV;ds^2,h})\right)\\\nonumber
		\end{align*}
		and
		\begin{align*}
		\varinjlim_{\overline{W_\epsilon}\subset W'}H^k(\Gamma(W',\sD^\bullet_{X,\bV;ds^2,h}))&\simeq \varinjlim_{\pi^{-1}(\overline{W_\epsilon})\subset \widetilde{W}'\subset \widetilde{U_x}}H^k\left(\Gamma(\widetilde{W}',\sD^\bullet_{\widetilde{U_x},\bV;ds^2,h})\right)\\\nonumber
		&\simeq\varinjlim_{\substack{\epsilon<\epsilon'<1\\x\in W_x\subset C_x}}H^k\left(\Gamma(\alpha^{-1}_x(\pi^{-1}_xW_x\times W_{\epsilon'} ),\sD^\bullet_{\widetilde{U_x},\bV;ds^2,h})\right).\\\nonumber
		\end{align*}
		As a consequence,
		(\ref{align_prop_Dstr_1}) is isomorphic to
		\begin{align*}
		\varinjlim_{\substack{\epsilon<\epsilon'<1\\x\in W_x\subset C_x}}H^k\left(\Gamma(\alpha^{-1}_x(\pi^{-1}_xW_x\times W_{\epsilon'} ),\sD^\bullet_{\widetilde{U_x},\bV;ds^2,h})\right)\to \varinjlim_{\substack{\epsilon'\to 0\\x\in W_x\subset C_x}}H^k\left(\Gamma(\alpha^{-1}_x( \pi^{-1}_xW_x\times W_{\epsilon'}),\sD^\bullet_{\widetilde{U_x},\bV;ds^2,h})\right)
		\end{align*}
		for every $\epsilon\in(0,1)$,
		where $W_x$ runs over open neighborhoods of $x\in C_x$.
		By Lemma \ref{lem_Kunneth_locL2}, both sides are isomorphic to
		\begin{align*}
		\varinjlim_{x\in W_x\subset C_x}H^k\left(\Gamma(W_x,\sD^\bullet_{C_x,\bV|_{C_x};ds^2|_{C_x},h|_{C_x}})\right).
		\end{align*}
		This proves that (\ref{align_prop_Dstr_1}) is an isomorphism.
		This argument also deduces (\ref{align_prop_Dstr_stalk}). If moreover $\sD^\bullet_{X,\bV;ds^2,h}$ is a complex of fine sheaves, (\ref{align_prop_Dstr_costalk}) follows from the calculation
		\begin{align*}
		h^k\left(i^!_x\sD^\bullet_{X,\bV;ds^2,h}\right)&\simeq\varinjlim_{x\in W'}\bH^{k}_{\rm cpt}(W',\sD^\bullet_{X,\bV;ds^2,h})^\ast\\\nonumber
		&\simeq\varinjlim_{x\in W'}H^{k}\left(\Gamma_{\rm cpt}(W',\sD^\bullet_{X,\bV;ds^2,h})\right)^\ast\quad(\textrm{fineness})\\\nonumber
		&\simeq\varinjlim_{\pi^{-1}\{x\}\subset \widetilde{W}'\subset \widetilde{U_x}}H^{k}\left(\Gamma_{\rm cpt}(\widetilde{W}',\sD^\bullet_{\widetilde{U_x},\bV;ds^2,h})\right)^\ast\\\nonumber
		&\simeq\varinjlim_{\substack{\epsilon>0\\x\in W_x\subset C_x}}H^{k}\left(\Gamma_{\rm cpt}(\alpha_x^{-1}( \pi_x^{-1}W_x\times W_\epsilon),\sD^\bullet_{\widetilde{U_x},\bV;ds^2,h})\right)^\ast\\\nonumber
		&\simeq\varinjlim_{x\in W_x\subset C_x}H^{k-2\dim S}\left(\Gamma_{\rm cpt}(W_x,\sD^\bullet_{C_x,\bV|_{C_x};ds^2|_{C_x},h|_{C_x}})\right)^\ast\quad \textrm{(Lemma \ref{lem_Thom_locL2})}.
		\end{align*}
		This proves Proposition \ref{prop_Dstr}.
	\end{proof}
    \subsection{Whitney stratified desingularization}
    A desingularization of a complex space provides local coordinates to study the asymptotic behavior of a given metric and to calculate its local $L^2$-cohomologies (see \cite{Hsiang_Pati1985,Saper1992}). However
    it is difficult to study the $L^2$-representation of the intersection complex via a general desingularization because it does not witnesses the stratification.
    The aim of this subsection is to combine the desingularization and the stratification of a complex space.
    \begin{defn}\label{defn_Lstr_desingularization}
    	Let $X$ be a complex space of dimension $n$ and $X^o\subset X_{\rm reg}$  a dense Zariski open subset. A Whitney stratified desingularization of $(X,X^o)$ consists of the following data.
    	\begin{enumerate}
    		\item A projective morphism $\pi:\widetilde{X}\to X$ such that
    		\begin{enumerate}
    			\item $\tilde{X}$ is smooth,
    			\item $\pi|_{\pi^{-1}X^o}:\pi^{-1}X^o\to X^o$ is biholomorphic.
    		\end{enumerate}
    		\item A Whitney stratification $\{X_i\}_{0\leq i\leq n}$ of $(X,X^o)$ such that
    		\begin{enumerate}
    			\item For every $i<n$, $\pi^{-1}X_i$ is a simple normal crossing divisor of $\widetilde{X}$.
    			\item Denote by $(\pi^{-1}X_{n-1})_{\rm red}=:E=\cup_{i\in I} E_i$ the irreducible decomposition of the (reduced) exceptional divisor of $\pi$. For every stratum $S$ of $\{X_i\}$ and every stratum $E_J:=\bigcap_{j\in J}E_j$, $\emptyset\neq J\subset I$, either $S\cap \pi(E_J)=\emptyset$ or $\pi: \pi^{-1}S\cap E_J\to S$ is a submersion.
    		\end{enumerate}
    	\end{enumerate}
    \end{defn}
    \begin{thm}\label{thm_Lstr_desingularization}
    	Let $X$ be a complex space and $X^o\subset X_{\rm reg}$ a dense Zariski open subset. Let $Y\subset X$ be a compact analytic subspace. Then there is a neighborhood $U$ of $Y$ and a Whitney stratified desingularization of $(U,U^o=X^o\cap U)$. In particular, Whitney stratified desingularization exists in the following cases.
    	\begin{itemize}
    		\item $X$ is a compact complex space.
    		\item $\overline{X}$ is a compact complex space and $\overline{X}^o\subset \overline{X}_{\rm reg}$ is a dense Zariski open subset. $X\subset \overline{X}$ is a Zariski open subset and $X^o=\overline{X}^o\cap X$. For example $(X,X^o)$ is a pair of complex algebraic varieties.
    		\item $X$ is a germ of complex space.
    	\end{itemize}
    \end{thm}
    \begin{proof}
    	The proof is done by induction on the depth. Let $n=\dim X$.

    	{\bf Initial step:}
    	Let $U$ be some neighborhood of $Y$ so that there is a projective map $\pi_{1}:U^{1}\to U$ satisfying the following statements.
    	\begin{itemize}
    		\item $U^1$ is smooth, $\pi_1:\pi_1^{-1}U^o\to U^o$ is biholomorphic.
    		\item $E^{(1)}:=\pi_1^{-1}(U\backslash U^o)$ is a simple normal crossing divisor.
    	\end{itemize}
        Take an arbitrary Whitney stratification $\{U_i^{(1)}\}$ of $(U,U^o)$ (c.f. \cite[Theorem 2.2]{Verdier1976}).
    	Condition (2) in Definition \ref{defn_Lstr_desingularization} holds for $\pi_{1}:\pi_{1}^{-1}(U^o)\to U^o$.

    	{\bf Induction step:} Assume that there is a map $\pi_{k}:U^{k}\to U$ and a Whitney stratification $\{U_i^{(k)}\}$ such that condition (1) holds for $\pi_k$ and condition (2) holds over $U\backslash U^{(k)}_{n-k}$.

    	After a possible shrinking of $U$, we may take a projective map $\rho:U^{k+1}\to U^k$ (an embedding resolution of $\pi^{-1}_k(U^{(k)}_{n-k})\subset U^k$) so that
    	\begin{itemize}
    		\item $U^{k+1}$ is smooth and $\rho:\rho^{-1}\pi_k^{-1}(U\backslash U^{(k)}_{n-k})\to \pi_k^{-1}(U\backslash U^{(k)}_{n-k})$ is biholomorphic.
    		\item Denote $\pi_{k+1}=\rho\pi_k$, then $E':=\pi^{-1}_{k+1}(U^{(k)}_{n-k})$ and $\pi^{-1}_{k+1}(U\backslash U^o)$ are simple normal crossing divisors.
    	\end{itemize}
    	Denote by $E'=\bigcup_{j\in I'}E'_j$  the irreducible decomposition and denote by $E'_J=\bigcup_{j\in J}E'_j$ for every $\emptyset\neq J\subset I'$. By generic smoothness (Bertini-Sard theorem) there is a smooth dense Zariski open subset $S\subset U_{n-k}^{(k)}\backslash U_{n-k-1}^{(k)}$ so that the following claims hold.
    	\begin{itemize}
    		\item For every stratum $E'_J$, $\emptyset\neq J\subset I'$, either $S\cap \pi_{k+1}(E'_J)=\emptyset$ or $\pi_{k+1}: \pi^{-1}_{k+1}(S)\cap E'_J\to S$ is a submersion.
    		\item By \cite[Theorem 2.2]{Verdier1976}, after removing a nowhere dense Zariski closed subspace of $S$, we assume that $(S,S')$ satisfies Verdier's condition (w) for every stratum $S'$ of $\{U_i^{(k)}\}$ such that $S\subset \overline{S'}$. In particular, such $(S,S')$ satisfies Whitney's conditions (a) and (b).
    	\end{itemize}
        By \cite[Theorem 2.2]{Verdier1976}, there is a Whitney stratification $\{U^{(k+1)}_i\}$ of $(U,U^o)$ so that
        \begin{itemize}
        	\item $\{U^{(k+1)}_i\}$ refines $\{U^{(k)}_i\}$,
        	\item For every $i\geq n-k$, $U^{(k+1)}_i=U^{(k)}_i$.
        	\item $S=U^{(k+1)}_{n-k}\backslash U_{n-k-1}^{(k+1)}$.
        \end{itemize}
    	By the constructions, $(\{U^{(k+1)}_i\},\pi_{k+1}:U^{k+1}\to U)$ satisfies condition (1) in Definition \ref{defn_Lstr_desingularization} and condition (2) holds over $U\backslash U^{(k+1)}_{n-k-1}$. Since $\dim (U\backslash U^{(k+1)}_{n-k-1})<\dim (U\backslash U^{(k)}_{n-k})$, this procedure stops after finite steps. We prove the theorem.
    \end{proof}
    \begin{lem}\label{lem_restrict_Whitney_desingularization}
    	Let $X$ be a complex space and $X^o\subset X_{\rm reg}$  a dense Zariski open subset. Let $(\{X_i\},\pi:\widetilde{X}\to X)$ be a Whitney stratified desingularization of $(X,X^o)$. Let $Z\subset X$ be an analytic subspace such that $Z\cap X^o\neq\emptyset$ and $Z$ intersects transversally with every stratum of $\{X_i\}$. Then $(\{X_i\cap Z\},\pi|_{\pi^{-1}Z}:\pi^{-1}Z\to Z)$ is a Whitney stratified desingularization of $(Z,X^o\cap Z)$.
    \end{lem}
    \begin{proof}
    	Since $Z$ intersects transversally with every stratum of $\{X_i\}$, $\{X_i\cap Z\}$ is a Whitney stratification of $Z$. It remains to show that $\pi^{-1}Z$ is smooth and $\pi^{-1}(X_i\cap Z)\subset \pi^{-1}Z$ is a simple normal crossing divisor for every $i$ such that $X_i\cap Z\neq\emptyset$. Denote by $E=\cup E_i$ the underlying reduced subspace of the exceptional divisor of $\pi$. By (2-b) in Definition \ref{defn_Lstr_desingularization}, the (schematic) intersection $E_J\cap \pi^{-1}Z$ is either empty or smooth for every stratum $E_J$ of $E$. Recall the following fact: for a point $x\in Y$ of a complex space and $D\subset Y$ a Cartier divisor passing through $x$, if $D$ is smooth at $x$, then $Y$ is smooth at $x$. As a consequence $\pi^{-1}Z$ is smooth and $\pi^{-1}(X_i\cap Z)\subset \pi^{-1}Z$ is either empty or a simple normal crossing divisor of $\pi^{-1}Z$. Condition (2-b) follows directly from the construction.
    \end{proof}
    \begin{prop}\label{prop_bimero_fibering_neigh}
    	Let $X$ be complex space and $X^o\subset X_{\rm reg}$ a dense Zariski open subset. Let $(\{X_i\},\pi:\widetilde{X}\to X)$ be a Whitney stratified desingularization of $(X,X^o)$. Then $\{X_i\}$ is a bimeromorphic fibering stratification of $(X,X^o)$.

    	To be precise, let $S$ be a stratum of $\{X_i\}$ and $x\in S$. Let $C_x\subset X$ be a $\dim X-\dim S$ dimensional germ of analytic subspace passing through $x$ and intersecting transversally with every stratum of $\{X_i\}$ ($C_x$ could be the cut off of $X$ by $\dim S$ hypersurfaces in general positions). There is a neighborhood $V$ of $\pi^{-1}\{x\}$ in $\widetilde{X}$, a neighborhood $x\in U_{S,x}$ in $S$ and a $C^\infty$ diffeomorphism
    	\begin{align}\label{align_alpha_x}
    	\alpha_x: V\simeq (V\cap\pi^{-1}C_x)\times U_{S,x}
    	\end{align}
    	such that
    	\begin{enumerate}
    		\item $\alpha_x(y)=(y,x)$ for any $y\in V\cap\pi^{-1}C_x$,
    		\item $p_2\circ\alpha_x|_{\pi^{-1}(U_{S,x})}=\pi|_{\pi^{-1}(U_{S,x})}$ where $p_2:(V\cap\pi^{-1}C_x)\times U_{S,x}\to U_{S,x}$ is the projection.
    		\item $\alpha_x(V\cap\pi^{-1}X^o)=(V\cap\pi^{-1}(C_x\cap X^o))\times U_{S,x}$.
    		\item For every stratum $E_J$ ($\emptyset\neq J\subset I$) of the exceptional divisor $E=\bigcup_{i\in I}E_i$ of $\pi$, $$\alpha_x(V\cap E_J)=(V\cap\pi^{-1}C_x\cap E_J)\times U_{S,x}.$$
    	\end{enumerate}
    \end{prop}
    \begin{proof}
    	Assume that $S\subset X_r\backslash X_{r-1}$, i.e. $\dim S=r$.
    	By the definition of Whitney stratified desingularization, $\pi^{-1}S$ is a simple normal crossing divisor of the complex manifold $Y:=\pi^{-1}(X_{r-1}\cup S)\subset\widetilde{X}$. Let $E=\bigcup_{i\in I}E_i$ be the reduced exceptional divisor of $\pi$. Let $I_S\subset I$ be the subset of indexes $i\in I$ so that $E_i\cap Y\subset \pi^{-1}S$ and $I'_S\subset I$ is the subset of indexes $i\in I$ so that $E_i\cap\pi^{-1}S\neq\emptyset$. Then for every $J\subset I_S$, $\pi:E_J\to S$ is a proper submersion. Let $U_{S,x}\subset S$ be an open neighborhood of $x$ so that its holomorphic tangent bundle admits a frame $w_1,\dots, w_{r}$. To get the data $(V,\alpha_x)$ claimed in the theorem, it suffices to construct $C^\infty$ complex vector fields $v_1,\dots, v_{r}$ on a neighborhood of $\pi^{-1}\{x\}\subset Y$ such that the following statements hold.
    	\begin{enumerate}
    		\item For every $i=1,\dots,r$ and every $J\subset I_S$, $v_i|_{E_J}\in T_{E_J}$ and lifts $w_i$ through $\pi$, i.e. for every $x\in E_J$, $\pi_\ast(v_i(x))=w_i(\pi(x))$.
    		\item The restrictions of $v_1,\dots, v_{r}$ on $V\cap\pi^{-1}C_x$ form a $C^\infty$ frame of the holomorphic normal bundle of $V\cap\pi^{-1}C_x\subset \widetilde{X}$.
    		\item For every $i=1,\dots, r$ and every $J\subset I$, $v_i|_{E_J}\in T_{E_J}$.
    	\end{enumerate}
    	Then the vector fields $v_i$ and $\overline{v_i}$ give the one-parameter groups $h_i,\overline{h}_i:(-\epsilon,\epsilon)\times W\to \widetilde{X}$ where $\epsilon>0$ is a constant independent of $i$ and $\pi^{-1}\{x\}\subset W\subset\pi^{-1}C_x$ is an open subset. By the conditions (1)-(3) of the vector fields $v_1,\dots, v_r$, the map
    	$$\alpha_x^{-1}:(-\epsilon,\epsilon)^{2r}\times W\to \widetilde{X},$$
    	$$\alpha_x^{-1}(t_1,\dots,t_{2r},x)=h_1(t_1,\overline{h}_1(t_2,h_2(t_3,\overline{h}_2(\dots h_r(t_{2r-1},\overline{h}_r(t_{2r},x))\dots))))$$
    	is well defined and its inverse is the fibering $\alpha_x$ that  we need for the theorem.

    	It remains to construct the vector fields $v_1,\dots, v_r$. This can be  done by three steps.
    	\begin{description}
    		\item[Step 1] Notice that $\pi^{-1}S\to S$ is a stratified submersion, i.e. every stratum is mapped submersively onto $S$. One may lift $w_1,\dots, w_r$ to $C^\infty$ vector fields $w'_1,\dots,w'_r$ on $\pi^{-1}S$ so that $w'_i|_{E_{J}}\in T_{E_J}$, $\forall J\subset I_S$. Actually one can lift $w_1,\dots, w_r$ to the minimal strata of $\pi^{-1}S$ and extend them to strata of higher dimensions step by step.
    		\item[Step 2] Extend $w'_1,\dots,w'_r$ to $C^\infty$ vector fields $w''_1,\dots,w''_r$ on $\bigcup_{i\in I'_S}E_i$ so that $w''_i|_{E_{J}}\in T_{E_J}$ for every $J\subset I'_S$.
    		\item[Step 3] Extend $w''_1,\dots,w''_r$ to $C^\infty$ vector fields $v_1,\dots,v_r$ on $\widetilde{X}$. By the previous two steps, conditions (1) and (3) hold. Since the vector fields are $C^\infty$, condition (2) automatically holds when $V$ is sufficiently small.
    	\end{description}
    	Notice that $E$ admits only simple normal crossing singularities. The $C^\infty$ extensions of the vector fields are possible in each step.
    \end{proof}
    \begin{defn}\label{defn_bimero_fib_whitneystr}
    	Notations as in Proposition \ref{prop_bimero_fibering_neigh},     	$(U_x,U_{S,x},C_x,V,\pi,\alpha_x)$ is called a bimeromorphic fibering neighborhood of $x\in S$ with respect to the Whitney stratified desingularization $(\{X_i\},\pi:\widetilde{X}\to X)$.
    \end{defn}
    The following proposition shows that the Whitney stratified desingularization owns good metrical properties. That is, if $(\{X_i\},\pi:\widetilde{X}\to X)$ is a Whitney stratified desingularization of $X\subset\bC^N$, the euclidean metric is bimeromorphically stratified along every stratum of $\{X_i\}$ in a nice way.
    \begin{prop}\label{prop_hermitian_stratified}
    	Let $X$ be a germ of complex space and $X^o\subset X_{\rm reg}$ a dense Zariski open subset. Let $ds^2_0$ be a hermitian metric on $X$ and let $(\{X_i\},\pi:\widetilde{X}\to X)$ be a Whitney stratified desingularization of $(X,X^o)$. Then $ds^2_0$ is stratified along any bimeromorphic fibering neighborhood with respect to $(\{X_i\},\pi:\widetilde{X}\to X)$.

    	Precisely, let $x\in S$ be a point in a stratum $S$ of $\{X_i\}$ and let $(U_x,U_{S,x},C_x,V,\pi_x,\alpha_x)$ be a bimeromorphic fibering neighborhood of $x\in S$ with respect to $(\{X_i\},\pi:\widetilde{X}\to X)$. Then
    	$ds^2_0$ is stratified along $(U_x,U_{S,x},C_x,V,\pi_x,\alpha_x)$.
    \end{prop}
\begin{proof}
	Let $v_1,\cdots, v_{2d}$ be the defining vector fields of the diffeomorphism $\alpha_x$, as constructed in the proof of Proposition \ref{prop_bimero_fibering_neigh}. Denote by $\omega_0$ the  $(1,1)$-form associated to $ds^2_0$. We claim that there is a neighborhood $U\subset X$ of $x$ such that
	\begin{enumerate}
		\item For every $i=1,\dots, 2d$, $\|v_i\|_{\omega_0}\sim 1$ over $U^o:=U\cap X^o$. For simplicity we identify $U^o$ and $\pi^{-1}U^o$.
		\item There is a constant $\epsilon>0$ such that $\angle_{\omega_0}(v_i,v_j)>\epsilon$ over $U^o$ for every $1\leq i\neq j\leq 2d$.
		\item There is a constant $\epsilon>0$ such that $\angle_{\omega_0}(v_i,w)>\epsilon$ over $U^o$ for every $i=1,\dots, 2d$ and every $w\in T_{C_x\cap X^o}$.
		\item Denote by $$q:V\simeq V\cap \pi^{-1}C_x\times U_{S,x}\to V\cap \pi^{-1}C_x$$ and $$p:V\simeq V\cap \pi^{-1}C_x\times U_{S,x}\to U_{S,x}$$ the projections. By abuse of notations we identify $X^o$ and $\pi^{-1}X^o$. Let $\widetilde{C}_s:= p^{-1}\{s\}$ for $s\in U_{S,x}$. There is a constant $C>0$ such that
		$$C^{-1}\omega_0|_{\widetilde{C}_s}\leq\left(q^\ast(\omega_0|_{\widetilde{C}_x})\right)|_{\widetilde{C}_s}\leq C\omega_0|_{\widetilde{C}_s}$$ for every $s\in U\cap U_{S,x}$.
	\end{enumerate}
	Assuming these claims, for every $\xi=\sum_{i=1}^d f_iv_i+q^\ast w$ where $w\in T_{C_x\cap X^o}$ we have
	\begin{align*}
	\|\xi\|_{\omega_0}^2&=\sum_{i=1}^d\|f_i\|^2\|v_i\|^2_{\omega_0}+\sum_{i=1}^d2{\rm Re}(f_iv_i,q^\ast w)+\sum_{i\neq j}2{\rm Re}(f_iv_i,f_jv_j)+\|q^\ast w\|^2_{\omega_0}\\\nonumber
	&\leq 3\sum_{i=1}^d\|f_i\|^2\|v_i\|^2_{\omega_0}+3\|q^\ast w\|^2_{\omega_0}\\\nonumber
	&\leq 3\sum_{i=1}^d\|f_i\|^2\|v_i\|^2_{\omega_0}+3C\| w\|^2_{\omega_0|_{C_x}}
	\end{align*}
	and
	\begin{align*}
	\|\xi\|_{\omega_0}^2&=\sum_{i=1}^d\|f_i\|^2\|v_i\|^2_{\omega_0}+\sum_{i=1}^d2{\rm Re}(f_iv_i,q^\ast w)+\sum_{i\neq j}2{\rm Re}(f_iv_i,f_jv_j)+\|q^\ast w\|^2_{\omega_0}\\\nonumber
	&\geq \sum_{i=1}^d\|f_i\|^2\|v_i\|^2_{\omega_0}+\|q^\ast w\|^2_{\omega_0}-\sum_{i=1}^d2\epsilon\|f_iv_i\|_{\omega_0}\|q^\ast w\|_{\omega_0}-\sum_{i\neq j}2\epsilon\|f_iv_i\|_{\omega_0}\|f_jv_j\|_{\omega_0}\\\nonumber
	&\geq \sum_{i=1}^d(1-\epsilon)\|f_i\|^2\|v_i\|^2_{\omega_0}+(1-\epsilon)C^{-1}\|w\|^2_{\omega_0|_{C_x}}.\nonumber
	\end{align*}
	As a consequence,
	$$\omega_0\sim ds^2_{U_{S,x}}+q^\ast \omega_0|_{\widetilde{C}_x}$$
	is stratified along $(U_x,U_{S,x},C_x,V,\pi_x,\alpha_x)$. Here $ds^2_{U_{S,x}}$ is some hermitian metric of $U_{S,x}$.

	Now we prove claims (1)-(4). By the construction, $\pi_\ast(v_1)(x),\dots,\pi_\ast(v_d)(x)\in T_{S,x}$ is a basis, hence we obtain claims (1) and (2) by continuity. Since $T_{C_{x},x}\cap T_{S,x}=0$, claim (3) follows by continuity. Now we prove claim (4). If it is false, then there is a sequence of points $\{x_n\}\in V\cap \widetilde{C}_{s_n}\cap \pi^{-1}X^o$ and vectors $v_n\in T_{V,x_n}$ such that
	\begin{enumerate}
		\item $\{s_n\}$ converges to $x$ and $\{x_n\}$ converges to  $x''\in\pi^{-1}\{x\}$,
		\item For every $n$, $\|v_n\|_{\pi^\ast\omega_0}=1$.
		\item For every $n$, either
		\begin{align}\label{align_omega0_str_case1}
		\|v_n\|_{\pi^\ast\omega_0}\geq n\|\rho_{s_n \ast}v_n\|_{\pi^\ast\omega_0}
		\end{align}
		or
		\begin{align}\label{align_omega0_str_case2}
		\|v_n\|_{\pi^\ast\omega_0}\leq \frac{1}{n}\|\rho_{s_n\ast}v_n\|_{\pi^\ast\omega_0}.
		\end{align}
		Here $\rho_{s}:\widetilde{C}_{s}\subset V\stackrel{q}{\to} \widetilde{C}_x$ is the diffeomorphism induced by $\alpha_x$.
	\end{enumerate}
	Without loss of generality we assume that (\ref{align_omega0_str_case1}) holds for all $n$. Let $\gamma:[0,1)\to V\cap\pi^{-1}C_x$ be a curve such that
	\begin{enumerate}
		\item $\gamma$ is continuous and $\gamma|_{(0,1)}$ is $C^\infty$,
		\item $\gamma(0)=x''$ and $q(x_n)\in\gamma((0,1))$ for every $n$.
	\end{enumerate}
	Let $Y=q^{-1}{\rm Im}\gamma$, then $Y\cap \pi^{-1}X^o$ is an immersed submanifold with its boundary $Y\backslash \pi^{-1}X^o=q^{-1}\gamma^{-1}\{0\}\subset \pi^{-1}S$. Then $\rho_{s \ast}: T_V|_{Y\cap \widetilde{C}_s}\to T_V|_{Y\cap \widetilde{C}_x}$ is a $C^\infty$ map smoothly depending on $s\in U_{S,x}$. By the constructions, $\pi$ sends $Y$ homeomorphically onto its image $\pi(Y)$ and $\pi(Y)\cap X^o$ is an immersed submanifold of $X^o$ whose boundary is $\pi(Y)\backslash X^o$. Hence $\rho_s:\widetilde{C}_s\to \widetilde{C}_x$ induces a continuous map $\tau_s:\pi(Y\cap\widetilde{C}_s)\to \pi(Y\cap\widetilde{C}_x)$ which continuously depends on $s$.  Note that $\alpha_x$ does not desend to $X$ since $v_1,\dots,v_d$ are not necessarily liftings of vector fields on $X$. This is the purpose that we introduce $Y$.

	Denote $E_s:={\rm Im}\left(\pi_\ast T_V|_{Y\cap \widetilde{C}_s}\to T_X|_{\pi(Y\cap \widetilde{C}_s)}\right)$, then $E_s\subset T_X|_{\pi(Y\cap \widetilde{C}_s)}$ is a closed subset. Note that $E_{s,\pi(y)}=T_{X,\pi(y)}\simeq T_{V,y}$ for every $y\in Y\cap \pi^{-1}X^o$ and $E_{s,\pi(y)}\simeq T_{V,y}/{\rm Ker}\pi_{\ast,y}$ for $y\in Y\backslash \pi^{-1}X^o$. Since ${\rm Ker}\pi_{\ast,z}=T_{\pi^{-1}\{x\},z}$ is the linear span of the tangent cone of $\pi^{-1}\{x\}$ at every $z\in\pi^{-1}\{s\}$, $\rho_{s\ast}$ sends ${\rm Ker}\pi_{\ast,y}$ into ${\rm Ker}\pi_{\ast,x}$ for every (unique) $y\in Y\cap\pi^{-1}\{s\}$. Therefore $\rho_{s\ast}$ induces a continuous map $\tau_{s\ast}:E_s\to E_x$ which continuously depends on $s$.

	By assumption, $\|\pi_\ast(v_n)\|_{\omega_0}=1$. So there is a subsequence $\pi_\ast(v_{n_k})\in E_{s_{n_k}}$ which converges to some $v\in E_x$ with $\|v\|_{\omega_0}=1$. By taking the limit of (\ref{align_omega0_str_case1})
	$$\|\pi_\ast(v_{n_k})\|_{\omega_0}\geq n\|\tau_{s_{n_k}\ast}\pi_\ast(v_{n_k})\|_{\omega_0},$$
	we obtain $1\geq+\infty$. Neither will (\ref{align_omega0_str_case2}) happen by the same argument. This proves claim (4) and we finish the proof of the proposition.
\end{proof}
	\subsection{Distinguished metric}\label{section_Grauert_metric}
	Let $X$ be a complex space of pure dimension $n$ and $X^o\subset X_{\rm reg}$ a dense Zariski open subset.
	Let $\widetilde{X}\to X$ be a good desingularization of $X\backslash X^o$ in the sense that
	\begin{itemize}
		\item $\widetilde{X}$ is a complex manifold, $\pi$ is a projective holomorphic map and $\pi|_{\pi^{-1}X^o}:\pi^{-1}X^o\to X^o$ is biholomorphic.
		\item The (reduced) exceptional loci $E:=\pi^{-1}(X\backslash X^o)$ is a simple normal crossing divisor of $\widetilde{X}$ with finite components.
	\end{itemize}
	Let $E=\bigcup_{i\in I}E_i$ be the irreducible decomposition. Since $\pi$ is projective, there are constants $a_i\in\bZ_{>0}$, $i\in I$ such that $\sO_{\widetilde{X}}(-\sum_{i\in I}a_iE_i)$ is $\pi$-positive, i.e. locally there is a hermitian metric $h$ on $\sO_{\widetilde{X}}(-\sum_{i\in I}a_iE_i)$ such that
	$\sqrt{-1}\Theta_h(\sO_{\widetilde{X}}(-\sum_{i\in I}a_iE_i))+\pi^\ast\omega_0$
	is positive definite on $\widetilde{X}$ for some hermitian metric $\omega_0$ on $X$.

	For each $i\in I$, let $h_i$ be a hermitian metric on $\sO_{\widetilde{X}}(E_i)$ such that $h^\ast=\prod_{i\in I}h_i^{\otimes a_i}$. Let $s_i\in \Gamma(\widetilde{X},\sO_{\widetilde{X}}(E_i))$ be the defining section of $E_i$ and denote $|s_i|:=|s_i|_{h_i}$ for short.

	By perturbing the vector $(a_i)_{i\in I}$ slightly in general directions and multiplying the common denominators we get integers
	$a_{ij}\in\bZ_{>0}$, $i\in I$, $j=1,\dots,m$ so that
	\begin{enumerate}
		\item After a possible shrinking of $X$, $(\sO_{\widetilde{X}}(-\sum_{i\in I}a_{ij}E_i),\prod_{i\in I}h^{a_{ij}\ast}_{i})$ is a hermitian line bundle with $\pi$-positive curvature for every $j=1,\dots, m$.
		\item ${\rm rank}(a_{ij})_{i\in I,j\in J}=m$ for any $J\subset\{1,\dots,m\}$ such that $|J|\leq\min(m,|I|)$.
	\end{enumerate}

	The distinguished potential is defined by
	\begin{align*}
	\varphi_{\pi}:=\sum_{j=1}^m\frac{1}{\log\left(-\log\prod_{i\in I}|s_i|^{2a_{ij}}\right)}.
	\end{align*}
	Denote $\psi_j=-\log\prod_{i\in I}|s_i|^{2a_{ij}}$, then
	\begin{align*}
	\ddbar\varphi_{\pi}=\sum_{j=1}^m\frac{2+\log\psi_j}{\psi_j^2\log^3\psi_j}\partial\psi_j\wedge\dbar\psi_j+\sum_{j=1}^m\frac{-\ddbar\psi_j}{\psi_j\log^2\psi_j}.
	\end{align*}
	Since $-\sqrt{-1}\ddbar\psi_j$ is $\pi$-positive for each $j=1,\dots,m$, locally $\sqrt{-1}\ddbar\varphi_{\pi}+\omega_0$ is positive definite for some hermitian metric $\omega_0$ on $X$.
	\begin{defn}\label{defn_Grauert_type_metric}
		Let $X$ be a complex space and $X^o\subset X_{\rm reg}$ a dense Zariski open subset. A hermitian metric $ds^2$ on $X^o$ is called distinguished associated to $(X,X^o)$ if for every point $x\in X$ there is a neighborhood $U$ and a Whitney stratified desingularization $\pi:\widetilde{U}\to U$ of $(U,X^o\cap U)$ (Definition \ref{defn_Lstr_desingularization}) such that $$ds^2|_U\sim\sqrt{-1}\ddbar\varphi_{\pi}+\omega_0.$$ Here $\omega_0$ is a hermitian metric on $U$. If $ds^2\sim \sqrt{-1}\ddbar\varphi_{\pi}+\omega_0$ for some global Whitney stratified desingularization $\pi:\widetilde{X}\to X$ and some hermitian metric $\omega_0$ on $X$, we say $ds^2$ is a distinguished metric associated to $\pi$.
	\end{defn}
	Every distinguished metric is locally complete and locally of finite volume. In particular they are complete and of finite volume when $X$ is compact. The proof of this fact will be postponed to Corollary \ref{cor_Gt_complete_finite_vol}.
	\begin{lem}\label{lem_Grauert_metric_exist}
		Let $X$ be a complex space and $X^o\subset X_{\rm reg}$ a dense Zariski open subset. Let $Y\subset X$ be a compact analytic subspace. Then there is a distinguished metric on a neighborhood of $Y$. If moreover $X$ admits a K\"ahler hermitian metric, then there is a K\"ahler distinguished metric on a neighborhood of $Y$. In particular,
		\begin{itemize}
			\item if $X$ is a germ of complex space, then after a possible shrinking of $X$ there is a K\"ahler distinguished metric associated to $(X,X^o)$;
			\item if $X$ is a compact (K\"ahler) complex space, then there is a (K\"ahler) distinguished metric associated to $(X,X^o)$;
			\item if $(X,X^o)$ is a pair of (quasi-projective) algebraic varieties, then there is a (K\"ahler) distinguished metric associated to $(X,X^o)$.
		\end{itemize}
	\end{lem}
	\begin{proof}
		By Theorem \ref{thm_Lstr_desingularization}, there is an open neighborhood $U$ of $Y$ and a Whitney stratified desingularization $\pi:\widetilde{U}\to U$ of $(U,X^o\cap U)$.
		Let $\omega_0$ be a (K\"ahler) hermitian metric on $X$. Then $\sqrt{-1}\ddbar\varphi_{\pi}+K\omega_0$, $K\gg 0$ is a positive $(1,1)$-form which gives a (K\"ahler) distinguished metric.
	\end{proof}
	\begin{lem}\label{lem_Grauert_metric_in_coordinates}
		Let $(\{X_i\},\pi:\widetilde{X}\to X)$ be a Whitney stratified desingularization of $(X,X^o)$ and $ds^2_{\pi}$ a distinguished metric associated to $\pi$. Let $\omega_{\widetilde{X}}$ be a hermitian metric on $\widetilde{X}$ and $\omega_0$ a hermitian metric on $X$.
		Let $S$ be a stratum of $\{X_i\}$ and $x\in \widetilde{X}$ such that $\pi(x)\in S$. Let $(z_1,\dots,z_n)$ be a $C^\infty$ complex coordinate of $\widetilde{X}$ near $x$ such that $x=(0,\dots,0)$ and the exceptional divisor $E$ of $\pi$ is defined by $\{z_1\cdots z_r=0\}$.  Then there is a quasi-isometry
		\begin{align*}
		ds^2_{\pi}&\sim\sum_{i=1}^r\frac{\sqrt{-1}dz_i\wedge d\overline{z}_i}{|z_i|^2\log^2|z_1\cdots z_r|^2\log^2(-\log|z_1\cdots z_r|^2)}\\\nonumber
		&+\frac{\omega_{\widetilde{X}}}{-\log|z_1\cdots z_r|^2\log^2(-\log|z_1\cdots z_r|^2)}+\pi^\ast\omega_0.
		\end{align*}
	\end{lem}
	\begin{proof}
		We assume that $\{s_i=0\}=\{z_i=0\}$ for $i=1,\dots, r$. Hence for every $i=1,\dots,r$,
		\begin{align*}
		|s_i|=f_iz_i+g_i\bar{z_i}
		\end{align*}
		for smooth functions $f_i$ and $g_i$. On the other hand $|s_i|\simeq\|w_i\|$ for some holomorphic function $w_i$ which vanishes on $\{s_i=0\}$ of order 1. Hence for every $i=1,\dots,r$,
		\begin{align*}
		z_i=\tilde{f}_iw_i+\tilde{g}_i\bar{w_i}
		\end{align*}
		for smooth functions $\tilde{f}_i$ and $\tilde{g}_i$.
		As a consequence $|s_i|/\|z_i\|$ is an invertible $C^\infty$ function for every $i=1,\dots, r$.
		Denote $\psi_j=-\log\prod_{i=1}^r|s_i|^{2a_{ij}}$. Then
		\begin{align}\label{align_psi_pi_negative}
		-\sqrt{-1}\ddbar\psi_j+K\pi^\ast\omega_0\sim\omega_{\widetilde{X}}, \quad K\gg0
		\end{align}
		and $\psi_j= \lambda_j-\log\prod_{i=1}^r\|z_i\|^{2a_{ij}}$ for some invertible $C^\infty$ function $\lambda_j$. In particular
		\begin{align*}
		\psi_j\sim u:=-\log\prod_{i=1}^r\|z_i\|^{2}.
		\end{align*}
		Hence we have
		\begin{align}\label{align_Grauert_1}
		\ddbar\varphi_{\pi}&=\sum_{j=1}^m\frac{2+\log(\psi_j)}{\psi_j^2\log^3(\psi_j)}\partial\psi_j\wedge\dbar\psi_j+\sum_{j=1}^m\frac{-\ddbar\psi_j}{\psi_j\log^2(\psi_j)}\\\nonumber
		&\sim\frac{1}{u^2\log^2u}\sum_{j=1}^m\partial\psi_j\wedge\dbar\psi_j+\frac{1}{u\log^2u}\sum_{j=1}^m-\ddbar\psi_j.
		\end{align}
		Denote $u_j:=-\log\prod_{i=1}^r\|z_i\|^{2a_{ij}}$, then $\psi_j= \lambda_j+u_j$. So
		\begin{align}\label{align_Grauert_2}
		\sqrt{-1}\partial\psi_j\wedge\dbar\psi_j&=\sqrt{-1}\partial(\lambda_j+u_j)\wedge\dbar(\lambda_j+u_j)\\\nonumber
		&=2\sqrt{-1}\partial u_j\wedge\dbar u_j+2\sqrt{-1}\partial \lambda_j\wedge\dbar \lambda_j-\sqrt{-1} \partial(\lambda_j-u_j)\wedge\dbar(\lambda_j-u_j)\\\nonumber
		&\lesssim\sqrt{-1}\partial u_j\wedge\dbar u_j+\sqrt{-1}\partial \lambda_j\wedge\dbar \lambda_j.
		\end{align}
		By the same calculation we know that
		\begin{align}\label{align_Grauert_3}
		\sqrt{-1}\partial u_j\wedge\dbar u_j\lesssim \sqrt{-1}\partial\psi_j\wedge\dbar\psi_j+\sqrt{-1}\partial\lambda_j\wedge\dbar \lambda_j.
		\end{align}
		Combining (\ref{align_psi_pi_negative}), (\ref{align_Grauert_1}), (\ref{align_Grauert_2}) and (\ref{align_Grauert_3}) we obtain that
		\begin{align*}
		ds^2_{\pi}&\sim\frac{\sqrt{-1}}{u^2\log^2u}\sum_{\alpha=1}^m\partial u_\alpha\wedge\dbar u_\alpha+\frac{\omega_{\widetilde{X}}}{u\log^2u}+\pi^\ast\omega_0\\\nonumber
		&=\frac{\sqrt{-1}}{u^2\log^2u}\sum_{1\leq i,j\leq r}\left(\sum_{\alpha=1}^ma_{\alpha i}a_{\alpha j}\right)\frac{dz_i\wedge d\bar{z}_j}{z_i\bar{z}_j}+\frac{\omega_{\widetilde{X}}}{u\log^2u}+\pi^\ast\omega_0\\\nonumber
		&\sim\frac{\sqrt{-1}}{u^2\log^2u}\sum_{i=1}^r\frac{dz_i\wedge d\bar{z}_i}{|z_i|^2}+\frac{\omega_{\widetilde{X}}}{u\log^2u}+\pi^\ast\omega_0.
		\end{align*}
	\end{proof}
	\begin{rmk}\label{rmk_Grauert_independent}
		As a corollary, the distinguished metric only depends on the choice of the desingularization $\pi:\widetilde{X}\to X$. Its local quasi-isometric class is independent of the choice of the sections $\{s_i\}$, the hermitian metrics $\{h_i\}$ and the constants $\{a_{ij}\}$.
	\end{rmk}
	\begin{prop}\label{prop_Grauert_stratified}
		Let $X$ be a germ of complex space and $X^o\subset X_{\rm reg}$ a dense Zariski open subset. Let $(\{X_i\},\pi:\widetilde{X}\to X)$ be a Whitney stratified desingularization of $(X,X^o)$ and $ds^2_\pi$ the associated distinguished metric. Then $ds^2_\pi$ is stratified along any bimeromorphic fibering neighborhood with respect to $(\{X_i\},\pi:\widetilde{X}\to X)$.

		More precisely, let $x\in S$ be a point in a stratum $S$ of $\{X_i\}$ and let $(U_x,U_{S,x},C_x,V,\pi_x,\alpha_x)$ be a bimeromorphic fibering neighborhood of $x\in S$ with respect to the Whitney stratified desingularization (Definition \ref{defn_bimero_fib_whitneystr}). Then
		$ds^2_\pi$ is stratified along $(U_x,U_{S,x},C_x,V,\pi_x,\alpha_x)$.

		Moreover, $ds^2_{\pi}|_{C_x}$ is a distinguished metric associated to the Whitney stratified desingularization $(\{C_x\cap X_i\},\pi:\pi^{-1}C_x\to C_x)$ (c.f. Lemma \ref{lem_restrict_Whitney_desingularization}).
	\end{prop}
	\begin{proof}
		Let $x\in S$ for some stratum $S$ of dimension $d$ and let $(U_x,U_{S,x},C_x,V,\pi,\alpha_x)$ be a bimeromorphic fibering neighborhood of $x\in S$ with respect to the Whitney stratified desingularization. Denote by $E$ the exceptional divisor of $\pi$.
		Let $y\in \pi^{-1}\{x\}$ and let $(z'_{1},\dots,z'_{n-d})$ be a holomorphic coordinate of $x'\in\pi^{-1}C_x$ so that $E\cap \pi^{-1}C_x:=\{z'_1\cdots z'_r=0\}$ for some $r\leq n-d$. Let $(z'_{n-d+1},\dots,z'_n)$ be a holomorphic coordinate of $x\in U_{S,x}$, then the diffeomorphism $\alpha_x: V\simeq (V\cap\pi^{-1}C_x)\times U_{S,x}$ (\ref{align_alpha_x}) gives a $C^\infty$ complex coordinate $(z_1,\dots,z_n)$ of $x'\in V$ where $z_i=z'_i\alpha_x$. Under this coordinate we have $E=\{z_1\cdots z_r=0\}$.

		By Lemma \ref{lem_Grauert_metric_in_coordinates} there is a quasi-isometry
		\begin{align*}
		ds^2_{\pi}&\sim\Phi+\pi^\ast\omega_0
		\end{align*}
		where
		$$\Phi=\sum_{i=1}^r\frac{\sqrt{-1}dz_i\wedge d\overline{z}_i}{|z_i|^2\log^2|z_1\cdots z_r|^2\log^2(-\log|z_1\cdots z_r|^2)}+\frac{\sum_{i=1}^ndz_i\wedge d\overline{z}_i}{-\log|z_1\cdots z_r|^2\log^2(-\log|z_1\cdots z_r|^2)}.$$
		The quasi-isometric class of $\Phi$ does not depend on the choice of the coordinate $(z_{n-d+1},\dots, z_n)$. Hence by Proposition \ref{prop_hermitian_stratified}, $ds^2_\pi$ is stratified along $(U_x,U_{S,x},C_x,V,\pi,\alpha_x)$.

		By Lemma \ref{lem_restrict_Whitney_desingularization}, $(\{C_x\cap X_i\},\pi_x:=\pi|_{\pi^{-1}C_x}:\pi^{-1}C_x\to C_x)$ is a Whitney stratified desingularization of $(C_x,C_x\cap X^o)$. Hence $\varphi_\pi|_{C_x}=\varphi_{\pi_x}$. Therefore
		$$ds^2_\pi|_{C_x}=\sqrt{-1}\ddbar\varphi_{\pi_x}+\pi^\ast_x(\omega_0|_{C_x})$$
		is a distinguished metric associated to $(\{C_x\cap X_i\},\pi_x)$.
	\end{proof}
	\subsection{Harmonic bundles}
	\subsubsection{pluriharmonic bundles}
	Let $(M,ds^2)$ be a K\"ahler manifold and $(\cV,\nabla)$ a holomorphic vector bundle with a flat connection on $M$. Let $h$ be a hermitian metric on $\cV$ which is not necessarily compatible with $\nabla$. Denote by $\bV:={\rm Ker}\nabla$ the complex local system associated to $(\cV,\nabla)$. $(\bV,h)$ is called a hermitian local system.

	Let $[-,-]$ be the graded Lie bracket of operators on $\sA^\bullet_M(\bV)$, i.e.
	$$[T_1,T_2]=T_1T_2-(-1)^{\deg T_1\deg T_2}T_2T_1.$$

	Let $\nabla=\nabla^{1,0}+\nabla^{0,1}$ be the bi-degree decomposition. There are unique operators
	$\delta'_h$ and $\delta''_h$ so that $\nabla^{1,0}+\delta''_h$ and $\nabla^{0,1}+\delta'_h$ are connections compatible with $h$. Denote $\nabla^c_h=\delta''_h-\delta'_h$. Then we have the K\"ahler identities (\cite[Page 15]{Simpson1992})
	\begin{align}\label{align_Kahler_id2}
	[\Lambda_{\omega},\nabla]=\sqrt{-1}(\nabla^c_h)^\ast,\quad [\Lambda_{\omega},\nabla^c_h]=-\sqrt{-1}\nabla^\ast_h.
	\end{align}
	Here $\omega$ is the K\"ahler form associated to $ds^2$ and $( )^\ast$ is the formal adjoint with respect to the metric $h$. Denote
	\begin{align*}
	\Theta_h(\nabla):=[\nabla,\nabla^c_h]=\nabla\nabla^c_h+\nabla^c_h\nabla
	\end{align*}
	and
	\begin{align*}
	\Delta_\nabla=\nabla\nabla^\ast_h+\nabla^\ast_h\nabla,\quad \Delta_{\nabla^c_h}={\nabla^c_h}{\nabla^c_h}^\ast+{\nabla^c_h}^\ast{\nabla^c_h}.
	\end{align*}
	We have
	\begin{lem}[Kodaira-Bochner formula for $\nabla$]\label{lem_Bochner_formula_harmonic_bundle}
		\begin{align*}
		\Delta_{\nabla}=\Delta_{\nabla^c_h}+[\Lambda_{\omega},\sqrt{-1}\Theta_h(\nabla)].
		\end{align*}
	\end{lem}
	\begin{proof}
		\begin{align*}
		\Delta_{\nabla}-\Delta_{\nabla^c_{h}}&=\sqrt{-1}\nabla[\Lambda_\omega,\nabla^c_h]+\sqrt{-1}[\Lambda_\omega,\nabla^c_h]\nabla+\sqrt{-1}\nabla^c_{h}[\Lambda_\omega,\nabla]+\sqrt{-1}[\Lambda_\omega,\nabla]\nabla^c_{h} \quad (\ref{align_Kahler_id2})\\\nonumber
		&=\sqrt{-1}\left(-\nabla\nabla^c_{h}\Lambda_\omega+\Lambda_\omega \nabla^c_{h}\nabla-\nabla^c_{h}\nabla\Lambda_\omega+\Lambda_\omega \nabla\nabla^c_{h}\right)\\\nonumber
		&=[\Lambda_\omega,\sqrt{-1}\Theta_h(\nabla)].\\\nonumber
		\end{align*}
	\end{proof}
	\begin{lem}\label{lem_twisted_curvature}
		$\Theta_{e^\varphi h}(\nabla)=\Theta_{h}(\nabla)+dd^c\varphi$, where $d^c=\dbar-\partial$.
	\end{lem}
	\begin{proof}
		For every $C^\infty$ sections $\alpha$, $\beta$ of $\cV$ there are equalities
		\begin{align}\label{align_lem_twisted_curvature_compatible1}
		d\left((\alpha,\beta)_he^\varphi\right)=\left((\nabla^{1,0}+\delta''_{e^\varphi h})\alpha,\beta\right)_he^\varphi+\left(\alpha,(\nabla^{1,0}+\delta''_{e^\varphi h})\beta\right)_he^\varphi
		\end{align}
		and
		\begin{align}\label{align_lem_twisted_curvature_compatible2}
		d\left((\alpha,\beta)_h\right)=\left((\nabla^{1,0}+\delta''_{h})\alpha,\beta\right)_h+\left(\alpha,(\nabla^{1,0}+\delta''_{h})\beta\right)_h.
		\end{align}
		Combining (\ref{align_lem_twisted_curvature_compatible1}) and (\ref{align_lem_twisted_curvature_compatible1}) we get
		\begin{align*}
		\delta''_{e^\varphi h}=\delta''_{h}+\dbar\varphi.
		\end{align*}
		Similarly,
		\begin{align*}
		\delta'_{e^\varphi h}=\delta'_{h}+\partial\varphi.
		\end{align*}
		As a consequence
		\begin{align*}
		\Theta_{e^\varphi h}(\nabla)&=\nabla(\delta''_{e^\varphi h}-\delta'_{e^\varphi h})+(\delta''_{e^\varphi h}-\delta'_{e^\varphi h})\nabla\\\nonumber
		&=\nabla(\nabla^c_{h}+d^c\varphi)+(\nabla^c_{h}+d^c\varphi)\nabla\\\nonumber
		&=\Theta_{h}(\nabla)+dd^c\varphi.
		\end{align*}
	\end{proof}
	\begin{defn}
		$(\bV,h)$ or $(\cV,\nabla,h)$ is called a pluriharmonic bundle if $\Theta_h(\nabla)=0$. In this case $h$ is called a pluriharmonic metric.
	\end{defn}
	The notion of harmonic bundle is used by Simpson \cite{Simpson1992} to establish a correspondence between local systems and semistable higgs bundles with vanishing Chern classes over a compact K\"ahler manifold. A typical example of pluriharmonic bundles comes from a polarized variation of Hodge structure (loc. cit.). A pluriharmonic bundle produces a $\lambda$-connection structure which gives a $\sD$-module on $\lambda=1$ and a higgs bundle on $\lambda=0$. This is the main subject of non-abelian Hodge theory. We will not go further on this splendid topic. Readers may consult \cite{Simpson1988,Simpson1990,Simpson1992,Sabbah2005,Mochizuki20072,Mochizuki20071} for more details.
	\subsubsection{$0$-tame harmonic bundles}
	The $0$-tame harmonic bundle is a subclass of the tame harmonic bundles in the sense of Simpson \cite{Simpson1990} and Mochizuki \cite{Mochizuki20072,Mochizuki20071}.
	\begin{defn}\label{defn_0tame_harmonic_bundle}
		Let $X$ be a complex space and $X^o\subset X_{\rm reg}$ a dense Zariski open subset. A pluriharmonic bundle $(\bV,h)$ on $X^o$ is called a $0$-tame harmonic bundle on $(X,X^o)$ if locally at every point $x\in X$ one has
		\begin{align}\label{align_defn_0tame}
		(\sum_{i=1}^r\|f_i\|^2)^{\epsilon}\lesssim|v|_h\lesssim (\sum_{i=1}^r\|f_i\|^2)^{-\epsilon},\quad\forall\epsilon>0
		\end{align}
		for every (multivalued) flat section $v$. Here $\{f_1,\dots,f_r\}$ is a local generator of the ideal sheaf defining $X\backslash X^o\subset X$.
	\end{defn}
    The $0$-tameness is independent of the choice of the local generators $f_1,\dots,f_r$. Let $\pi:\widetilde{X}\to X$ be a desingularization of $X\backslash X^o\subset X$ so that the preimage of $X\backslash X^o$ is supported on a simple normal crossing divisor $E$. Then (\ref{align_defn_0tame}) is equivalent to the estimate
    \begin{align}\label{align_defn_0tame2}
    \|z_1\cdots z_r\|^{\epsilon}\lesssim|v|_h\lesssim \|z_1\cdots z_r\|^{-\epsilon},\quad\forall\epsilon>0
    \end{align}
    where $(z_1,\dots,z_n)$ is an arbitrary holomorphic local coordinate of $\widetilde{X}$ such that $E=\{z_1\cdots z_r=0\}$.

    One has a more general notion of tame harmonic bundles, with (\ref{align_defn_0tame}) replaced by
    \begin{align*}
    |v|_h\lesssim (\sum_{i=1}^r\|f_i\|^2)^{-b}
    \end{align*}
	for some $b>0$. Or equivalently, the eigenvalues of the higgs field are multivalued holomorphic one-forms with poles along (some resolution of) $X\backslash X^o$ with order at most 1. For a tame harmonic bundle $(V,h)$ on $(X,X^o)$ and a holomorphic disc $\gamma:\Delta\to X$ so that $\gamma(0)\in X$ and $\gamma(\Delta^\ast)\subset X^o$, the restriction $(\gamma^\ast\bV,\gamma^\ast h)$ is a tame harmonic bundle on $\Delta^\ast$ which satisfies the estimate $|v|_{\gamma^\ast h}\lesssim \|w\|^{-b}$
	for some $b>0$. Here $w$ is the coordinate of $\Delta$. According to \cite[\S 3]{Simpson1990} there is a canonical filtration of bundles
	$F_\alpha$ consisting of the holomorphic sections of $\gamma^\ast\cV$ (the associated vector bundle of $\gamma^\ast\bV$) with the estimate $|v|_{\gamma^\ast h}\lesssim \|w\|^{\alpha-\epsilon}$, $\forall\epsilon>0$. The $0$-tameness is equivalent to that: for each such $\gamma:\Delta\to X$, $F_0$ is the only object that matters in the filtration for both $(\gamma^\ast\bV,\gamma^\ast h)$ and $(\gamma^\ast\bV^\ast,\gamma^\ast h^\ast)$. This is the origin of the name "$0$-tame". Using notions in \cite{Simpson1990,Mochizuki20072}, $0$-tame harmonic bundles are exactly the tame harmonic bundles without parabolic structures.

	The asymptotic behavior of the $0$-tame harmonic metric coincides with the Hodge metric on a variation of Hodge structure. Let $(\bV,h)$ be a $0$-tame harmonic bundle on $$(\Delta^\ast)^n\times\Delta^m=\{(z_1,\dots,z_n,w_1,\dots,w_m)|0<\|z_1\|,\dots,\|z_n\|<1,\|w_1\|,\dots, \|w_m\|<1\}.$$
	Let $(E,\theta)$ be the associated higgs bundle with regular singularities (\cite{Simpson1990,Mochizuki20072}) and $N_i$ the nilpotent part of the residue ${\rm Res}_{D_i}\theta$ along $D_i:=\{z_i=0\}$.
	\begin{thm}[Mochizuki, \cite{Mochizuki20072}, Part 3, Chapter 13]\label{thm_0tame_estimate}
		Let $(\bV,h)$ be a $0$-tame harmonic bundle on $(\Delta^\ast)^n\times\Delta^m$. Let
		$$p:\bH^{n}\times \Delta^m\to (\Delta^\ast)^n\times \Delta^m,$$
		$$(z_1,\dots,z_n,w_1,\dots,w_m)\mapsto(e^{2\pi\sqrt{-1}z_1},\dots,e^{2\pi\sqrt{-1}z_n},w_1,\dots,w_m)$$
		be the universal covering. Let
		$W^{(1)}=W(N_1),\dots,W^{(n)}=W(N_1+\cdots+N_n)$ be the residue weight filtrations on $V:=\Gamma(\bH^n,p^\ast\cV)^{p^\ast\nabla}$.
		Then
		for any $v\in V$ such that $0\neq [v]\in {\rm Gr}_{l_n}^{W^{(n)}}\cdots{\rm Gr}_{l_1}^{W^{(1)}}V$, one has
		\begin{align}\label{align_VHS_regular_sing}
		|v|_{h_\bV}\sim \left(\frac{\log|s_1|}{\log|s_2|}\right)^{l_1}\cdots\left(-\log|s_n|\right)^{l_n}
		\end{align}
		over any region of the form
		$$\left\{(s_1,\dots, s_n,w_1,\dots,w_m)\in (\Delta^\ast)^n\times \Delta^m\bigg|\frac{\log|s_1|}{\log|s_2|}>\epsilon,\dots,-\log|s_n|>\epsilon,(w_1,\dots,w_m)\in K\right\}$$
		for any $\epsilon>0$ and an arbitrary compact subset $K\subset \Delta^m$.
	\end{thm}

	\begin{rmk}\label{rmk_harmonic_bundle_regular_singularity_examples}
		There are two ways to get a $0$-tame harmonic bundle.
		\begin{enumerate}
			\item Let $X$ be a complex space and $X^o\subset X_{\rm reg}$ a dense Zariski open subset. Let $(\cV,\nabla,\cF^\bullet,h_\bV)$ be a polarized $\bZ$-variation of Hodge structure on $X^o$, or more generally, an $\bR$-polarized variation of Hodge structure with quasi-unipotent local monodromies. Then $(\bV={\rm Ker}\nabla,h_\bV)$ is a $0$-tame harmonic bundle on $(X,X^o)$. Indeed by desingularization locally we may resolve the subspace $X\backslash X^o$ to a simple normal crossing divisor $E$ of some smooth manifold.
			By \cite[Theorem 5.21]{Cattani_Kaplan_Schmid1986} the Hodge metric satisfies the estimate (\ref{align_VHS_regular_sing}).
			\item Let $X$ be a compact K\"ahler manifold and $X^o=X\backslash D$ for some normal crossing divisor $D\subset X$. Let $\bV$ be a local system on $X^o$ which does not necessarily have quasi-unipotent local monodromy. Depending on the work by Jost and Zuo \cite{Jost_Zuo1997}, Mochizuki (\cite[Theorem 1.19]{Mochizuki20072}) shows that the $0$-tame harmonic metric exists on $\bV$ if and only if $\bV$ is semisimple. Moreover the $0$-tame harmonic metric is unique up to flat automorphisms. In fact this metric shares many other good properties similar to the Hodge metric on a variation of Hodge structure.

			For a compact K\"ahler space $X$, a dense Zariski open subset $X^o\subset X_{\rm reg}$ and a semisimple local system $\bV$ on $X^o$, there is a $0$-tame harmonic metric on $\bV$ (unique up to flat automorphisms) since one can blowup $X\backslash X^o$ to a simple normal crossing divisor.
		\end{enumerate}
	\end{rmk}
	\begin{prop}\label{prop_harmonic_metric_stratified}
		Let $X$ be a complex space and $X^o\subset X_{\rm reg}$ a dense Zariski open subset. Let $(\bV,h)$ be a $0$-tame harmonic bundle on $(X,X^o)$. Let $(\{X_i\},\pi:\widetilde{X}\to X)$ be a Whitney stratified desingularization of $(X,X^o)$. Then $(\bV,h)$ is stratified along any bimeromorphic fibering neighborhood with respect to the Whitney stratified desingularization.

		To be precise, let $x\in S$ be a point in a stratum $S$ of $\{X_i\}$ and let $(U_x,U_{S,x},C_x,V,\pi_x,\alpha_x)$ be an arbitrary bimeromorphic fibering neighborhood of $x\in S$ with respect to the Whitney stratified desingularization. Then $(\bV,h)$ is stratified along $(U_x,U_{S,x},C_x,V,\pi_x,\alpha_x)$.

		Moreover, $(\bV|_{C_x}, h|_{C_x})$ is a $0$-tame harmonic bundle on $(C_x,C_x\cap X^o)$.
	\end{prop}
	\begin{proof}
		Let $x\in S$ for some stratum $S$ of dimension $d$ and let $(U_x,U_{S,x},C_x,V,\pi,\alpha_x)$ be a bimeromorphic fibering neighborhood of $x\in S$ with respect to the Whitney stratified desingularization. Denote by $E$ the exceptional divisor of $\pi$.
		Let $y\in \pi^{-1}\{x\}$ and let $(z'_{1},\dots,z'_{n-d})$ be a holomorphic coordinate of $x'\in\pi^{-1}C_x$ so that $E\cap \pi^{-1}C_x:=\{z'_1\cdots z'_r=0\}$ for some $r\leq n-d$. Let $(z'_{n-d+1},\dots,z'_n)$ be a holomorphic coordinate of $x\in U_{S,x}$, then the diffeomorphism $\alpha_x: V\simeq (V\cap\pi^{-1}C_x)\times U_{S,x}$ (\ref{align_alpha_x}) gives a $C^\infty$ complex coordinate $(z_1,\dots,z_n)$ of $x'\in V$ where $z_i=z'_i\alpha_x$. Under this coordinate we have $E=\{z_1\cdots z_r=0\}$.

		By Theorem \ref{thm_0tame_estimate}, the quasi-isometric class of $h$ does not involve $z_{n-d+1},\dots, z_n$. Hence $h$ is stratified along $(U_x,U_{S,x},C_x,V,\pi_x,\alpha_x)$.

		Since $\pi^{-1}C_x=\{z_{n-d+1}=\dots=z_n=0\}$, the last statement of the theorem follows.
	\end{proof}
	\section{Local Geometry of Singularities}
	\subsection{Basic setting}\label{section_setting}
	This subsection contains the basic settings for the local vanishings (\ref{align_vanishing}) and (\ref{align_vanishing_cpt}) which will be used in the rest of the paper.

	Let $w=(w_1,\dots,w_N)$ be the coordinate of $\bC^N$. Denote $\|w\|^2:=\sum_{i=1}^N\|w_i\|^2$.
	Let $0\in Z\subset\Omega\subset\bC^N$ be a closed analytic subset of a bounded Stein open subset $\Omega$ of $\bC^N$. Assume that $Z$ is of pure dimension $n$ and $Z^o\subset Z_{\rm reg}$ is a dense Zariski open subset such that $0\notin Z^o$.  Let $(\bV,h)$ be a $0$-tame harmonic bundle on $(Z,Z^o)$. Let $(\{Z_i\}_{0\leq i\leq n},\pi:\widetilde{Z}\to Z)$ be a Whitney stratified desingularization of $(Z,Z^o)$.  Let $E=\sum_{i\in I}E_i$ be the finite irreducible decomposition of the (reduced) exceptional divisor $E:=\pi^{-1}(Z\backslash Z^o)$.
	For every $\emptyset\neq J\subset I$, denote $E_J:=\bigcap_{j\in J}E_j$. We also denote $E_\emptyset:=Z^o$.

	Assume that $\{0\}$ is the unique minimal stratum of $\{Z_i\}$. As a consequence $E_{\{0\}}:=\pi^{-1}\{0\}$ is a simple normal crossing divisor of $\widetilde{Z}$ and $E_{\{0\}}\cap E_J\neq\emptyset$ for every stratum $E_J$, $\emptyset\neq J\subset I$.
	The basic model for such a germ of space is $C_x$ in Proposition \ref{prop_Dstr}.

	Since $\pi$ is projective, we may choose integers $a_i\in\bZ_{>0}$, $i\in I$ so that $\sO_{\widetilde{Z}}(\sum_{i\in I}a_i E_i)$ admits a real analytic hermitian metric $h$ satisfying
	\begin{align}\label{align_setting_curvature_E}
	-\sqrt{-1}\Theta_{h}(\sO_{\widetilde{Z}}(\sum_{i\in I}a_i E_i))>0.
	\end{align}
	This is possible by \S \ref{section_Grauert_metric}. Let $h'$ be the real analytic hermitian metric on $\sO_{\widetilde{Z}}(-\sum_{i\in I}a_i E_i)$ with $\pi$-positive curvature along $E$, i.e. $\sqrt{-1}\Theta_{h'}(\sO_{\widetilde{Z}}(-\sum_{i\in I}a_i E_i))+\pi^\ast\omega_0>0$ near $E$ for some hermitian metric $\omega_0$ on $Z$. Since locally at $0\in Z$, $\omega_0\leq\sqrt{-1}\ddbar\psi$ for some bounded real analytic function $\psi$ (such as $K\|w\|^2$, $K\gg0$), $h:=e^{\psi}h'^\ast \sim h'^\ast$ is a hermitian metric with negative curvature near $E_{\{0\}}$. Shrinking $Z$ we get (\ref{align_setting_curvature_E}).

	For each $i\in I$, let $h_i$ be a hermitian metric on $\sO_{\widetilde{Z}}(E_i)$ such that $h=\prod_{i\in I}h_i^{a_i}$.
	Let $s_i\in \Gamma(\widetilde{Z},\sO_{\widetilde{Z}}(E_i))$ be the defining section of $E_i$ and denote
	$$|s_i|:=|s_i|_{h_i},\quad |s(x)|:=\prod_{i\in I}|s_i(x)|^{a_i}$$ for short. By (\ref{align_setting_curvature_E}) we get
	\begin{align}\label{align_rho2_psh}
	\sqrt{-1}\ddbar\log|s|^2>0,
	\end{align}
	i.e. $\log|s|^2$ is a strictly plurisubharmonic function on $\widetilde{Z}$.

	Let $a_{ij}\in\bZ_{>0}$, $i\in I$, $j=1,\dots, m$ be constants so that
	\begin{enumerate}
		\item For every $j=1,\dots, m$, $(\sO_{\widetilde{Z}}(-\sum_{i\in I}a_{ij}E_i),\prod_{i\in I}h^{a_{ij}\ast}_{i})$ is a hermitian line bundle with positive curvature.
		\item ${\rm rank}(a_{ij})_{i\in I,j\in J}=m$ for any $J\subset\{1,\dots,m\}$ such that $|J|\leq\min(m,|I|)$.
	\end{enumerate}
	Define the plurisubharmonic function
	\begin{align*}
	\varphi_{\pi}:=\sum_{j=1}^m\frac{1}{\log\left(-\log\prod_{i\in I}|s_i|^{2a_{ij}}\right)}.
	\end{align*}
	With a possible shrinking of $Z$ and a rescaling of $\bC^N$ we assume that
	\begin{align*}
	\log(\prod_{i\in I}|s_i|^{a_{ij}})\leq -e^{m},\quad j=1,\dots, m
	\end{align*}
	and
	\begin{align}\label{align_metric}
	ds^2_{\pi}=\sqrt{-1}\ddbar\varphi_{\pi}+\sqrt{-1}\ddbar\|w\|^2.
	\end{align}
	This is a distinguished metric due to Remark \ref{rmk_Grauert_independent}.

	Denote $\psi_j=-\log\prod_{i\in I}|s_i|^{2a_{ij}}$. Then $-\psi_1,\dots, -\psi_m$ are strictly plurisubharmonic functions and
	\begin{align}\label{align_expend_Gt}
	\sqrt{-1}\ddbar\varphi_{\pi}=\sqrt{-1}\sum_{j=1}^m\frac{2+\log\psi_j}{\psi_j^2\log^3\psi_j}\partial\psi_j\wedge\dbar\psi_j+\sum_{j=1}^m\frac{-\sqrt{-1}\ddbar\psi_j}{\psi_j\log^2\psi_j}>0.
	\end{align}
	Hence $\varphi_\pi$ is a strictly plurisubharmonic function.

	As a consequence,
	\begin{align}\label{align_grad_norm_Gt}
	|d\varphi_{\pi}|_{\sqrt{-1}\ddbar\varphi_{\pi}}&\leq\sum_{j=1}^m\left|\frac{d\psi_j}{\psi_j\log^2\psi_j}\right|_{\sqrt{-1}\ddbar\varphi_{\pi}}\\\nonumber
	&\leq\sum_{j=1}^m\frac{2}{\psi_j\log^2\psi_j}\left(\frac{\psi^2_j\log^3\psi_j}{2+\log\psi_j}\right)^{\frac{1}{2}}\\\nonumber
	&\leq\sum_{j=1}^m\frac{2}{\log\psi_j}\leq 2.\\\nonumber
	\end{align}
	Denote $\rho_1(x)=\|\pi(x)\|$ and $\rho_2(x)=|s(x)|$.

	\subsection{Distinguished neighborhoods}\label{section_U_delta}
	In this subsection we introduce cofinal systems of neighborhoods of $0$ and $E_{\{0\}}=\pi^{-1}\{0\}$. They play crucial roles in proving (\ref{align_vanishing}) and (\ref{align_vanishing_cpt}).

	Denote
	\begin{align*}
	U_{\delta}:=\left\{w\in Z\big|\|w\|<\delta\right\},\quad\forall 0<\delta\ll 1
	\end{align*}
	and $U^o_{\delta}:=U_{\delta}\cap Z^o$.
	$\{U_{\delta}\}$ is a cofinal system of precompact open neighborhoods. Let $\widetilde{U_{\delta}}:=\pi^{-1}U_{\delta}$. $\{\widetilde{U_{\delta}}\}$ is a cofinal system of precompact open neighborhoods of $E_{\{0\}}$.

	The structure of $U_{\delta}$ is summarized by the following lemma.
	\begin{lem}\label{lem_U_delta}
		For every $0<\delta'<\delta\ll1$, there is a fibering
		\begin{align*}
		\alpha:U^o_{\delta}\backslash U^o_{\delta'}\simeq[\delta',\delta)\times M_{\delta}
		\end{align*}
		where $M_{\delta}:=\big\{w\in Z^o\big|\|w\|=\delta\big\}$ and the first factor of $\alpha$ is $\rho_1$. Moreover,
		\begin{align}\label{align_Gt_bd1}
		ds^2_{\pi}\sim d\rho_1^2+ds^2_{\pi}|_{M_{\delta}}
		\end{align}
		over $U^o_{\delta}\backslash U^o_{\delta'}$.
	\end{lem}
	\begin{proof}
		Since $\rho_1:\widetilde{Z}\backslash E_{\{0\}}\to\bR$ is a real analytic proper map, the critical values of the restriction maps $\rho_1|_{E_J}$, $J\subset I$ are isolated (\cite[Corollary 2.8]{Milnor1968}). Hence $\rho_1$ is a  proper stratified submersion over $(0,\epsilon)$ for some $0<\epsilon\ll 1$.

		By the same proof of Proposition \ref{prop_bimero_fibering_neigh} there is a $C^\infty$ vector field ${\bf v}$ on $\widetilde{U_{\delta}}\backslash\widetilde{U_{\delta'}}$ which lifts $\frac{\partial}{\partial \rho_1}$ via the proper stratified smooth submersion $\rho_1:\widetilde{U_{\delta}}\backslash\widetilde{U_{\delta'}}\to[\delta',\delta)$. Moreover, for every stratum $E_J$ of the exceptional divisor $E$, $J\subset I$, we have ${\bf v}|_{E_J}\in T_{E_J}$. Its associated one-parameter group gives the foliation $\alpha$. (\ref{align_Gt_bd1}) follows from Proposition \ref{prop_Grauert_stratified}.
	\end{proof}
	Next we introduce another cofinal system of neighborhoods of $E_{\{0\}}$.
	Let ${\rm grad}_{\omega_{\widetilde{Z}}}\rho_2$ be the gradient vector field of $\rho_2=|s|$ with respect to a hermitian metric $\omega_{\widetilde{Z}}$ on $\widetilde{Z}$. The following lemma characterizes the asymptotic behavior of ${\rm grad}_{\omega_{\widetilde{Z}}}\rho_2$.
	\begin{lem}\label{lem_flow_rho_2}
		Near every point $y\in E_J$ there is a complex chart $(z_1=r_1e^{\sqrt{-1}\theta_1},\dots,z_n=r_ne^{\sqrt{-1}\theta_n})$ and integers $b_1,\dots, b_q\in\bZ_{>0}$ such that $y=(0,\dots,0)$, $E_J=\{z_1\cdots z_q=0\}$ and
		\begin{align*}
		\rho_2=r^{b_1}_1\cdots r^{b_q}_q,\quad {\rm grad}_{g}\rho_2=r_1\dots r_q\sum_{i=1}^q\frac{b_i}{r_i}\frac{\partial}{\partial r_i}.
		\end{align*}
		Here $g:=\sum_{i=1}^ndz_id\overline{z}_i$.
		As a consequence, the integral curve of ${\rm grad}_{g}\rho_2$ is
		\begin{align*}
		\gamma_g(x,t)=x+r_1\cdots r_q\left(\frac{b_1}{r_1},\dots,\frac{b_q}{r_q}\right)t.
		\end{align*}
	\end{lem}
	\begin{proof}
		Let $(w_1,\dots, w_n)$ be a holomorphic chart near $y$ such that $E_J=\{w_1\cdots w_q=0\}$.
		Since there is an invertible $C^\infty$ function $u>0$ such that $\rho_2=u(w_1,\dots,w_n)|w_1^{b_1}\dots w_q^{b_q}|$. Then $(z_1,\dots, z_n):=(w_1,\dots, u(w_1,\dots, w_n)w_n)$ is the complex chart claimed in the lemma. The remaining claims of the lemma follow from direct calculations.
	\end{proof}
	\begin{lem}\label{lem_integral_curve_rho_2}
		There is a $C^\infty$ function $\ell$ such that for every $0<\delta\ll1$ and every
		\begin{align*}
		x\in L_\delta:=\big\{x\in Z^o\big|\rho_2(x)=\ell(\delta), \rho_1(x)\leq\delta\big\},
		\end{align*}
		the integral curve $\gamma(x,t)$ of ${\rm grad}_{\omega_{\widetilde{Z}}}\rho_2$ with $\gamma(x,\ell(\delta))=x$ exists on the interval $(0,\ell(\delta)]$. Moreover $\{\gamma(x,t)|x\in L_\delta, t\in(0,\ell(\delta))\}\subset U^o_{2\delta}$ and for every $x\in L_\delta$, $\gamma(x,t)$ can be continuously extended to $t=0$ such that $\gamma(x,0)\in E$.
	\end{lem}
	\begin{proof}
		Note that $\omega_{\widetilde{Z}}$ is locally quasi-isometric to the euclidean metric (with respect to complex charts). By Lemma \ref{lem_flow_rho_2} the claims of the proposition hold locally around every point $x\in \rho_1^{-1}\{\delta\}\cap E$ for $\delta\ll1$. Since $\rho_1^{-1}\{\delta\}\cap E$ is compact, making $\ell(\delta)$ small enough leads to Lemma \ref{lem_integral_curve_rho_2}.
	\end{proof}
	Define
	\begin{align*}
	L_\delta:=\big\{x\in Z^o\big|\rho_2(x)=\ell(\delta), \rho_1(x)\leq\delta\big\},
	\end{align*}
	\begin{align*}
	V_\delta:=\bigcup_{x\in L_\delta,t\in[0,\ell(\delta))}\gamma(x,t)
	\end{align*}
	and $V^o_\delta:=Z^o\cap V_\delta$. By Lemma \ref{lem_integral_curve_rho_2}, $\{V_\delta\}_{0<\delta\ll 1}$ is a cofinal system of precompact neighborhoods of $E_{\{0\}}$.
	The one-parameter subgroup associated to ${\rm grad}_{\omega_{\widetilde{Z}}}\rho_2$ gives a foliation
	\begin{align}\label{align_fiberation_V_delta}
	V^o_\delta\sim(0,\ell(\delta))\times L_\delta
	\end{align}
	whose first factor is $\rho_2$.
	The following proposition describes the asymptotic behavior of $ds^2_{\pi}$ along the fibering (\ref{align_fiberation_V_delta}).
	\begin{prop}\label{prop_Grauert_asymptotic_behavior}
		Notations as in \S \ref{section_setting}. Fix $0<\delta\ll 1$. Then
		\begin{align*}
		\frac{d\rho_2^2}{\rho_2^2\log^2\rho_2\log^2(-\log\rho_2)}+\frac{ds^2_\pi|_{L_\delta}}{\log^2\rho_2\log^2(-\log\rho_2)}
		\lesssim ds^2_{\pi}|_{V^o_\delta}
		\end{align*}
		and
		\begin{align*}
		ds^2_{\pi}|_{V^o_\delta}\lesssim\frac{d\rho_2^2}{\rho_2^2\log^2\rho_2\log^2(-\log\rho_2)}+\frac{ds^2_\pi|_{L_\delta}}{\log^2(-\log\rho_2)}+\pi^\ast\omega_0.
		\end{align*}
	\end{prop}
	\begin{proof}
		Let $S$ be a stratum of $\{Z_i\}$ and $x\in \widetilde{Z}$ such that $\pi(x)\in S$. Let $(z_1,\dots,z_n)$ be a holomorphic coordinate of $\widetilde{Z}$ near $x$ such that $E=\{z_1\cdots z_r=0\}$. Denote by $z_i=|z_i|e^{\sqrt{-1}\theta_i}$. By Lemma \ref{lem_Grauert_metric_in_coordinates} there is a quasi-isometry
		\begin{align*}
		ds^2_{\pi}&\sim\sum_{i=1}^r\frac{d|z_i|^2+|z_i|^2d\theta^2_i}{|z_i|^2\log^2|z_1\cdots z_r|^2\log^2(-\log|z_1\cdots z_r|^2)}\\\nonumber
		&+\frac{\sum_{i=r+1}^n dz_id\bar{z}_i}{-\log|z_1\cdots z_r|^2\log^2(-\log|z_1\cdots z_r|^2)}+\pi^\ast ds^2_0.
		\end{align*}
		Here $ds^2_0$ is a hermitian metric on $Z$.

		Let $|s|=|z_1^{b_1}\cdots z_r^{b_r}|u^{-1}$ for some invertible function $u$.
		Denote
		$$t_i:=\frac{b_i\log|z_i|^2}{\log|us|^2},\quad i=1,\dots,r.$$
		Then $(t_1,\dots,t_r)$ lies in the domain
		$$D_r:=\{(t_1,\dots,t_r)\in(0,1)^r|t_1+\cdots+t_r=1\}.$$
		Denoting $\xi=\log|us|^2$, we claim that
		\begin{align}\label{align_Gt_remove_mixterm}
		g:=\sum_{i=1}^rd(t_i\xi)^2\sim \sum_{i=1}^rt_i^2d\xi^2+\xi^2dt_i^2
		\end{align}
		as metrics on $(1,+\infty)\times D_r$.
		For this it suffices to show that there is a constant $0<\epsilon<1$ such that
		\begin{align*}
		\left|(\frac{\partial}{\partial \xi},v)_g\right|\leq\epsilon\|\frac{\partial}{\partial \xi}\|_g\|v\|_g, \quad \forall v\in T_{D_r}.
		\end{align*}
		Assume the converse, i.e. there are vectors $v_n=\sum_{i=1}^ra_{ni}\frac{\partial}{\partial t_i}\in T_{D_r,x_n}$ such that $$\sum_{i=1}^ra_{ni}=0,\quad\sum_{i=1}^r\|a_{ni}\|^2=1$$ and
		\begin{align}\label{align_mixedterm_goaway}
		(1-\frac{1}{n})\|\frac{\partial}{\partial \xi}\|_g\leq\|v_n\|^{-1} \left|(\frac{\partial}{\partial \xi},v_n)_g\right|\leq \|\frac{\partial}{\partial \xi}\|_g
		\end{align}
		for every $n$.
		Assume that $v_n$ converges to $0\neq v=\sum_{i=1}^ra_i\frac{\partial}{\partial t_i}\in T_{\overline{D_r},x}$ where $\sum_{i=1}^ra_i=0$ and $\sum_{i=1}^r\|a_i\|^2=1$, then the limit of (\ref{align_mixedterm_goaway}) gives
		\begin{align*}
        \left|\sum_{i=1}^r t_ia_i\right|=\left(\sum_{i=1}^r t_i^2\right)\left(\sum_{i=1}^r \|a_i\|^2\right).
        \end{align*}
        This equality holds only if there is $\lambda\in\bC$ such that $t_i=\lambda a_i$, $i=1,\dots,r$.
		This contradicts to the fact that $\sum_{i=1}^rt_i=1$, $\sum_{i=1}^ra_i=0$ and $\sum_{i=1}^r\|a_i\|^2=1$.

		As a consequence of (\ref{align_Gt_remove_mixterm}) we have that
		\begin{align*}
		&ds^2_{\pi}\sim\sum_{i=1}^r\frac{d(\log|z_i|^2)^2+d\theta^2_i}{\log^2|s|^2\log^2(-\log|s|^2)}
		+\frac{\sum_{i=r+1}^n dz_id\bar{z}_i}{-\log|s|^2\log^2(-\log|s|^2)}+\pi^\ast ds^2_0\\\nonumber
		&\sim\sum_{i=1}^r\frac{d(t_i\log(|us|^2))^2}{\log^2|s|^2\log^2(-\log|s|^2)}
		+\frac{-\log^{-1}|s|^2\sum_{i=1}^rd\theta^2_i+\sum_{i=r+1}^n dz_id\bar{z}_i}{-\log|s|^2\log^2(-\log|s|^2)}+\pi^\ast ds^2_0\\\nonumber
		&\sim\sum_{i=1}^r\frac{t_i^2d(\log(|s|^2))^2+\log^2 |s|^2dt_i^2}{\log^2|s|^2\log^2(-\log|s|^2)}
		+\frac{-\log^{-1}|s|^2\sum_{i=1}^rd\theta^2_i+\sum_{i=r+1}^n dz_id\bar{z}_i}{-\log|s|^2\log^2(-\log|s|^2)}+\pi^\ast ds^2_0\quad(\ref{align_Gt_remove_mixterm})\\\nonumber
		&\sim\frac{d|s|^2}{|s|^2\log^2|s|^2\log^2(-\log|s|^2)}
		+\frac{-\log|s|^2\sum_{i=r+1}^n dz_id\bar{z}_i+\log^2 |s|^2\sum_{i=1}^rdt_i^2+\sum_{i=1}^rd\theta^2_i}{\log^2|s|^2\log^2(-\log|s|^2)}+\pi^\ast ds^2_0.
		\end{align*}
		From this estimate we obtain the proposition.
	\end{proof}
	\begin{cor}\label{cor_Gt_complete_finite_vol}
		Every distinguished metric $ds^2$ is locally complete and locally of finite volume. I.e. for every point $x\in X$, there is a compact neighborhood $U$ such that $(U,ds^2)$ is complete and of finite volume.
	\end{cor}
	\begin{proof}
		Since the statement is local, we consider a germ of complex space and use the notations in \S \ref{section_setting}. By (\ref{align_expend_Gt}), near $Z\backslash Z^o=\{\psi_1=+\infty\}$ we have
		\begin{align*}
		|d\log\log(\psi_{1})|_{ds^2_\pi}\leq|d\log\log(\psi_{1})|_{\sqrt{-1}\ddbar\varphi_{\pi}}=\left|\frac{d\psi_1}{\psi_1\log\psi_1}\right|_{\sqrt{-1}\ddbar\varphi_{\pi}}\leq2.
		\end{align*}
		Since $\log\log(\psi_{1})=\log\log(-\log\prod_{i\in I}|s_i|^{2a_{i1}})$ is an exhaustive $C^\infty$ function on $\overline{U_\delta}\cap Z^o$ for $0<\delta\ll 1$,
		the local completeness follows from the Hopf-Rinow theorem.

		By Proposition \ref{prop_Grauert_asymptotic_behavior}, we get that
		\begin{align*}
		ds^2\lesssim\frac{d\rho_2^2}{\rho_2^2\log^2\rho_2\log^2(-\log \rho_2)}+\frac{ds^2|_{L_\delta}}{\log^2(-\log \rho_2)}+\omega_{\widetilde{Z}}.
		\end{align*}
		As a consequence,
		\begin{align*}
		\int_{V_\delta}{\rm vol}_{ds^2}\lesssim \int_{0}^{\delta}\frac{d\rho_2}{\rho_2\log \rho_2\log^2(-\log\rho_2)}\int{\rm vol}_{L_\delta}<\infty.
		\end{align*}
	\end{proof}
	\subsection{Consistency of local $L^2$-cohomologies}
    Let $0<a<\delta'<\delta\ll 1$ and
    $W=U^o_\delta\cap\{\rho_1>a\}$.
    Let us consider the diagram in the sense of Mayer-Vietoris
    \begin{align}\label{align_MV1}
    \xymatrix{
    0\ar[r]& D^\bullet_{\rm min}(U_{\delta'}^o\cap W)\ar[r]& D^\bullet_{\rm min}(U_{\delta'}^o)\oplus D^\bullet_{\rm min}(W)\ar[r]^-{r}& D^\bullet_{\rm min}(U_{\delta}^o)\ar[r]& 0
    }
    \end{align}
    where $r(\alpha,\beta)=\widetilde{\alpha}+\widetilde{\beta}$ is the sum of the zero extensions.
    Here we write $D^\bullet_{\rm min}(-):=D^\bullet_{\rm min}(-,\bV;ds^2_\pi,h)$ for short.

    We claim that (\ref{align_MV1}) is exact. The only nontrivial part is the surjectivity of $r$.
    Let $\eta_1,\eta_2\in C^\infty(0,\delta)$ be nonnegative functions such that ${\rm supp}(\eta_1)\subset(0,\delta')$, ${\rm supp}(\eta_2)\subset(a,\delta)$ and $\eta_1+\eta_2=1$. Denote $\lambda_1:=\rho_1^\ast\eta_1$ and $\lambda_2:=\rho_1^\ast\eta_2$.
    By Lemma \ref{lem_U_delta}, $\lambda_1$, $\lambda_2$, $d\lambda_1$, $d\lambda_2$ all have bounded norms on $U^o_{\delta}$. Hence for every $\alpha\in D^\bullet_{\rm min}(U^o_{\delta})$, we have
    \begin{align*}
    \lambda_1\alpha\in D^\bullet_{\rm min}(U^o_{\delta'}),\quad \lambda_2\alpha\in D^\bullet_{\rm min}(W)
    \end{align*}
    and
    $\widetilde{\lambda_1\alpha}+\widetilde{\lambda_2\alpha}=\alpha$.
    This proves the surjectivity of $r$.
	\begin{prop}\label{prop_cpt_vs_min}
		For every $0<\delta'<\delta\ll 1$ and every $k$, the canonical morphisms
		\begin{align}\label{align_iso_1}
		H^k_{(2),\rm min}(U^o_{\delta'},\bV;ds^2_{\pi},h)\to H^k_{(2),\rm min}(U^o_{\delta},\bV;ds^2_{\pi},h)
		\end{align}
		and
		\begin{align}\label{align_iso_2}
		\bH^k_{\rm cpt}(U_{\delta},\sD^\bullet_{Z,\bV;ds^2_{\pi},h})\to H^k_{(2),\rm min}(U^o_{\delta},\bV;ds^2_{\pi},h)
		\end{align}
		are isomorphisms.
	\end{prop}
	\begin{proof}
		Denote $H^i_{(2),\rm min}(-):=H^i_{(2),\rm min}(-,\bV;ds^2_\pi,h)$ for short.
		The exact sequence (\ref{align_MV1}) gives the $L^2$-Mayer-Vietoris sequence
		\begin{align}\label{align_MV_PD1}
		\xymatrix{
			\cdots\ar[r]& H^i_{(2),\min}(U_{\delta'}^o\cap W)\ar[r]& H^i_{(2),\min}(U_{\delta'}^o)\oplus H^i_{(2),\min}(W)\ar[r]& H^i_{(2),\min}(U_{\delta}^o)\ar[r]&\cdots.
		}
		\end{align}
		By Proposition \ref{prop_L2_Poincare_duality}, for every $b_1<b_2$ one has
		$$H^i_{(2),\min}\left((b_1,b_2),dt^2\right)\simeq H^{1-i}_{(2),\max}\left((b_1,b_2),dt^2\right)^\ast\simeq H^{1-i}\left([b_1,b_2]\right)^\ast\simeq
		\begin{cases}
		\bC, & i=1 \\
		0, & i\neq 1
		\end{cases}.
		$$
		Lemma \ref{lem_U_delta} and the $L^2$-Kunneth formula (Proposition \ref{prop_L2_Kunneth}) give
		\begin{align*}
		H^k_{(2),\min}(W)\simeq \bigoplus_{i+j=k}H^i_{(2),\min}\left((a,\delta),dt^2\right)\otimes H^j_{(2),\min}(M_{a})\simeq H^{k-1}_{(2),\min}(M_{a})
		\end{align*}
		and
		\begin{align*}
		H^k_{(2),\min}(U_{\delta'}^o\cap W)\simeq \bigoplus_{i+j=k}H^i_{(2),\min}\left((a,\delta'),dt^2\right)\otimes H^j_{(2),\min}(M_{a})\simeq H^{k-1}_{(2),\min}(M_{a}).
		\end{align*}
		Hence for every $k$ the map of extension by zero
		$$H^k_{(2),\min}(U_{\delta'}^o\cap W)\to H^k_{(2),\min}(W)$$
		is an isomorphism.
		Hence by (\ref{align_MV_PD1}) we obtain that (\ref{align_iso_1}) is an isomorphism.

		For (\ref{align_iso_2}),
		recall that
		\begin{align*}
		\Gamma_{\rm cpt}(U_{\delta}, \sD^\bullet_{X,\bV;ds^2_{\pi},h})=\bigcup_{\delta'<\delta}D^\bullet_{\rm min}(U^o_{\delta'})\subset D^\bullet_{\rm min}(U^o_{\delta}).
		\end{align*}
		Here $D^\bullet_{\rm min}(U^o_{\delta'})$ is naturally a subspace of $D^\bullet_{\rm min}(U^o_{\delta})$ via extending forms by zero.
		By (\ref{align_iso_1}), for every $0<\delta'<t\leq\delta\ll1$, the canonical map
		\begin{align*}
		D^\bullet_{\rm min}(U^o_{\delta'})\to D^\bullet_{\rm min}(U^o_{t})
		\end{align*}
		is a quasi-isomorphism. By taking direct limit of $t$, the natural map
		\begin{align*}
		D^\bullet_{\rm min}(U^o_{\delta'})\to\Gamma_{\rm cpt}(U_{\delta}, \sD^\bullet_{X,\bV;ds^2_{\pi},h})\simeq\varinjlim_{\delta'<t<\delta}D^\bullet_{\rm min}(U^o_{t})
		\end{align*}
		is a quasi-isomorphism.
		As a consequence the inclusion map
		\begin{align*}
		\Gamma_{\rm cpt}(U_{\delta}, \sD^\bullet_{X,\bV;ds^2_{\pi},h})\subset D^\bullet_{\rm min}(U^o_{\delta})
		\end{align*}
		is a quasi-isomorphism. This implies that (\ref{align_iso_2}) are isomorphisms.
	\end{proof}
	\section{Local Vanishings}
	\subsection{Modify the boundary $\partial U_{\delta}$}\label{section_modify_dsGt}
	Notations as in \S \ref{section_setting} and \S \ref{section_U_delta}.
	The incompleteness of $ds^2_{\pi}$ at $\partial U_\delta$ would cause some analytical difficulties (\cite[Proposition 2.1.1]{Hormander1965}).
	The aims of this subsection are to modify $ds^2_{\pi}$ at $\partial U_\delta$ to a complete metric and to reduce proving (\ref{align_vanishing_cpt}) to solving $\nabla$-equation in the Hilbert space of square integrable forms with respect to this complete K\"ahler metric and a weight.

	Let $0<\delta\ll1$.
	Denote $\psi_{\delta}=-4n\log(\delta^2-\|w\|^2)$.
	Then
	$ds^2_{\delta}:=\sqrt{-1}\partial\dbar(\varphi_{\pi}+\psi_\delta)$ is a complete metric on $U_\delta^o$ (Corollary \ref{cor_Gt_complete_finite_vol}).
	\begin{lem}\label{lem_i<n}
		Let $0\leq i\leq 2n$, $0<\delta'<\delta<1$ and $\alpha\in L_{(2)}^i(U_{\delta'}^o,\bV;ds^2_{\delta'},e^{\psi_{\delta'}}h)$. Denote by $\tilde{\alpha}$ the zero extension of $\alpha$ on $U_{\delta}^o$, i.e. $\tilde{\alpha}|_{U^o_{\delta'}}=\alpha$ and $\tilde{\alpha}|_{U^o_{\delta}\backslash U^o_{\delta'}}=0$. Then $\tilde{\alpha}\in L_{(2)}^i(U_{\delta}^o,\bV;ds^2_{\pi},h)$. Moreover if $\alpha\in D_{\rm max}^{i}(U_{\delta'}^o,\bV;ds^2_{\delta'},e^{\psi_{\delta'}}h)$, then $\tilde{\alpha}\in D^i_{\rm max}(U_{\delta}^o,\bV;ds^2_{\pi},h)$ and $\nabla\tilde{\alpha}=\widetilde{\nabla\alpha}$.
	\end{lem}
	\begin{proof}
		Assume that the eigenvalues of $ds^2_{\pi}$ measured by $\sqrt{-1}\partial\dbar\|w\|^2|_{U_{\delta}^o}$ are  $\eta_1,\dots,\eta_n$. Since $\sqrt{-1}\partial\bar{\partial}\|w\|^2|_{U^o_\delta}\lesssim ds_{\pi}^2$, we have $1\lesssim \eta_k$ for each $k$. Since
		\begin{align*}
		\partial\dbar\psi_{\delta'}=\frac{4n\partial\dbar\|w\|^2}{\delta'^2-\|w\|^2}+\frac{4n\partial\|w\|^2\wedge \dbar\|w\|^2}{(\delta'^2-\|w\|^2)^2},
		\end{align*}
		the eigenvalues of $\sqrt{-1}\partial\dbar\psi_{\delta'}$ measured by $\sqrt{-1}\partial\dbar\|w\|^2$ are
		$$\lambda'_k\sim\frac{4n}{\delta'^2-\|w\|^2},\quad k=1,\dots,N-1,\quad \frac{4n}{\delta'^2-\|w\|^2}\lesssim\lambda'_{N}\lesssim\frac{4n}{(\delta'^2-\|w\|^2)^2}.$$
		By Courant's minimax principle, the eigenvalues of $ds^2_{\delta'}$ measured by $\sqrt{-1}\partial\dbar\|w\|^2|_{U^o_\delta}$ are
		\begin{align*}
		\lambda_k\sim\eta_k+\frac{4n}{\delta'^2-\|w\|^2},\quad k=1,\dots,n-1
		\end{align*}
		and
		\begin{align*}
		\eta_n+\frac{4n}{\delta'^2-\|w\|^2}\lesssim\lambda_{n}\lesssim\eta_n+\frac{4n}{(\delta'^2-\|w\|^2)^2}.
		\end{align*}
		As a consequence, we obtain that for any form $\alpha$ on $U_{\delta'}^o$,
		\begin{align}\label{align_norm_Gt_vs_delta}
		\|\alpha\|_{ds^2_{\pi},h}\lesssim\|\alpha\|_{ds^2_{\delta'},e^{\psi_{\delta'}}h}.
		\end{align}
		This proves the first part. For the second part, let $\alpha\in D_{\rm max}^{i}(U_{\delta'}^o,\bV;ds^2_{\delta'},e^{\psi_{\delta'}}h)$.  Since $ds^2_{\delta'}$ is complete, there is a sequence $\{\alpha_k\}\subset A^i_{\rm cpt}(U^o_{\delta'},\bV)$ which converges to $\alpha$ under the graph norm with respect to $(ds^2_{\delta'},e^{\psi_{\delta'}}h)$.
		By (\ref{align_norm_Gt_vs_delta}) we obtain the convergences
		$$\widetilde{\alpha_k}\to\widetilde{\alpha},\quad\nabla(\widetilde{\alpha_k})=\widetilde{\nabla\alpha_k}\to \widetilde{\nabla\alpha}$$
		in the spaces $L_{(2)}^i(U_{\delta}^o,\bV;ds^2_{\pi},h)$ and $ L_{(2)}^{i+1}(U_{\delta}^o,\bV;ds^2_{\pi},h)$ respectively.
		Thus $$\nabla\widetilde{\alpha}=\widetilde{\nabla\alpha}\in L_{(2)}^{i+1}(U_{\delta}^o,\bV;ds^2_{\pi},h),$$
		so $\tilde{\alpha}\in D^i_{\rm max}(U_{\delta}^o,\bV;ds^2_{\pi},h)$.
	\end{proof}
	\subsection{An estimate for pluriharmonic bundle}
	The main result of this subsection is
	\begin{thm}\label{thm_L2-estimate_main}
		Let $(M,ds^2)$ be a complete K\"ahler manifold with $\omega$ its K\"ahler form. Let $(\cV,\nabla,h)$ be a pluriharmonic bundle on $M$ and let $\bV:={\rm Ker}\nabla$ be the associated local system. Let $\varphi$, $F$ be $C^\infty$ functions on $M$ and $C>0$ a constant. If $\omega=\sqrt{-1}\ddbar(F+C\varphi)$ and $|dF|^2_{\infty}<8C$, then for every $r<n:=\dim M$,
		\begin{align*}
		\|\alpha\|^2_{ds^2,e^{\varphi}h}\leq\frac{8C^2}{8C-|dF|^2_\infty}\left(\|\nabla\alpha\|^2_{ds^2,e^{\varphi}h}+\|\nabla^\ast_{e^\varphi h}\alpha\|_{ds^2,e^{\varphi}h}^2\right),\quad \forall\alpha\in {\rm Dom}^r\nabla\cap{\rm Dom}^r\nabla^\ast_{e^{\varphi}h}.
		\end{align*}
		Here $\nabla:L^\bullet_{(2)}(M,\bV;ds^2,e^\varphi h)\to L^{\bullet+1}_{(2)}(M,\bV;ds^2,e^\varphi h)$ is the unbounded operator in the sense of distribution and $\nabla^\ast_{e^{\varphi}h}$ is its Hilbert adjoint.
	\end{thm}
	\begin{proof}
		For any differential form $\xi$, let $L(\xi)$ be the operator defined by $L(\xi)(\alpha)=\xi\wedge\alpha$.
		Since
		\begin{align}\label{align_Kahler_Id3}
		[L(dF),\Lambda_\omega]=\sqrt{-1}L(d^cF)^\ast\quad\textrm{\cite[Chap. 5, (3.22)]{Wells1980}},
		\end{align}
		we obtain that
		\begin{align*}
		&[\nabla^c_{e^{\varphi}h},L(d^cF)^\ast]+[L(dF),\nabla^\ast_{e^{\varphi}h}]\\\nonumber
		=&[\nabla^c_{e^{\varphi}h},\sqrt{-1}[\Lambda_\omega,L(dF)]]+[L(dF),\sqrt{-1}[\Lambda_{\omega},\nabla^c_{e^{\varphi}h}]]\quad (\ref{align_Kahler_Id3}),(\ref{align_Kahler_id2})\\\nonumber
		=&\sqrt{-1}[\Lambda_\omega,[\nabla^c_{e^{\varphi}h},L(dF)]]\\\nonumber
		=&[L(\sqrt{-1}dd^cF),\Lambda_\omega].
		\end{align*}
		Denote $(-,-)_\varphi:=(-,-)_{ds^2,e^{\varphi}h}$ and $\|-\|_\varphi:=\|-\|_{ds^2,e^{\varphi}h}$ for simplicity. By Lemma \ref{lem_Bochner_formula_harmonic_bundle} and Lemma \ref{lem_twisted_curvature} we get that
		\begin{align*}
		\|\nabla\alpha\|^2_\varphi+\|\nabla^\ast_{e^{\varphi}h}\alpha\|^2_\varphi=\|\nabla^c_{e^{\varphi}h}\alpha\|^2_\varphi+\|(\nabla^c_{e^{\varphi}h})^\ast\alpha\|^2_\varphi+\left([\Lambda_{\omega},L(\sqrt{-1}dd^c\varphi)]\alpha,\alpha\right)_\varphi
		\end{align*}
		for every $\alpha\in A^r_{\rm cpt}(M,\bV)$. This implies that
		\begin{align*}
		&\left|([L(\sqrt{-1}dd^cF),\Lambda_\omega]\alpha,\alpha)_\varphi\right|\\\nonumber
		=&\left|([\nabla^c_{e^{\varphi}h},L(d^cF)^\ast]\alpha,\alpha)_\varphi+([L(dF),\nabla^\ast_{e^{\varphi}h}]\alpha,\alpha)_\varphi\right|\\\nonumber
		=&\left|(\alpha,d^cF\wedge(\nabla^c_{e^{\varphi}h})^\ast\alpha)_\varphi+(\nabla^c_{e^{\varphi}h}\alpha,d^cF\wedge\alpha)_\varphi+(dF\wedge \nabla^\ast_{e^{\varphi}h}\alpha,\alpha)_\varphi+(dF\wedge\alpha,\nabla\alpha)_\varphi\right|\\\nonumber
		\leq&|dF|_\infty\|\alpha\|_\varphi\left(\|\nabla^c_{e^{\varphi}h}\alpha\|_\varphi+\|(\nabla^c_{e^{\varphi}h})^\ast\alpha\|_\varphi+\|\nabla\alpha\|_\varphi+\|\nabla^\ast_{e^{\varphi}h}\alpha\|_\varphi\right)\\\nonumber
		\leq&\frac{|dF|^2_\infty}{4C}\|\alpha\|^2_\varphi+C\left(\|\nabla^c_{e^{\varphi}h}\alpha\|^2_\varphi+\|(\nabla^c_{e^{\varphi}h})^\ast\alpha\|^2_\varphi+\|\nabla\alpha\|^2_\varphi+\|\nabla^\ast_{e^{\varphi}h}\alpha\|^2_\varphi\right)\\\nonumber
		=&\frac{|dF|^2_\infty}{4C}\|\alpha\|^2_\varphi+C\left(2\|\nabla\alpha\|^2_\varphi+2\|\nabla^\ast_{e^{\varphi}h}\alpha\|^2_\varphi+([L(\sqrt{-1}dd^c\varphi),\Lambda_\omega]\alpha,\alpha)_\varphi\right).\\\nonumber
		\end{align*}
		Hence
		\begin{align*}
		\left([L(\sqrt{-1}dd^c(-F-C\varphi)),\Lambda_\omega]\alpha,\alpha\right)_\varphi-\frac{|dF|^2_\infty}{4C}\|\alpha\|^2_\varphi\leq 2C\left(\|\nabla\alpha\|^2_\varphi+\|\nabla^\ast_{e^{\varphi}h}\alpha\|^2_\varphi\right).
		\end{align*}
		Since $\omega=\sqrt{-1}\ddbar(F+C\varphi)$, we have that
		\begin{align*}
		\left([L(\sqrt{-1}dd^c(-F-C\varphi)),\Lambda_\omega]\alpha,\alpha\right)_\varphi=\left(-2[L(\omega),\Lambda_\omega]\alpha,\alpha\right)_\varphi=2(n-r)\|\alpha\|^2_\varphi.
		\end{align*}
		So
		\begin{align*}
		\left(2n-2r-\frac{|dF|^2_\infty}{4C}\right)\|\alpha\|^2_\varphi\leq 2C\left(\|\nabla\alpha\|^2_\varphi+\|\nabla^\ast_{e^{\varphi}h}\alpha\|^2_\varphi\right).
		\end{align*}
		Since $M$ is complete, $A^r_{\rm cpt}(M,\bV)$ is dense in ${\rm Dom}^r\nabla\cap{\rm Dom}^r\nabla^\ast_{e^{\varphi}h}$ under the graph norm $$\|-\|_\varphi+\|\nabla(-)\|_\varphi+\|\nabla^\ast_{e^{\varphi}h}(-)\|_\varphi.$$ The theorem is proved.
	\end{proof}
	\subsection{Local vanishing}
	The main result of this subsection is
	\begin{thm}\label{thm_local_vanishing}
		Notations as in \S \ref{section_setting} and \S \ref{section_U_delta}. Then for every $0<\delta\ll1$ and every $k\leq n$,
		\begin{align}\label{align_thm_local_vanishing}
		\bH^k_{\rm cpt}(U_{\delta},\sD^\bullet_{Z,\bV;ds^2_{\pi},h})=0.
		\end{align}
	\end{thm}
	The proof is postponed to the end of this subsection.
	We use the notations in \S \ref{section_setting},  \S \ref{section_U_delta} and \S\ref{section_modify_dsGt}. Recall that $\rho_2=|s|$ and $\psi_\delta=-4n\log(\delta^2-\|w\|^2)$.
	By (\ref{align_rho2_psh})
	\begin{align*}
	\varphi_{\delta,\epsilon}:=\epsilon\log\rho_2+\psi_\delta,\quad\forall\epsilon\geq 0
	\end{align*}
	is a plurisubharmonic function.  By assumption  (\ref{align_metric}) we have
	\begin{align*}
	ds^2_{\pi}\sim\sqrt{-1}\ddbar (\varphi_{\pi}+\|w\|^2).
	\end{align*}
	Let
	\begin{align*}
	ds^2_{\delta,\epsilon}:=\sqrt{-1}\ddbar(\varphi_{\pi}+\varphi_{\delta,\epsilon}), \quad\forall \epsilon\in[0,1).
	\end{align*}
	Then $ds^2_{\delta,\epsilon}$ are complete metrics that locally uniformly converge to $\sqrt{-1}\ddbar(\varphi_{\pi}+\psi_\delta)=ds^2_{\delta}$ over $U^o_\delta$ when $\epsilon\to 0$.
	For simplicity we denote $(-,-)_{\delta,\epsilon}:=(-,-)_{ds^2_{\delta,\epsilon},e^{\varphi_{\delta,\epsilon}}h}$ and $\|-\|_{\delta,\epsilon}:=\|-\|_{ds^2_{\delta,\epsilon},e^{\varphi_{\delta,\epsilon}}h}$.
	\begin{lem}\label{lem_twist_epsilon_good}
		For every $\epsilon\in(0,1)$ and every $k$ one has the bounded inclusion
		\begin{align*}
		L_{(2)}^k(U_\delta^o,\bV;ds^2_\delta,e^{\psi_\delta}h)\subset L_{(2)}^k(U_\delta^o,\bV;ds^2_{\delta,\epsilon},e^{\varphi_{\delta,\epsilon}}h).
		\end{align*}
	\end{lem}
	\begin{proof}
		Let $\epsilon\in(0,1)$.
		By Lemma \ref{lem_Grauert_metric_in_coordinates} we get that
		\begin{align*}
		\log^{-2}\rho_2ds^2_{\delta,\epsilon}\lesssim ds^2_\delta\lesssim ds^2_{\delta,\epsilon}.
		\end{align*}
		As a consequence
		\begin{align}\label{estimate1}
		\|\alpha\|^2_{\delta,\epsilon}=\int_{U^o_\delta}|\alpha|^2_{\delta,\epsilon}e^{\varphi_{\delta,\epsilon}}{\rm vol}_{ds^2_{\delta,\epsilon}}\lesssim\int_{U^o_\delta}|\alpha|^2_{\delta,0}e^{\psi_\delta}\rho_2^\epsilon(\log\rho_2)^{2n}{\rm vol}_{ds^2_{\delta}}\lesssim\|\alpha\|^2_{\delta,0}.
		\end{align}
		This proves the lemma.
	\end{proof}

	\begin{lem}\label{lem_est_dsdeltaepsilon}
		Let $r\leq n$ and $\alpha\in L_{(2)}^r(U_\delta^o,\bV;ds^2_{\delta,\epsilon},e^{\varphi_{\delta,\epsilon}}h)$. Assume that $\alpha=\nabla\beta$ for some $\beta\in L_{(2)}^{r-1}(U_\delta^o,\bV;ds^2_{\delta,\epsilon},e^{\varphi_{\delta,\epsilon}}h)$. Then for every $u\in A^r_{\rm cpt}(U_{\delta}^o,\bV)$ there is an estimate
		\begin{align*}
		|(\alpha,u)_{\delta,\epsilon}|\leq \sqrt{2}\|\alpha\|_{\delta,\epsilon}\|\nabla^\ast_{e^{\varphi_{\delta,\epsilon}}h} u\|_{\delta,\epsilon}.
		\end{align*}
	\end{lem}
	\begin{proof}
		Since $\varphi_{\pi}$, $\log\rho_2$ and $\psi_\delta$ are plurisubharmonic functions, by (\ref{align_grad_norm_Gt}) we have
		\begin{align*}
		|d\varphi_{\pi}|^2_{ds^2_{\delta,\epsilon}}\leq|d\varphi_{\pi}|^2_{\sqrt{-1}\ddbar \varphi_{\pi}}\leq 4.
		\end{align*}
		Hence Theorem \ref{thm_L2-estimate_main} (taking $C=1$, $F=\varphi_\pi$ and $\varphi=\varphi_{\delta,\epsilon}$) gives us the estimate
		\begin{align*}
		\|v\|^2_{\delta,\epsilon}\leq 2\left(\|\nabla v\|^2_{\delta,\epsilon}+\|\nabla^\ast_{e^{\varphi_{\delta,\epsilon}}h}v\|_{\delta,\epsilon}^2\right),\quad \forall v\in {\rm Dom}^{r-1}\nabla\cap{\rm Dom}^{r-1}\nabla^\ast_{e^{\varphi_{\delta,\epsilon}}h}.
		\end{align*}
		Here we denote $${\rm Dom}^{k}\nabla:={\rm Dom}\nabla\cap L^k_{(2)}(U_\delta^o,\bV;ds^2_{\delta,\epsilon},e^{\varphi_{\delta,\epsilon}}h),\quad {\rm Im}^{k}\nabla:={\rm Im}\nabla\cap L^k_{(2)}(U_\delta^o,\bV;ds^2_{\delta,\epsilon},e^{\varphi_{\delta,\epsilon}}h),\quad \textrm{etc.}$$
		As a consequence,
		\begin{align*}
		\|v\|_{\delta,\epsilon}\leq\sqrt{2}\|\nabla v\|_{\delta,\epsilon},\quad \forall v\in {\rm Dom}^{r-1}\nabla\cap({\rm Ker}^{r-1}\nabla)^\bot.
		\end{align*}
		By \cite[Theorem 1.1.1]{Hormander1965}, ${\rm Im}^{r}\nabla$ is closed and there is an orthogonal decomposition
		\begin{align*}
		L^r_{(2)}(U^o_\delta,\bV;ds^2_{\delta,\epsilon},e^{\varphi_{\delta,\epsilon}}h)={\rm Im}^{r}\nabla\oplus{\rm Ker}^r\nabla^{\ast}_{e^{\varphi_{\delta,\epsilon}}h}.
		\end{align*}
		Hence there is a unique decomposition
		\begin{align*}
		u=u_1+u_2,\quad u_1\in {\rm Im}^r\nabla, u_2\in {\rm Ker}^r\nabla^\ast_{e^{\varphi_{\delta,\epsilon}}h}.
		\end{align*}
		Since $u\in A^r_{\rm cpt}(U_{\delta}^o,\bV)$, $\nabla^\ast_{e^{\varphi_{\delta,\epsilon}}h}u_1=\nabla^\ast_{e^{\varphi_{\delta,\epsilon}}h} u$ is weighted square integrable. Hence by \cite[Theorem 1.1.1]{Hormander1965} one gets $$\|u_1\|_{\delta,\epsilon}\leq \sqrt{2}\|\nabla^\ast_{e^{\varphi_{\delta,\epsilon}}h} u_1\|_{\delta,\epsilon}.$$ Therefore
		\begin{align*}
		|(\alpha,u)_{\delta,\epsilon}|=|(\alpha,u_1)_{\delta,\epsilon}|\leq\|\alpha\|_{\delta,\epsilon}\|u_1\|_{\delta,\epsilon}\leq \sqrt{2}\|\alpha\|_{\delta,\epsilon}\|\nabla^\ast_{e^{\varphi_{\delta,\epsilon}}h} u_1\|_{\delta,\epsilon}= \sqrt{2}\|\alpha\|_{\delta,\epsilon}\|\nabla^\ast_{e^{\varphi_{\delta,\epsilon}}h} u\|_{\delta,\epsilon}.
		\end{align*}
	\end{proof}
	\begin{lem}\label{lem_solve_d_integral}
		Let $0<\delta''<\delta'<\delta\ll1$ so that $\overline{\pi^{-1}U_{\delta''}}\subset V_{\delta'}$ and $\overline{V_{\delta'}}\subset \pi^{-1}{U_{\delta}}$.
		Let $0<k\leq 2n$ and $\alpha\in D^k_{\rm max}(U^o_{\delta},\bV;ds^2_{\delta},e^{\psi_\delta}h)$ such that $\nabla\alpha=0$ and ${\rm supp}\alpha$ is compact in $U_{\delta''}$. For every $\epsilon\in(0,1)$, there is $\beta\in D^{k-1}_{\rm max}(U^o_\delta,\bV;ds^2_{\delta,\epsilon},e^{\varphi_{\delta,\epsilon}}h)$ such that $\nabla\beta=\alpha$.
	\end{lem}
	\begin{proof}
		Recall the fiberation (\ref{align_fiberation_V_delta}):
		\begin{align*}
		V^o_{\delta'}\simeq (0,\ell(\delta'))\times L_{\delta'},\quad
		\end{align*}
		whose first component is $\rho_2$. Fix an isomorphism of the local systems
		\begin{align*}
		\bV\simeq q^\ast \bV|_{L_{\delta'}}
		\end{align*}
		where $q:(0,\ell(\delta'))\times L_{\delta'}\to (0,\ell(\delta'))$ is the projection.
		By (\ref{align_rho2_psh}), $\sqrt{-1}\ddbar\varphi_{\delta,\epsilon}$ represents a non-degenerate metric on $\overline{V_{\delta'}}$. Hence by Proposition \ref{prop_Grauert_asymptotic_behavior}, there is a constant $C\geq 1$ such that for every $\epsilon\in[0,1)$ one has
		\begin{align}\label{align_ds2_deltaepsilon_1}
		\frac{d\rho_2^2}{\rho_2^2\log^2\rho_2\log^2(-\log \rho_2)}+\frac{ds^2_\pi|_{L_{\delta'}}}{\log^2\rho_2\log^2(-\log \rho_2)}
		\leq C^2 ds^2_{\delta,\epsilon}|_{V^o_{\delta'}}
		\end{align}
		and
		\begin{align}\label{align_ds2_deltaepsilon_2}
		C^{-2}ds^2_{\delta,\epsilon}|_{V^o_{\delta'}}\leq \frac{d\rho_2^2}{\rho_2^2\log^2\rho_2\log^2(-\log \rho_2)}+ds^2_\pi|_{L_{\delta'}}.
		\end{align}
		Let $\alpha\in A^k(U^o_{\delta},\bV)$ such that ${\rm supp}\alpha$ is compact in $U^o_{\delta''}$. Let
		$\alpha|_{V^o_{\delta'}}=d\rho_2\wedge\alpha_0+\alpha_1$ be the decomposition so that
		$$\iota_{\frac{\partial}{\partial \rho_2}}\alpha_0=0,~~\iota_{\frac{\partial}{\partial \rho_2}}\alpha_1=0.$$
		Denote
		\begin{align*}
		K\alpha:=\int_{\rho_2}^{\ell(\delta')}\alpha_0(t,\cdot)dt.
		\end{align*}
		Recall that $\nabla(\alpha'\otimes e)=d\alpha'\otimes e$ for $\alpha'\in A^\ast(U^o_{\delta})$ and $e\in\Gamma(U^o_{\delta},\bV)$.
		Then
		\begin{align}\label{align_K_homotopy_to_Id}
		K\nabla\alpha+\nabla K\alpha=-\alpha.
		\end{align}
		Denote $$N_1(\epsilon'):=\max_{0\leq s\leq\ell(\delta')}s^{\epsilon'}\left(-\log s\log(-\log s)\right)^{2n-1},\quad \forall\epsilon'>0.$$
		By (\ref{align_ds2_deltaepsilon_1}) and (\ref{align_ds2_deltaepsilon_2}) we obtain that
		\begin{align}\label{align_|ds|_asymtopic}
		C^{-1}\rho_2\leq|d\rho_2|_{ds^2_{\delta,\epsilon}}\leq CN_1(\epsilon')\rho_2^{1-\epsilon'},\quad \forall \epsilon\geq0, \forall\epsilon'>0
		\end{align}
		and
		\begin{align}\label{align_|alpha0|_asymtopic}
		\rho_2^{\epsilon'}|\alpha_0|_{ds^2_{\delta,\epsilon},h}\leq C^{2n-1}N_1(\epsilon')|\alpha_0|_{ds^2_{\pi}|_{L_{\delta'}},h}\leq C^{4n-2}N_1(\epsilon')|\alpha_0|_{ds^2_{\delta},h},\quad\forall \epsilon>0, \forall \epsilon'>0.
		\end{align}
		For every $t\in(0,\ell(\delta'))$, denote $M_t:=\{t\}\times L_{\delta'}$.
		Denote by $g_{t,\epsilon}:=ds^2_{\delta,\epsilon}|_{M_t}$, $h_t:=h|_{M_t}$ and denote by ${\rm vol}_{t,\epsilon}$ the volume form of  $(M_t,g_{t,\epsilon})$. For every form $\gamma$, $|\gamma|_{g_{t,\epsilon},h_t}:=\big|\gamma|_{M_t}\big|_{g_{t,\epsilon},h_t}$ is a function on $t\in(0,\ell(\delta'))$. By (\ref{align_ds2_deltaepsilon_1}) and (\ref{align_ds2_deltaepsilon_2}), for every $\epsilon\geq 0$, every $\epsilon'>0$ and every $0<t<t'<\ell(\delta')<1$, one has
		\begin{align}\label{align_vol_asymtopic}
		t^{\epsilon'}{\rm vol}_{t,\epsilon}\leq C^{2n-1}t^{\epsilon'}{\rm vol}_{ds^2_\pi|_{L_{\delta'}}}\leq C^{4n-2}N_1(\epsilon'){\rm vol}_{t,0}
		\end{align}
		and
		\begin{align}\label{align_vol_asymtopic2}
		t^{\epsilon'}{\rm vol}_{t,\epsilon}\leq C^{2n-1}t'^{\epsilon'}{\rm vol}_{ds^2_\pi|_{L_{\delta'}}}
		\leq C^{4n-2} N_1(\epsilon'){\rm vol}_{t',\epsilon}.
		\end{align}
		Since $h$ is $0$-tame (\ref{align_defn_0tame2}),
		$$N_2(\epsilon'):=\max\left\{s^{\epsilon'}\frac{|v|_{h_s}}{|v|_{h_{\ell(\delta')}}},s^{\epsilon'}\frac{|v|_{h_{\ell(\delta')}}}{|v|_{h_s}}\bigg|0\leq s\leq\ell(\delta'),0\neq v\in \bV\right\}$$
		is finite for every $\epsilon'>0$.
		Hence by (\ref{align_|alpha0|_asymtopic}) we have
		\begin{align*}
		t^{\epsilon'}|\alpha_0|_{g_{t,\epsilon},h_{t}}&\leq C^{2n-1}N_1(\frac{\epsilon'}{3})N_2(\frac{\epsilon'}{3})t^{\frac{\epsilon'}{3}}|\alpha_0|_{ds^2_\pi|_{L_{\delta'}},h_{\ell(\delta')}}\\\nonumber
		&\leq C^{2n-1}N_1(\frac{\epsilon'}{3})N_2(\frac{\epsilon'}{3})t'^{\frac{\epsilon'}{3}}|\alpha_0|_{ds^2_\pi|_{L_{\delta'}},h_{\ell(\delta')}}\\\nonumber
		&\leq C^{4n-2}N_1(\frac{\epsilon'}{3})N_2(\frac{\epsilon'}{3})^2|\alpha_0|_{g_{t',\epsilon},h_{t'}}
		\end{align*}
		for every $\epsilon>0$, every $\epsilon'>0$ and every $0<t<t'<\ell(\delta')$.
		It follows that for every $\epsilon>0$,
		\begin{align}\label{align_ds2_s=t}
		&\left|\int_\rho^{\ell(\delta')}\alpha_0dt\right|^2_{g_{\rho,\epsilon},h_\rho}\\\nonumber
		\leq&\left(\int_{\rho}^{\ell(\delta')}|\alpha_0|^2_{g_{\rho,\epsilon},h_{\rho}}t^{1+\epsilon/2}dt\right)\left(\int_{\rho}^{\ell(\delta')}t^{-1-\epsilon/2}dt\right)\\\nonumber
		\leq& C^{8n-4}N_1(\frac{\epsilon}{24})^2N_2(\frac{\epsilon}{24})^4\rho^{-\frac{\epsilon}{4}}\left(\int_{\rho}^{\ell(\delta')}|\alpha_0|^2_{g_{t,\epsilon},h_t}t^{1+\epsilon/2}dt\right)\left(\int_{\rho}^{\ell(\delta')}t^{-1-\epsilon/2}dt\right).
		\end{align}
		Hence for every $\epsilon\in(0,1)$, we have
		\begin{align*}
		&\int_{V_{\delta'}}\left|\int_{\rho}^{\ell(\delta')}\alpha_0dt\right|^2_{g_{\rho,\epsilon},h_\rho}\rho^{\epsilon}{\rm vol}_{ds^2_{\delta,\epsilon}}\\\nonumber
		=&\int_0^{\ell(\delta')}\int_{M_\rho}\left|\int_\rho^{\ell(\delta')}\alpha_0dt\right|^2_{g_{\rho,\epsilon},h_\rho}{\rm vol}_{\rho,\epsilon}\rho^{\epsilon}|d\rho|^{-1}_{ds^2_{\delta,\epsilon}}d\rho\\\nonumber
		\lesssim&\int_0^{\ell(\delta')}\int_{M_\rho}\left(\int_{\rho}^{\ell(\delta')}|\alpha_0|^2_{g_{t,\epsilon},h_t}t^{1+\epsilon/2}dt\right)\left(\int_{\rho}^{\ell(\delta')}t^{-1-\epsilon/2}dt\right){\rm vol}_{\rho,\epsilon}\rho^{\frac{3\epsilon}{4}}|d\rho|^{-1}_{ds^2_{\delta,\epsilon}}d\rho\quad(\ref{align_ds2_s=t})\\\nonumber
		\lesssim&\int_0^{\ell(\delta')}\int_{M_\rho}\left(\int_{\rho}^{\ell(\delta')}|\alpha_0|^2_{g_{t,\epsilon},h_t}t^{1+\epsilon/2}dt\right){\rm vol}_{\rho,\epsilon}\rho^{\epsilon/4}|d\rho|^{-1}_{ds^2_{\delta,\epsilon}}d\rho\\\nonumber
		\lesssim&\int_0^{\ell(\delta')}\left(\int_{\rho}^{\ell(\delta')}\int_{M_t}|\alpha_0|^2_{g_{t,\epsilon},h_t}{\rm vol}_{t,\epsilon}t^{1+\epsilon/2}dt\right)\rho^{\epsilon/8}|d\rho|^{-1}_{ds^2_{\delta,\epsilon}}d\rho\quad(\ref{align_vol_asymtopic2})\\\nonumber
		=&\int_{0}^{\ell(\delta')}\int_{M_t}|\alpha_0|^2_{g_{t,\epsilon},h_t}{\rm vol}_{t,\epsilon}t^{1+\epsilon/2}dt\int_{0}^{t}s^{\epsilon/8}|ds|^{-1}_{ds^2_{\delta,\epsilon}}ds\\\nonumber
		\lesssim&\int_{0}^{\ell(\delta')}\int_{M_t}|\alpha_0|^2_{g_{t,\epsilon},h_t}{\rm vol}_{t,\epsilon}t^{1+5\epsilon/8}dt\quad(\ref{align_|ds|_asymtopic})\\\nonumber
		\lesssim&\int_{0}^{\ell(\delta')}\int_{M_t}|dt\wedge\alpha_0|^2_{g_{t,\epsilon},h_t}{\rm vol}_{t,\epsilon}t^{-1+5\epsilon/8}dt\quad(\ref{align_|ds|_asymtopic})\\\nonumber
		\lesssim&\int_{0}^{\ell(\delta')}\int_{M_t}|\alpha|^2_{ds^2_\delta}{\rm vol}_{t,0}t^{\epsilon/4}|dt|^{-1}_{ds^2_\delta}dt\quad(\ref{align_|ds|_asymtopic}),  (\ref{align_|alpha0|_asymtopic}), (\ref{align_vol_asymtopic})\\\nonumber
		\lesssim&\int_{V_{\delta'}}|\alpha|^2_{ds^2_\delta}{\rm vol}_{ds^2_\delta}.
		\end{align*}
		Here the controlling constants of $\lesssim$ only depend on $\epsilon$ and $\delta'$. Since $e^{\psi_\delta}\sim 1$ over $\overline{V_{\delta'}}$, we obtain that
		\begin{align*}
		\|K\alpha\|_{\delta,\epsilon}\leq C'\|\alpha\|_{\delta,0},\quad \forall\alpha\in A_{\rm cpt}(U^o_{\delta''},\bV),
		\end{align*}
		for some constant $C'>0$ depending only on $\epsilon$, $\delta$, $\delta'$ and $\delta''$.

		Hence by (\ref{estimate1}) and  (\ref{align_K_homotopy_to_Id}) we obtain that
		\begin{align*}
		\|K\alpha\|_{\delta,\epsilon}+\|\nabla K\alpha\|_{\delta,\epsilon}
		\leq\|K\alpha\|_{\delta,\epsilon}+\|\alpha\|_{\delta,\epsilon}+\|K\nabla \alpha\|_{\delta,\epsilon}
		\leq (C'+1)(\|\alpha\|_{\delta,0}+\|\nabla \alpha\|_{\delta,0}).
		\end{align*}
		Now let $\alpha\in D^k_{\rm max}(U_{\delta}^o,\bV;ds^2_{\delta},e^{\psi_\delta}h)$ such that $\nabla \alpha=0$ and ${\rm supp}\alpha$ is compact in $U_{\delta''}$. Since $ds^2_\delta$ is a complete metric, there are $\alpha_n\in A^k_{\rm cpt}(U_{\delta''}^o,\bV)$, $n\in\bZ_{>0}$, such that
		\begin{align}\label{align_graph_convergence}
		\alpha_n\to\alpha,\quad \nabla\alpha_n\to \nabla\alpha=0
		\end{align}
		under the $L^2$-norm with respect to $ds^2_{\delta}$ and $e^{\psi_{\delta}} h$. By Lemma \ref{lem_twist_epsilon_good}, (\ref{align_graph_convergence}) are also convergent sequences with respect to $ds^2_{\delta,\epsilon}$ and $e^{\varphi_{\delta,\epsilon}}h$.
		By the boundedness of $K$, we obtain that
		$$\|K\alpha_n-K\alpha_m\|_{\delta,\epsilon}\leq C'\|\alpha_n-\alpha_m\|_{\delta,0}.$$
		Hence $K\alpha_n$ is a Cauchy sequence. Denote by $\beta\in L_{(2)}^{k-1}(U_{\delta}^o,\bV;ds_{\delta,\epsilon}^2,e^{\varphi_{\delta,\epsilon}}h)$ its limit. Applying (\ref{align_K_homotopy_to_Id}) on $\alpha_n$, i.e.
		$$\nabla K\alpha_n+K\nabla \alpha_n=-\alpha_n$$
		and taking its limit
		we obtain the convergent sequences
		\begin{align*}
		K\alpha_n\to\beta,\quad \nabla K\alpha_n\to -\alpha
		\end{align*}
		with respect to the metrics $ds^2_{\delta,\epsilon}$ and $e^{\varphi_{\delta,\epsilon}} h$. Hence $\beta\in D^{k-1}_{\rm max}(U^o_\delta,\bV;ds^2_{\delta,\epsilon},e^{\varphi_{\delta,\epsilon}}h)$ and
		$\nabla \beta=-\alpha$. This proves the lemma.
	\end{proof}
	Now we are ready to prove the main result of this subsection.
	\begin{proof}[Proof of Theorem \ref{thm_local_vanishing}]
		Let $k\leq n$ and $\alpha\in D^{k}_{\rm max}(U_\delta^o,\bV;ds^2_{\pi},h)$ such that $\nabla \alpha=0$. Since $\{\pi^{-1}U_\delta\}$ and $\{V_\delta\}$ are cofinal systems of neighborhoods of $E_{\{0\}}$, we may choose $0<\delta_2<\delta'<\delta_1<\delta\ll1$ so that $\overline{\pi^{-1}U_{\delta_2}}\subset V_{\delta'}$, $\overline{V_{\delta'}}\subset \pi^{-1}{U_{\delta_1}}$ and $\overline{U_{\delta_1}}\subset U_\delta$. The proof is divided into two steps.

		{\bf Step 1:} Assume that
		${\rm supp}\alpha$ is compact in $U_{\delta_2}$. By Lemma \ref{lem_twist_epsilon_good},
		$$\alpha\in D^{k}_{\rm max}(U_{\delta_1}^o,\bV;ds^2_{\delta_1},e^{\psi_{\delta_1}}h)\subset L_{(2)}^k(U_{\delta_1}^o,\bV;ds^2_{\delta_1,\epsilon},e^{\varphi_{\delta_1,\epsilon}}h)$$
		for every $\epsilon\in(0,1)$.
		By Lemma \ref{lem_est_dsdeltaepsilon} and Lemma \ref{lem_solve_d_integral} we have the estimate
		\begin{align*}
		|(\alpha,u)_{\delta_1,\epsilon}|^2\leq \sqrt{2}\|\alpha\|_{\delta_1,\epsilon}\|\nabla^\ast_{e^{\varphi_{\delta_1,\epsilon}}h} u\|^2_{\delta_1,\epsilon},\quad \forall u\in A^k_{\rm cpt}(U_{\delta_1}^o,\bV).
		\end{align*}
		By taking $\epsilon\to0$ we obtain that
		\begin{align*}
		|(\alpha,u)_{\delta_1,0}|^2\leq \sqrt{2}\|\alpha\|_{\delta_1,0}\|\nabla^\ast_{e^{\psi_{\delta_1}}h} u\|^2_{\delta_1,0},\quad \forall u\in A^k_{\rm cpt}(U_{\delta_1}^o,\bV).
		\end{align*}
		Since $ds^2_{\delta_1}$ is complete, this implies the estimate
		\begin{align*}
		|(\alpha,u)_{\delta_1,0}|^2\leq \sqrt{2}\|\alpha\|_{\delta_1,0}\|\nabla^\ast_{e^{\psi_{\delta_1}}h} u\|^2_{\delta_1,0},\quad\forall u\in {\rm Dom}^k\nabla^\ast_{e^{\psi_{\delta_1}}h}.
		\end{align*}
		Therefore by [\cite{Ohsawa2018}, Lemma 2.1] there is $\beta\in D^{k-1}_{\rm max}(U_{\delta_1}^o,\bV;ds^2_{\delta_1},e^{\psi_{\delta_1}}h)$ such that $\nabla\beta=\alpha$. Let $\widetilde{\beta}$ be the zero extension of $\beta$ on $U^o_\delta$. By Lemma \ref{lem_i<n}, $$\widetilde{\beta}\in D^{k-1}_{\rm max}(U_{\delta}^o,\bV;ds^2_{\pi},h)$$ and
		$\nabla\widetilde{\beta}=\alpha$. This proves (\ref{align_thm_local_vanishing}) for the first case.

		{\bf Step 2:}
		By Proposition \ref{prop_cpt_vs_min}, the 'extension by zero' map
		\begin{align*}
		\bH^k_{\rm cpt}(U_{\delta_2},\sD^\bullet_{Z,\bV;ds^2_{\pi},h})\to \bH^k_{\rm cpt}(U_{\delta},\sD^\bullet_{Z,\bV;ds^2_{\pi},h})
		\end{align*}
		is an isomorphism for every $k$. Hence every cohomology class in $\bH^k_{\rm cpt}(U_{\delta},\sD^\bullet_{Z,\bV;ds^2_{\pi},h})$ can be represented by some locally $L^2$-integrable $\nabla$-closed form $\alpha$ whose support is compact in $U_{\delta_2}$. Therefore by Step 1 we obtain the vanishing (\ref{align_thm_local_vanishing}).
	\end{proof}
	\section{Proof of the Main Theorem}
	We are ready to prove the main theorem (Theorem \ref{thm_main_harmonic}) of this paper. Before that let us recall the statements used in the bi-induction.
	\begin{thm}\label{thm_main_harmonic_proof}
		Let $X$ be a complex space of pure dimension and $X^o\subset X_{\rm reg}$ a dense Zariski open subset. Let $ds^2$ be a distinguished metric associated to $(X,X^o)$ and $(\bV,h)$ a $0$-tame harmonic bundle on $(X,X^o)$. Then
		\begin{enumerate}
			\item $\sD^\bullet_{X,\bV;ds^2,h}$ is a complex of fine sheaves.
			\item There is a canonical quasi-isomorphism
			$$\sD^\bullet_{X,\bV;ds^2,h}\simeq IC_X(\bV)$$ in $D(X)$.
		\end{enumerate}
	\end{thm}
    \begin{thm}\label{thm_main_L2_local_duality_proof}
    	Notations as in Theorem \ref{thm_main_harmonic_proof}. Let $x\in X$ be a point. Assume that
    	$$\varinjlim_{x\in U}\left(\bH^k_{\rm cpt}(U,\sD^\bullet_{X,\bV^\ast;ds^2,h^\ast})\right)^\ast$$
    	is finite dimensional for every $0\leq k\leq\dim X$ and
    	$$\varinjlim_{x\in U}\left(\bH^{\dim X}_{\rm cpt}(U,\sD^\bullet_{X,\bV;ds^2,h})\right)^\ast$$
    	is finite dimensional.
    	Then for every $0\leq k\leq \dim X$ there is an isomorphism
    	\begin{align}\label{align_weak_local_duality}
    	\varinjlim_{x\in U}\bH^{2\dim X-k}(U,\sD^\bullet_{X,\bV;ds^2,h})\simeq \varinjlim_{x\in U}\left(\bH^k_{\rm cpt}(U,\sD^\bullet_{X,\bV^\ast;ds^2,h^\ast})\right)^\ast
    	\end{align}
    	where $U$ ranges over all open neighborhoods of $x$.
    \end{thm}
Denote the statements
\begin{itemize}
	\item ${\bf IC}_n$: Theorem \ref{thm_main_harmonic_proof} holds when $\dim X\leq n$.
	\item ${\bf VD}_n$: Theorem \ref{thm_main_L2_local_duality_proof} holds when $\dim X\leq n$.
\end{itemize}
The proof of Theorem \ref{thm_main_harmonic} (= Theorem \ref{thm_main_harmonic_proof}) follows a bi-induction procedure:
\begin{description}
	\item[Initial Step] ${\bf IC}_0$. This case is trivial.
	\item[Bi-induction Step] ${\bf IC}_{n-1}\Rightarrow {\bf VD}_{n}$, ${\bf VD}_{n}\Rightarrow {\bf IC}_{n}$.
\end{description}
\subsection{Proof of ${\bf IC}_{n-1}\Rightarrow {\bf VD}_{n}$}
Since the problem is local, we assume that $x=0\in X\subset \bC^N$ and $\dim X\leq n$. We use the notations in \S \ref{section_setting} and \S\ref{section_U_delta}. By Proposition \ref{prop_cpt_vs_min} the natural maps
\begin{align*}
\left(H^k_{(2),\rm min}(U^o_\delta,\bV^\ast;ds^2,h^\ast)\right)^\ast\to\left(\bH^k_{\rm cpt}(U_\delta,\sD^\bullet_{X,\bV^\ast;ds^2,h^\ast})\right)^\ast\to\varinjlim_{x\in U}\left(\bH^k_{\rm cpt}(U,\sD^\bullet_{X,\bV^\ast;ds^2,h^\ast})\right)^\ast
\end{align*}
are isomorphisms for every $0<\delta\ll 1$ and every $0\leq k\leq 2\dim X$.
By Proposition \ref{prop_L2_Poincare_duality} there are dualities
\begin{align*}
H^{2\dim X-k}_{(2),\rm max}(U^o_\delta,\bV;ds^2,h)\simeq \left(\bH^k_{\rm cpt}(U_\delta,\sD^\bullet_{X,\bV^\ast;ds^2,h^\ast})\right)^\ast
\end{align*}
when $k<\dim X$ and a duality
\begin{align*}
H^{\dim X}_{(2),\rm max}(U^o_\delta,\bV;ds^2,h)\simeq \left(\bH^{\dim X}_{\rm cpt}(U_\delta,\sD^\bullet_{X,\bV^\ast;ds^2,h^\ast})\right)^\ast
\end{align*}
whenever
$H^{\dim X}_{(2),\rm max}(U^o_\delta,\bV;ds^2,h)$
is finite dimensional.
Taking direct limits we obtain the duality (\ref{align_weak_local_duality}).

The rest of the proof is devoted to show that $H^{\dim X}_{(2),\rm max}(U^o_\delta,\bV;ds^2,h)$
is finite dimensional. For simplicity we write $D^\bullet_{\rm min}(-):=D^\bullet_{\rm min}(-,\bV;ds^2,h)$ and $H^\ast_{(2),\rm min}(-):=H^\ast_{(2),\rm min}(-,\bV;ds^2,h)$.

Let $0<\delta''<\delta'<\delta\ll 1$. Denote $W:= (\overline{U_\delta}\cap X^o)\backslash \overline{U_{\delta''}}$.
By the same argument as (\ref{align_MV1}), one has the exact sequence in the sense of Mayer-Vietoris:
\begin{align*}
\xymatrix{
	0\ar[r]& D^\bullet_{\rm min}(W\cap U^o_{\delta'})\ar[r]& D^\bullet_{\rm min}(W)\oplus D^\bullet_{\rm min}(U^o_{\delta'})\ar[r]& D^\bullet_{\rm min}(\overline{U_\delta}\cap X^o)\ar[r]& 0.
}
\end{align*}
Hence we obtain the $L^2$-Mayer-Vietoris sequence
\begin{align}\label{align_mainproof_MV_PD1}
\xymatrix{
	\cdots\ar[r]& H^i_{(2),\min}(W\cap U^o_{\delta'})\ar[r]& H^i_{(2),\min}(U^o_{\delta'})\oplus H^i_{(2),\min}(W)\ar[r]& H^i_{(2),\min}(\overline{U_\delta}\cap X^o)\ar[r]&\cdots
}.
\end{align}
Notice that
\begin{align*}
H^i_{(2),\min}\left((\delta'',\delta],dt^2\right)\simeq H^i_{\rm dR}\left([\delta'',\delta],\{\delta''\}\right)=0,\quad\forall i
\end{align*}
and
\begin{align*}
H^i_{(2),\min}\left((\delta'',\delta'),dt^2\right)\simeq H^{1-i}_{(2),\max}\left((\delta'',\delta'),dt^2\right)^\ast\simeq H^{1-i}\left([\delta'',\delta']\right)^\ast\simeq
\begin{cases}
\bC, & i=1 \\
0, & i\neq 1
\end{cases}.
\end{align*}
Therefore Lemma \ref{lem_U_delta} and the $L^2$-Kunneth formula (Proposition \ref{prop_L2_Kunneth}) imply that
\begin{align*}
H^k_{(2),\min}(W)\simeq \bigoplus_{i+j=k}H^i_{(2),\min}\left((\delta'',\delta],dt^2\right)\otimes H^j_{(2),\min}(M_{\delta},\bV|_{M_{\delta}};ds^2|_{M_\delta},h|_{M_{\delta}})=0,\quad\forall k
\end{align*}
and
\begin{align*}
H^k_{(2),\min}(W\cap U^o_{\delta'})&\simeq \bigoplus_{i+j=k}H^i_{(2),\min}\left((\delta'',\delta'),dt^2\right)\otimes H^j_{(2),\min}(M_{\delta},\bV|_{M_{\delta}};ds^2|_{M_\delta},h|_{M_{\delta}})\\\nonumber
&\simeq H^{k-1}_{(2),\min}(M_{\delta},\bV|_{M_{\delta}};ds^2|_{M_\delta},h|_{M_{\delta}}),\quad\forall k.
\end{align*}
Here $M_\delta:=\{z\in X^o|\|z\|=\delta\}$.
Since $ds^2|_{M_\delta}$ is complete, for every $k$ we have
\begin{align}\label{align_main_PD_iso1}
H^{k}_{(2),\min}(M_{\delta},\bV|_{M_{\delta}};ds^2|_{M_\delta},h|_{M_{\delta}})&\simeq H^{k}_{(2),\max}(M_{\delta},\bV|_{M_{\delta}};ds^2|_{M_\delta},h|_{M_{\delta}})\\\nonumber
&\simeq \bH^k(\overline{M_{\delta}},\sD^\bullet_{\overline{M_{\delta}},\bV|_{M_{\delta}};ds^2|_{M_\delta},h|_{M_{\delta}}}).
\end{align}

{\bf Claim:} \textit{Assuming ${\bf IC}_{n-1}$, the cohomology sheaves of  $\sD^\bullet_{\overline{M_{\delta}},\bV|_{M_{\delta}};ds^2|_{M_\delta},h|_{M_{\delta}}}$ are constructible when $0<\delta\ll1$.}

Assume that the claim is true. Since $$\overline{M_\delta}=\{z\in X|\|z\|=\delta\}$$ is a compact real analytic space, by (\ref{align_main_PD_iso1}) $H^{k}_{(2),\min}(M_{\delta},\bV|_{M_{\delta}};ds^2|_{M_\delta},h|_{M_{\delta}})$ is finite dimensional for every $k$. By the assumption of ${\bf VD}_n$ and Proposition \ref{prop_cpt_vs_min},
$$H^{\dim X}_{(2),\rm min}(U^o_{\delta'})\simeq \bH^{\dim X}_{\rm cpt}(U_{\delta},\sD^\bullet_{X,\bV;ds^2,h})\simeq \left(\varinjlim_{x\in U}\left(\bH^{\dim X}_{\rm cpt}(U,\sD^\bullet_{X,\bV;ds^2,h})\right)^\ast\right)^\ast$$
is finite dimensional. By (\ref{align_mainproof_MV_PD1}) and the local completeness of $ds^2$ we obtain that
$$H^{\dim X}_{(2),\rm max}(U^o_\delta,\bV;ds^2,h)\simeq H^{\dim X}_{(2),\min}(\overline{U_\delta}\cap X^o)$$
is finite dimensional. The proof is finished.
\begin{proof}[Proof of the claim]
	Let us fix a Whitney stratified desingularization $(\{X_i\},\pi:\widetilde{X}\to X)$ of $X$ as in \S \ref{section_setting}. Since $0<\delta\ll1$ we assume that $0$ is the unique stratum of dimension $0$ and $\overline{M_\delta}$ intersects transversally with every stratum of $\{X_i\}$. For simplicity we denote $\sD^\bullet_{\overline{M_{\delta}}}:=\sD^\bullet_{\overline{M_{\delta}},\bV|_{M_{\delta}};ds^2|_{M_\delta},h|_{M_{\delta}}}$.
	Therefore Lemma \ref{lem_U_delta}
	and the $L^2$-Poincar\'e lemma (Lemma \ref{lem_Kunneth_locL2}) show that
	\begin{align}\label{align_main_VD_boundary_L2}
	\sD^\bullet_{\overline{M_{\delta}}}\simeq i^\ast\sD^\bullet_{X,\bV;ds^2,h},
	\end{align}
	where $i:\overline{M_\delta}\to X$ is the closed embedding. By Proposition \ref{prop_Dstr} and Proposition \ref{prop_Grauert_stratified}, the cohomology sheaves of $\sD^\bullet_{X,\bV;ds^2,h}$ are weakly constructible, so are the cohomology sheaves of $\sD^\bullet_{\overline{M_{\delta}}}$. It remains to show that for every $x\in \overline{M_\delta}$ and every $k$, $h^k(\sD^\bullet_{\overline{M_{\delta}}})_x$ is finite dimensional. This is done in two cases.
	\begin{description}
		\item[Case I: $x\in M_\delta$] In this case $\sD^\bullet_{\overline{M_{\delta}},x}\simeq \bV_{x}$ because of the usual Poincar\'e Lemma with coefficients in a local system (c.f. \cite[Proposition 2.19]{Deligne1970}).
		\item[Case II: $x\in S\cap (\overline{M_{\delta}}\backslash M_\delta)$ for some stratum $S$ of $\{X_i\}$] By (\ref{align_main_VD_boundary_L2}) and Proposition \ref{prop_Dstr} we obtain that
		\begin{align*}
		h^k(\sD^\bullet_{\overline{M_{\delta}}})_x\simeq
		h^k(\sD^\bullet_{X,\bV;ds^2,h})_x\simeq
		h^k(\sD^\bullet_{C_x,\bV|_{C_x};ds^2|_{C_x},h|_{C_x}})_x
		\end{align*}
		for some germ of complex subspace $x\in C_x\subset X$ which intersects transversally with every stratum of $\{X_i\}$. By Proposition \ref{prop_Grauert_stratified} and Proposition \ref{prop_harmonic_metric_stratified} we see that $ds^2|_{C_x}$ is distinguished associated to $(C_x,C_x\cap X^o)$ and $(\bV|_{C_x\cap X^o},h|_{C_x\cap X^o})$ is a $0$-tame harmonic bundle on $(C_x,C_x\cap X^o)$. Since $\dim C_x=\dim X-\dim S<n$, by ${\bf IC}_{n-1}$ there is a quasi-isomorphism
		$$\sD^\bullet_{C_x,\bV|_{C_x};ds^2|_{C_x},h|_{C_x}}\simeq IC_{C_x}(\bV|_{C_x\cap X^o}).$$
		As a consequence, $h^k(\sD^\bullet_{\overline{M_{\delta}}})_x$ is finite dimensional for every $k$.
	\end{description}
	This finishes the proof of the claim.
\end{proof}
\subsection{Proof of ${\bf VD}_{n}\Rightarrow {\bf IC}_{n}$}
The proof is divided into several steps.
		\begin{enumerate}
			\item {\bf Fineness:} The Fineness of $\sD^\bullet_{X,\bV;ds^2,h}$ follows from Lemma \ref{lem_fine_sheaf}.
			\item {\bf Localizing:} By Corollary \ref{cor_IC_is_local} we assume that $X$ is a germ of complex space and
			$$ds^2\sim ds^2_\pi:=\sqrt{-1}\ddbar\varphi_\pi+\omega_0$$
			for some Whitney stratified desingularization $\pi:\widetilde{X}\to X$ of $(X,X^o)$ and a certain hermitian metric $\omega_0$ on $X$.
			\item {\bf Verification of Proposition \ref{prop_GM_criterion}:}
			\begin{enumerate}
				\item $\sD^\bullet_{X,\bV;ds^2,h}|_{X^o}\simeq \bV$: This is the usual Poincar\'e Lemma with coefficients in a local system (c.f. \cite[Proposition 2.19]{Deligne1970}).
				\item {\bf Weakly constructible of $\sD^\bullet_{X,\bV;ds^2,h}$:} This follows from Proposition \ref{prop_Dstr} together with Proposition \ref{prop_Grauert_stratified} and Proposition \ref{prop_harmonic_metric_stratified}.
				\item {\bf Local Vanishing:} Using the notations in Theorem \ref{thm_local_vanishing}, one has
				\begin{align*}
				\bH^k_{\rm cpt}(U_{\delta},\sD^\bullet_{Z,\bV;ds^2_{\pi},h})=0,\quad \forall k\leq \dim Z
				\end{align*}
				and
				\begin{align*}
				\bH^k_{\rm cpt}(U_{\delta},\sD^\bullet_{Z,\bV^\ast;ds^2_{\pi},h^\ast})=0,\quad \forall k\leq \dim Z
				\end{align*}
				because the dual hermitian local system $(\bV^\ast,h^\ast)$ is also a $0$-tame harmonic bundle on $(Z,Z^o)$. By ${\bf VD}_{n}$ we obtain that
				\begin{align*}
				H^k_{(2),\rm max}(U^o_{\delta},\bV;ds^2_{\pi},h)=0,\quad \forall k\geq \dim Z.
				\end{align*}
				Taking $Z=C_x$ in Proposition \ref{prop_Dstr} for every point $x\in X$, we obtain the required local vanishings (\ref{align_GMD_vanishing}) and (\ref{align_GMD_covanishing}).
			\end{enumerate}
		\end{enumerate}

	\section{Applications: K\"ahler Package on the Intersection cohomology}
	The $L^2$-representation of Hodge modules provides a way to study Hodge modules by using differential geometrical methods as well as $L^2$-techniques. In this section we use Theorem \ref{thm_main_harmonic} to give a differential geometrical proof of the Hodge-Lefschetz package of a pure Hodge module, or more generally a $0$-tame harmonic bundle. This is one of the motivations why Cheeger-Goresky-MacPherson post Conjecture \ref{conj_CGM}.
	\subsection{Poincar\'e duality and Verdier duality}
	A typical property of $IC_X(\bV)$ is its Verdier duality, i.e.
	\begin{align}\label{align_dual_IC}
	IC_X(\bV)\simeq\bD(IC_X(\bV^\ast)).
	\end{align}
	This induces the global duality (Poincar\'e duality)
	\begin{align*}
	IH^{2\dim X-k}(X,\bV^\ast)\simeq \left(IH^k(X,\bV)\right)^\ast,\quad \forall k
	\end{align*}
	in the case when $X$ is compact.
	Here $\bD$ is the Verdier dual operator on $D(X)$. Theorem \ref{thm_main_harmonic} shows that these dualities could be realized by integration as in the classical case.

	Locally at every point $x\in X$, (\ref{align_dual_IC}) is equivalent to
	\begin{align*}
	\varinjlim_{x\in U}\bH^{2\dim X-k}(U,IC_X(\bV)))\simeq \varinjlim_{x\in U}\left(\bH^k_{\rm cpt}(U,IC_X(\bV^\ast))\right)^\ast,\quad \forall k
	\end{align*}
	where $U$ ranges over all open neighborhoods of $x$. The following theorems could be viewed as differential geometrical proofs of the local and global dualities of intersection complexes.
	\begin{thm}[Locally Duality]\label{thm_main_L2_local_duality}
		Notations as in Theorem \ref{thm_main_harmonic}. Let $x\in X$ be a point. Then there is a duality
		\begin{align*}
		\varinjlim_{x\in U}\bH^{2\dim X-k}(U,\sD^\bullet_{X,\bV;ds^2,h})\simeq \varinjlim_{x\in U}\left(\bH^k_{\rm cpt}(U,\sD^\bullet_{X,\bV^\ast;ds^2,h^\ast})\right)^\ast,\quad \forall k
		\end{align*}
		where $U$ ranges over all open neighborhoods of $x$. Both sides of the isomorphism are finite dimensional. Moreover, the duality is induced by the limits of the pairings
		\begin{align*}
		\bH^{2\dim X-k}(U,\sD^\bullet_{X,\bV;ds^2,h})\times \bH^k_{\rm cpt}(U,\sD^\bullet_{X,\bV^\ast;ds^2,h^\ast})\to\bC,\quad \forall k,
		\end{align*}
		$$([\alpha],[\beta])\mapsto\int_{U\cap X^o}\alpha\wedge\beta.$$
	\end{thm}
    \begin{proof}
    	Since the problem is local, we assume that $x=0\in X\subset \bC^N$. Let $U_\delta=\left\{z\in X\big|\|z\|<\delta\right\}$ and $U_\delta^o=X^o\cap U_\delta$. By Proposition \ref{prop_cpt_vs_min} the natural maps
    	\begin{align*}
    	\left(H^k_{(2),\rm min}(U^o_\delta,\bV^\ast;ds^2,h^\ast)\right)^\ast\to\left(\bH^k_{\rm cpt}(U_\delta,\sD^\bullet_{X,\bV^\ast;ds^2,h^\ast})\right)^\ast\to\varinjlim_{x\in U}\left(\bH^k_{\rm cpt}(U,\sD^\bullet_{X,\bV^\ast;ds^2,h^\ast})\right)^\ast
    	\end{align*}
    	are isomorphisms for every $k$ and every $0<\delta\ll 1$. Hence by Theorem \ref{thm_main_harmonic}, $H^k_{(2),\rm min}(U^o_\delta,\bV^\ast;ds^2,h^\ast)$ is finite dimensional for every $k$. Now
    	the theorem is a consequence of Proposition \ref{prop_L2_Poincare_duality}.
    \end{proof}
	\begin{thm}[Poincar\'e Duality]\label{thm_main_L2_duality}
		Notations as in Theorem \ref{thm_main_harmonic}. Assume that $X$ is a compact complex space. Then for every $k$ there is a perfect pairing
		\begin{align*}
		H^{2\dim X-k}_{(2)}(X^o,\bV;ds^2,h)\times H^k_{(2)}(X^o,\bV^\ast;ds^2,h^\ast)\to\bC,\quad ([\alpha],[\beta])\mapsto \int_{X^o}\alpha\wedge\beta.
		\end{align*}
	\end{thm}
	\begin{proof}
		By Theorem \ref{thm_main_harmonic}, both $L^2$-cohomologies $H^{\ast}_{(2)}(X^o,\bV;ds^2,h)$ and $H^{\ast}_{(2)}(X^o,\bV^\ast;ds^2,h^\ast)$ are of finite dimension. Now the theorem follows from Proposition \ref{prop_L2_Poincare_duality}.
	\end{proof}
	\subsection{Lefschetz package}
	By the Lefschetz package we mean the Hard Lefschetz and the Lefschetz decomposition. Due to the classical Hodge theory this package holds for the de Rham cohomology of every compact K\"ahler manifold. The Lefschetz package for semisimple local systems is established by Simpson \cite{Simpson1992}. More generally, Kashiwara \cite{Kashiwara1998} conjectures that the Lefschetz package should hold for any semisimple perverse sheaf over a projective algebraic variety\footnote{The conjecture of Kashiwara is even more general. It asserts that the Lefschetz package should be true in the relative case, for semisimple holonomic $\sD$-modules with possibly irregular singularities. For $\sD$-modules with regular singularities over a quasi-projective algebraic variety,  Kashiwara's conjecture is proved by C. Sabbah \cite{Sabbah2005} in lisse case and by T. Mochizuki \cite{Mochizuki20072,Mochizuki20071} and V. Drinfeld \cite{Drinfeld2001} in general.}. In this subsection we prove this conjecture over a compact K\"ahler base by using the $L^2$-representation of the intersection complex.

	Let $X$ be a complex space. Define $\sA^\bullet_X$ to be the complex of piecewise $C^\infty$ forms which can be extended smoothly over some neighborhood of a local embedding $X\subset\bC^N$. I.e. locally $\sA^\bullet_X:=\sA^\bullet_{\bC^N}|_X$. Since the restriction of sheaves is an exact functor, $\sA^\bullet_X$ is a complex of fine sheaves and the canonical inclusion
	$\iota:\bC_X\to\sA^\bullet_X$ is a quasi-isomorphism. We say a closed $(1,1)$-form $\omega\in\sA^2_X$ is positive if $\omega$ is the K\"ahler form of some hermitian metric on $X$, i.e. locally $\omega$ is the restriction of a K\"ahler form through a local embedding $X\subset \bC^N$. A cohomology class $\xi\in H^2(X,\bR)$ is a K\"ahler class if it is represented by some closed positive $(1,1)$-form. A typical example of positive forms is $\omega_{FS}|_X$ where $X\subset \mathbb{P}^N$ is a projective variety and $\omega_{FS}$ is the K\"ahler class associated to the Fubini-Study metric.

	Let $X^o\subset X_{\rm reg}$ be a dense Zariski open subset and $(\bV,h)$ a $0$-tame harmonic bundle on $(X,X^o)$.
	Let $ds^2$ be a distinguished metric associated to $(X,X^o)$. Since every form in $\sA^\bullet_X$ has bounded length with respect to $ds^2$ (Lemma \ref{lem_fine_sheaf}), there is a well defined wedge product
	\begin{align*}
	\sA^\bullet_X\otimes \sD^\bullet_{X,\bV;ds^2,h}\to \sD^\bullet_{X,\bV;ds^2,h}.
	\end{align*}
	This makes $\sD^\bullet_{X,\bV;ds^2,h}$ a differential graded $\sA^\bullet_X$-module.

	Let $\tau:IC_X(\bV)\to \sD^\bullet_{X,\bV;ds^2,h}$ be the quasi-isomorphism in Theorem \ref{thm_main_harmonic}. Since $\tau$ is $\bC_X$-linear, the diagram
	\begin{align*}
	\xymatrix{
		\bC_X\otimes IC_X(\bV)\ar[r]\ar[d]^{\iota\otimes\tau} & IC_X(\bV)\ar[d]^\tau\\
		\sA^\bullet_X\otimes \sD^\bullet_{X,\bV;ds^2,h}\ar[r]& \sD^\bullet_{X,\bV;ds^2,h}
	}
	\end{align*}
	is commutative in $D(X)$. As a consequence, the diagram of wedge products
	\begin{align*}
	\xymatrix{
		H^i(X,\bC_X)\otimes IH^j(X,\bV)\ar[r]\ar[d] & IH^{i+j}(X,\bV)\ar[d]\\
		H^i(\Gamma(\sA^\bullet_X))\otimes H_{(2)}^j(X^o,\bV;ds^2,h)\ar[r]& H_{(2)}^{i+j}(X^o,\bV;ds^2,h)
	}
	\end{align*}
	is commutative for every $i$, $j$ when $X$ is compact.
	\begin{thm}\label{thm_Lefschetz_package}
		Let $X$ be a compact K\"ahler space of dimension $n$ and $X^o\subset X_{\rm reg}$ a dense Zariski open subset. Let $\bV$ be a local system admitting a $0$-tame harmonic metric. Let $\omega\in H^2(X,\bR)$ be a K\"ahler class and denote $L_{\omega}(-):=\omega\wedge-$. Then the following statements hold.
		\begin{description}
			\item[Hard Lefschetz] The morphism
			\begin{align*}
			L_{\omega}^k:IH^{n-k}(X,\bV)\to IH^{n+k}(X,\bV)
			\end{align*}
			is an isomorphism for every $k>0$.
			\item[Lefschetz Decomposition] Define the primitive cohomology
			\begin{align*}
			IP^{n-k}(X,\bV):={\rm Ker}\left(L_{\omega}^{k+1}:IH^{n-k}(X,\bV)\to IH^{n+k+2}(X,\bV)\right),\quad \forall k\geq 0.
			\end{align*}
			Then for every $k\geq n$ there is an orthogonal decomposition
			\begin{align*}
			IH^k(X,\bV)=\bigoplus_{i}L_\omega^i IP^{k-2i}(X,\bV).
			\end{align*}
		\end{description}
	\end{thm}
	\begin{proof}
		Let $\omega_\pi=\sqrt{-1}\ddbar\varphi_\pi+\omega$ be the K\"ahler form of the distinguished metric $ds^2_\pi$ associated to some Whitney stratified desingularization $\pi:\widetilde{X}\to X$ of $(X,X^o)$ (Lemma \ref{lem_Grauert_metric_exist}). By Theorem \ref{thm_main_harmonic} there is a canonical isomorphism
		\begin{align*}
		IH^\ast(X,\bV)\simeq H^\ast_{(2)}(X^o,\bV;ds^2_\pi,h).
		\end{align*}
		Denote
		\begin{align}\label{align_L_omegapi}
		L_{\omega_\pi}:H^\ast_{(2)}(X^o,\bV;ds^2_\pi,h)\to H^{\ast+2}_{(2)}(X^o,\bV;ds^2_\pi,h),\quad [\alpha]\mapsto[\omega_\pi\wedge\alpha].
		\end{align}
		Notice that $d\omega_\pi=0$ and $|\omega_\pi|_{\omega_\pi}=1$ is bounded, so (\ref{align_L_omegapi}) is well defined. Since $IH^\ast(X,\bV)$ is finite dimensional, so is $H^\ast_{(2)}(X^o,\bV;ds^2_\pi,h)$. Hence
		\begin{align*}
		H^k_{(2)}(X^o,\bV;ds^2_\pi,h)\simeq \sH^k:=\{\alpha\in D^k_{\rm max}(X^o,\bV;ds^2_\pi,h)|\nabla\alpha=0,\nabla^\ast_h\alpha=0\},\quad\forall k
		\end{align*}
		where the isomorphism is compatible with the operator $\omega_\pi\wedge-$. Hence $L_{\omega_\pi}$ and its dual $\Lambda_{\omega_\pi}$ send harmonic forms to harmonic forms. As in the classical situation, they form a part of an ${\bf sl}(2)$-representation on $\sH^\ast$.  Now the standard argument shows the Hard Lefschetz and Lefschetz Decomposition for $L_{\omega_\pi}$.

		Since the Lefschetz operator $L_\omega$ is represented by $\omega\wedge-$, the theorem follows if we could show that the diagram
		\begin{align*}
		\xymatrix{
			H^\ast_{(2)}(X^o,\bV;ds^2_\pi,h)\ar[r]^{L_{\omega_\pi}}\ar[d]^\simeq & H^{\ast+2}_{(2)}(X^o,\bV;ds^2_\pi,h) \ar[d]^\simeq\\
			IH^\ast(X,\bV) \ar[r]^{L_\omega} & IH^{\ast+2}(X,\bV)
		}
		\end{align*}
		is commutative. Equivalently $[\omega\wedge\alpha]=[\omega_\pi\wedge\alpha]$ for every $\alpha\in {\rm Ker}\nabla$. This is because
		\begin{align*}
		\omega_\pi\wedge\alpha-\omega\wedge\alpha=\sqrt{-1}\ddbar\varphi_\pi\wedge\alpha=\nabla(\sqrt{-1}\dbar\varphi_\pi\wedge\alpha).
		\end{align*}
		By (\ref{align_grad_norm_Gt}), $|\dbar\varphi_\pi|_{\omega_\pi}$ is bounded , so $\dbar\varphi_\pi\wedge\alpha\in D^{\ast+1}_{\rm max}(X^o,\bV;ds^2_\pi,h)$. This finishes the proof.
	\end{proof}
	Note that every semisimple perverse sheaf on a complex space is a direct sum of intersection complexes with coefficients in a semisimple local system. Using the $0$-tame harmonic metrics on semisimple local systems (\cite{Mochizuki20072,Mochizuki20071,Jost_Zuo1997}, Remark \ref{rmk_harmonic_bundle_regular_singularity_examples}), we are able to show
	\begin{cor}[Kashiwara's conjecture over K\"ahler space: Absolute case]\label{cor_Kashiwara_conj}
		Let $X$ be a compact K\"ahler space and $P$ a semisimple perverse sheaf on $X$. Let $\omega\in H^2(X,\bC)$ be a K\"ahler class and $L_\omega$ the associated Lefschetz operator. Then
		\begin{align*}
		L^k_\omega:\bH^{-k}(X,P)\to \bH^{k}(X,P),\quad \forall k>0
		\end{align*}
		is an isomorphism. As a consequence, $L_\omega$ induces a Lefschetz decomposition on $\bH^\ast(X,P)$.
	\end{cor}
    \begin{proof}
    	Since $P$ is semisimple, it is a direct sum of $IC_{Z_i}(\bV_i)[\dim Z_i]$, $i\in I$ where $Z_i$ is an irreducible Zariski closed analytic subspace of $X$ and $\bV_i$ is a semisimple local system over some Zariski open subset $Z^o_i\subset Z_{i,\rm reg}$. Now the corollary follows from Theorem \ref{thm_Lefschetz_package}.
    \end{proof}
	\subsection{Pure Hodge structure}\label{section_Hodge_structures}
	Let $X$ be a compact K\"ahler space (or more generally let $X$ be in the Fujiki class $(\sC)$) and $(\cV,\nabla,\mathcal{F}^\bullet,h)$ an $\bR$-polarized variation of Hodge structure on a dense Zariski open subset $X^o\subset X_{\rm reg}$ with quasi-unipotent local monodromies. By Lemma \ref{lem_Grauert_metric_exist} there is a K\"ahler distinguished metric $ds^2$. Let $\bV:={\rm Ker}\nabla$ and
	let
	$$\mathcal{F}^p_{(2)}:=\left\{\alpha\in\sD^\bullet_{X,\bV;ds^2,h}\bigg|\alpha=\sum_{P\geq p}\alpha^{P,Q}\in \sL^{P,Q}(\bV)\right\},$$
	where $\sL^{P,Q}(\bV)$ stands for the sheaf of differential forms with total degree $(P,Q)$. Since $ds^2$ is complete (Corollary \ref{cor_Gt_complete_finite_vol}), the filtered complex $(\sD^\bullet_{X,\bV;ds^2,h},\mathcal{F}^\bullet_{(2)})$ and the Lefschetz operator $L$ associated to $ds^2$ provide a polarized pure Hodge structure on $H^\ast_{(2)}(X^o,\bV;ds^2,h)$ (\cite[\S 6]{Kashiwara_Kawai1987} or \cite[\S 7]{Zucker1979}). Moreover the Lefschetz package (Theorem \ref{thm_Lefschetz_package}) is compatible with this Hodge structure.
	\begin{conj}\label{conj_Hodge_structure}
		Notations as in Theorem \ref{thm_main_harmonic}.
		Assume that $X$ is a compact K\"ahler space and $ds^2$ is a K\"ahler distinguished metric.
		Then the $L^2$-pure Hodge structure on $H^\ast_{(2)}(X^o,\bV;ds^2,h)$ is compactible with Saito's pure Hodge structure on $IH^\ast(X,\bV)$ via the canonical isomorphism in Theorem \ref{thm_main_harmonic}.
	\end{conj}
	This conjecture is related to the Harris-Zucker conjecture \cite[Conjecture 5.3]{Harris-Zucker2001}, which claims that the Hodge structure of Saito is compatible with the $L^2$-Hodge structure via the isomorphism (\cite{Saper_Stern1990,Looijenga1988})
	$$H^\ast_{(2)}(X^o,\bV;ds^2,h)\simeq IH^\ast(X,\bV)$$
	for a connected component $X^o$ of a Shimura variety and its Satake-Baily-Borel compactification $X$. Here $\bV$ is the local system associated to a rational representation of the group in the Shimura datum.

	There are several evidences for Conjecture \ref{conj_Hodge_structure}.
	\begin{enumerate}
		\item When $X\backslash X^o$ is a subset of isolated points and $\bV=\bC$ is the trivial Hodge module, Conjecture \ref{conj_Hodge_structure} holds by applying Zucker's argument in \cite{Zucker1987}.
		\item In \cite[Theorem 1.1]{SZ2021} the authors show that the top Hodge piece $\mathcal{F}^{\rm top}_{(2)}$ is quasi-isomorphic to Saito's $S$-sheaf $S(IC_X(\bV))$, the lowest (re-indexed) Hodge piece of the Hodge module $IC_X(\bV)$. Hence we obtain
		\begin{thm}\label{thm_resolve_S_sheaf}
			Let $X$ be a compact K\"ahler space of dimension $n$ and $X^o\subset X_{\rm reg}$ a dense Zariski open subset. Let $\bV=(\cV,\nabla,\cF^\bullet,h)$ be an $\bR$-polarized variation of Hodge structure of weight $r$ on $X^o$ with quasi-unipotent local monodromies. Let $H^{p,q}_{(2)}(X^o,\bV)$ be the $(p,q)$-summand in the $L^2$-Hodge structure of $H^\ast_{(2)}(X^o,\bV;ds^2,h)$ and let $IH^{p,q}(X,\bV)$ be the $(p,q)$-summand of Saito's Hodge structure on $IH^\ast(X,\bV)$. Then for every $q\geq0$ there is a natural isomorphism
			$$IH^{n+r,q}(X,\bV)\simeq H^{n+r,q}_{(2)}(X^o,\bV).$$
		\end{thm}
	\end{enumerate}

	\bibliographystyle{plain}
	\bibliography{CGM}

\end{document}